\documentclass[11pt, a4paper]{article}
\usepackage[utf8]{inputenc}
\usepackage[usenames,dvipsnames]{xcolor}
\usepackage{amsmath}
\usepackage{amsthm}
\usepackage{amsfonts}
\usepackage{amssymb}
\usepackage{array}
\usepackage{graphicx} 
\usepackage{color}
\usepackage[english]{babel}
\usepackage{mathrsfs}
\usepackage{graphicx}
\usepackage{dsfont}
\definecolor{orangebis}{rgb}{0.99,0.25,0.00}
\definecolor{greenbis}{rgb}{0.10,0.85,0.10}
\definecolor{bluebis}{rgb}{0.10,0.30,0.99}
\usepackage[final]{hyperref}   
\hypersetup{
    linktoc=page,
    linkcolor=red,          
    citecolor=blue,        
    filecolor=blue,      
    urlcolor=cyan,
   colorlinks=true      }     

\newcommand{\red}[1]{\textcolor{black}{#1}}
\newcommand{\blue}[1]{\textcolor{black}{#1}}

\usepackage{multicol}

\usepackage{lipsum}
\usepackage{geometry}
\author{Hugo Vanneuville\thanks{Univ. Lyon 1, Institut Camille Jordan, 69100 Villeurbanne, France, supported by the ERC grant Liko No 676999}}
\title{The annealed spectral sample of Voronoi percolation}
\date{}

\theoremstyle{plain}
\newtheorem{thm}{Theorem}[section]
\newtheorem{prop}[thm]{Proposition}
\newtheorem{lem}[thm]{Lemma}

\newtheorem{cor}[thm]{Corollary}

\newtheorem{lemma}[thm]{Lemma}
\newtheorem{claim}[thm]{Claim}

\theoremstyle{definition}
\newtheorem{defi}[thm]{Definition}

\theoremstyle{remark}
\newtheorem{rem}[thm]{Remark}

\theoremstyle{remark}

\newcommand{\margin}[1]{\textcolor{magenta}{*}\marginpar{ \vskip -1cm \textcolor{magenta} {\it #1 }  }}

\renewcommand{\margin}[1]{}

\newcommand{\N}{\mathbb{N}}
\newcommand{\R}{\mathbb{R}}

\newcommand{\Z}{\mathbb{Z}}
\newcommand{\Q}{\mathbb{Q}}

\newcommand{\diam}{\text{\textup{diam}}}
\newcommand{\Pro}{\mathbb{P}}
\newcommand{\E}{\mathbb{E}}
\newcommand{\T}{\mathbb{T}}

\newcommand{\Var}{\text{\textup{Var}}}
\newcommand{\Cov}{\text{\textup{Cov}}}
\newcommand{\Prob}{\text{\textup{\textbf{P}}}}
\newcommand{\Ex}{\text{\textup{\textbf{E}}}}
\newcommand{\Piv}{\text{\textup{\textbf{Piv}}}}
\newcommand{\arm}{\text{\textup{\textbf{A}}}}
\newcommand{\cross}{\textup{\text{Cross}}}
\newcommand{\dense}{\textup{\text{Dense}}}

\newcommand{\qbc}{\textup{\text{QBC}}}
\newcommand{\gi}{\textup{\text{GI}}}
\newcommand{\gp}{\textup{\text{GP}}}

\newcommand{\un}{\mathds{1}}
\newcommand{\petito}[1]{o\mathopen{}\left(#1\right)}
\newcommand{\grandO}[1]{O\mathopen{}\left(#1\right)}

\def\T{\mathbb{T}}

\newcommand{\cond}{\, \Big| \,}
\renewcommand{\textbf}[1]{\begingroup\bfseries\mathversion{bold}#1\endgroup}
\def\diam{\mathrm{diam}}

\def\E{\mathbb{E}} 

\def \eps {\epsilon}

\def\<#1{\langle #1\rangle}

\def\bi{\begin{itemize}}  
\def\ei{\end{itemize}}
\def\bnum{\begin{enumerate}} 
\def\enum{\end{enumerate}}

\numberwithin{equation}{section}
\begin{document}

\maketitle

\abstract{In this paper, we introduce and study the annealed spectral sample of Voronoi percolation, which is a continuous and finite point process in $\R^2$ whose definition is mostly inspired by the spectral sample of Bernoulli percolation introduced in \cite{garban2010fourier} by Garban, Pete and Schramm. We show a clustering effect as well as estimates on the full lower tail of this spectral object. 

Our main motivation is the study of two models of dynamical critical Voronoi percolation in the plane. In the first model, the Voronoi tiling does not evolve in time while the colors of the cells are resampled at rate $1$. In the second model, the centers of the cells move according to (independent) long range stable Lévy processes but the colors do not evolve in time. We prove that for these two dynamical processes there exist almost surely exceptional times with an unbounded monochromatic component.


}



\tableofcontents

\section{Introduction}\label{s.intro}

\subsection{Models and main results}\label{ss.main}

In this paper, we study two models of planar \textbf{dynamical Voronoi percolation} at criticality and we prove results of existence of exceptional times with unbounded clusters. Let us first define the model of (static) planar Voronoi percolation. To this purpose, we first need a Poisson point process $\eta$ of intensity $1$ in the plane. The \textbf{Voronoi cells} of the points $x \in \eta$ are the sets $C(x)=\{u \in \R^2 \, : \, \forall x' \in \eta , \, ||x-u||_2 \leq ||x'-u||_2 \}$. It is not difficult to show that a.s. all the cells are bounded convex polygons. Let $p \in [0,1]$. Conditionally on $\eta$, we construct a coloring of the plane in black and white as follows: each Voronoi cell is colored in black with probability $p$ and in white with probability $1-p$, independently of the other cells.\footnote{\blue{Note that in this definition the points which are at the boundary of both a black cell and a white cell are colored in black \textbf{and} white. This does not change anything in the proofs and results for static Voronoi percolation but this will help us in the study of some dynamical models (we will actually use this only in Appendix \ref{a.2nd} and this will be of no importance for a wide family of dynamical processes, see this appendix for more details).}} We write $\omega \in \{-1,1\}^\eta$ for the corresponding colored point process, where $1$ means black and $-1$ means white. The distribution of $\omega$ will be denoted by $\Pro_p$ (when we study events that depend only on $\eta$, we will sometimes omit the subscript $p$). Note that, if we condition on $\eta$, then the distribution of $\omega$ is $\Prob_p^\eta := (p\delta_1+(1-p)\delta_{-1})^{\otimes \eta}$.
\medskip

Let us be a little more precise about measurability issues: we let $\Omega'$ denote the set of locally finite subsets of $\R^2$ and we equip $\Omega'$ with the $\sigma$-algebra generated by the functions $\overline{\eta} \in \Omega' \mapsto |\overline{\eta} \cap A|$ where $A$ spans the Borel subsets of the plane. Also, we let $\Omega=\cup_{\overline{\eta} \in \Omega'} \{-1,1\}^{\overline{\eta}}$ and we equip $\Omega$ with the $\sigma$-algebra generated by the functions $\overline{\omega} \in \Omega \mapsto |\overline{\omega}^{-1}(1) \cap A|$ and $\overline{\omega} \in \Omega \mapsto |\overline{\omega}^{-1}(-1) \cap A|$ where $A$ still spans the Borel subsets of the plane. The measure $\Pro_p$ is defined on this $\sigma$-algebra.
\medskip

The \textbf{critical parameter} of planar Voronoi percolation is defined as follows:
\[
p_c = \inf \{ p \in [0,1] \, : \, \Pro_p \left[ 0 \leftrightarrow \infty \right]  > 0 \} \, ,
\]
where $\{ 0 \leftrightarrow \infty \}$ is the event that there is a black path from $0$ to infinity. It has been shown by Bollob\'{a}s and Riordan~\cite{bollobas2006critical} that $p_c=1/2$ (see~\cite{duminil2017exponential,ahlberg2017noise} for more recent proofs). More precisely, if $p \leq 1/2$ then a.s.\ there is no unbounded black component while if $p>1/2$ then a.s. there exists a unique unbounded black component.\footnote{The proof of the result for $p\leq 1/2$ goes back to the work by Zvavitch~\cite{zvavitch1996critical}.}
\medskip

In this paper, we are only interested in critical Voronoi percolation, so \textbf{we fix $p=p_c=1/2$ once and for all}. We are interested in two models of dynamical Voronoi percolation. These two models are Markov processes with $\Pro_{1/2}$ as an invariant measure. In all the paper, we sample them initially according to this measure.
\medskip

The first dynamical model is defined analogously to the model of dynamical percolation from~\cite{olle1997dynamical} (in this paper, one considers a bond percolation configuration on some graph and resamples each bond at rate $1$, independently of the other bonds). We sample a Voronoi percolation model of parameter $1/2$ (i.e. we sample the variables $\eta$ and $\omega \in \{-1,1\}^\eta$ as above) and, conditionally on $\eta$, we resample the color of each cell at rate $1$, independently of the other cells. In particular, \textbf{$\eta$ does not change in time}. We obtain a dynamical process $(\omega^{froz}(t))_{t \in \R_+}$ that we call \textbf{frozen dynamical Voronoi percolation}. This is a Markov càdlàg process with values in $\Omega$ (for the metric on $\Omega$ defined in Appendix~\ref{s.simple}), so $(\omega^{froz}(t))_{t \in \R_+}$ can be seen as a random variable with values in the Skorokhod space on $\Omega$.
\medskip

In the second dynamical model, the points of $\eta$ move around but their colors do not evolve in time. Let $\mu$ be the law of a planar Lévy process\footnote{One could probably consider \red{non-Lévy processes}. But we were not interested in this question and we have studied Lévy processes to gain simplicity.} starting from $0$, sample a Voronoi percolation model of parameter $1/2$ as above, and let each point $x \in \eta$ evolve independently of the other points according to a process of law $\mu$. We obtain a dynamical process $(\omega^\mu(t))_{t \in \R_+}$ that we call \textbf{$\mu$-dynamical Voronoi percolation}. This is also a Markov càdlàg process. We denote the underlying non-colored point configuration by $(\eta^\mu(t))_{t \in \R_+}$. See Appendix~\ref{s.simple} for a more precise construction of $(\omega^\mu(t))_{t \in \R_+}$.
\medskip

Let us note that an analogous model has been studied by van den Berg, Meester and White in~\cite{van1997dynamic} for the Poisson-Boolean model, which is another continuous percolation model. However, they study the model in any dimension and away from criticality, which is very different from the present paper.
\medskip

We define the notion of \textbf{exceptional times} exactly as in the case of dynamical Bernoulli percolation (\cite{olle1997dynamical}): a time $t \in \R_+$ is exceptional if \textbf{there exists an unbounded black component} at time $t$. Note that if we fix some $t$ then a.s. it is not exceptional. The question we are interested in is whether or not there exist (random) exceptional times.\footnote{Note that the event of existence of exceptional times is measurable with respect to (the completion of) the Skorokhod $\sigma$-algebra. Moreover, Kolmogorov $0$-$1$ law implies that either a.s. there is no exceptional time or a.s. for every open non empty interval $J \subseteq \R_+$ there are infinitely many exceptional times in $J$. We refer to~\cite{olle1997dynamical} for similar observations.} Our main results are the following:
\begin{thm}\label{t.main_frozen}
Consider frozen dynamical Voronoi percolation. A.s. there exist exceptional times at which there is an unbounded black component. 
\end{thm}

\begin{thm}\label{t.main_levy}

There exists $\alpha_0 \in ]0,+\infty[$ such that the following holds. Let $\mu$ be the law of a planar Lévy process such that there exists $\alpha \in ]0,\alpha_0]$ satisfying
\begin{equation}\label{e.alpha}
\exists c > 0,\, \forall L\in [1,+\infty[, \, \forall t \in [0,1], \, \Pro \left[ ||X_t||_2 \geq L \right] \geq c\frac{t}{L^\alpha} \, ,
\end{equation}
where $X \sim \mu$. Also, let $(\omega^\mu(t))_{t \in  \R_+}$ be a $\mu$-dynamical Voronoi percolation process. Then, a.s. there exist exceptional times $t$ at which there is an unbounded black component in $\omega^\mu(t)$.
\end{thm}

The \textbf{$\alpha$-stable Lévy processes} satisfy~\eqref{e.alpha}. Thus, our result applies to $\alpha$-stable processes with sufficiently large range but it does not apply if $\alpha$ is too large and in particular it does not apply to the Brownian motion. However, we believe that this result holds for any Lévy process that is not the zero process.
\medskip

We can also consider a third dynamical model by ``mixing'' the two dynamics above. We have the following result whose proof is essentially the same as Theorem~\ref{t.main_frozen}.
\begin{thm}\label{t.Brownian}
Let $\mu$ be the distribution of a planar Lévy process, consider a Voronoi percolation configuration of parameter $1/2$, and define a dynamical process \red{by letting each point move independently of the other points according to a Lévy process of law $\mu$ and by resampling the color of each point at rate $1$ independently of the other points and of the Lévy processes}. Then, there exist exceptional times with a black unbounded component.
\end{thm}
(In particular, the above holds when the Lévy process is a Brownian motion.)
\medskip

Theorem~\ref{t.main_frozen} is analogous to the results of existence of exceptional times for dynamical Bernoulli percolation on the triangular lattice $\mathcal{T}$ and the square lattice $\Z^2$ that have been proved by Schramm and Steif~\cite{schramm2010quantitative} and Garban, Pete and Schramm~\cite{garban2010fourier}, and Theorem~\ref{t.main_levy} is analogous to the result of existence of exceptional times for Bernoulli percolation evolving according to an exclusion process that has been proved in~\cite{GV}.
In the present paper, we are highly inspired by the methods from~\cite{garban2010fourier} and~\cite{GV} (in this last paper, the methods from~\cite{garban2010fourier} were also central). In \cite{garban2010fourier}, Garban, Pete and Schramm define and study the \textbf{spectral sample} of Bernoulli percolation. In the present paper, we introduce and study an analogue of the spectral sample which is a random variable with values in the finite subsets of $\R^2$ that we call the \textbf{annealed spectral sample of Voronoi percolation}. See Section~\ref{s.ann} for the definition of this object.

\subsection{Previous results on planar Voronoi percolation}

As mentionned above, Zvavitch~\cite{zvavitch1996critical} has proved that $p_c \geq 1/2$ and Bollob\'{a}s and Riordan~\cite{bollobas2006critical} have proved that this an equality. Planar Voronoi percolation has then been studied in several works (we refer to Sections \ref{ss.NS_for_voro} and \ref{ss.arm} for precise statements): i) in \cite{tassion2014crossing}, Tassion proves the annealed box-crossing property at parameter $p=1/2$, ii) in \cite{ahlberg2015quenched}, Ahlberg, Griffiths, Morris and Tassion prove a quenched version of this result and prove a quenched noise sensitivity result, iii) in \cite{ahlberg2017noise}, Ahlberg and Baldasso show noise sensitivity results where the noise consists in relocalizing a small portion of points and prove that the near-critical window of Voronoi percolation is of polynomial size, iv) in \cite{V1}, we prove the quasi-multiplicativity property for arm events and we deduce annealed scaling relations, v) in \cite{V2} we prove quantitative quenched results for Voronoi percolation and we deduce a strict inequality on the exponent of the percolation function.

We see the present paper as the third of a series of three papers whose first two are \cite{V1} and \cite{V2} and whose final goal is to study the annealed spectral sample defined in Section \ref{s.ann}. Indeed, the quasi-multiplicativity property from \cite{V1} and the quantitative quenched properties from \cite{V2} are central tools in the present paper. Let us also note that the proofs from \cite{V1,V2} highly rely on the annealed and quenched box-crossing properties from \cite{tassion2014crossing,ahlberg2015quenched}.

Let us finally point out the work \cite{benjamini1998conformal}, where Benjamini and Schramm prove an important intermediate result towards conformal invariance of planar (and also 3D) Voronoi percolation, which remains one of the main conjectures for this model.

\subsection{Noise sensitivity for Voronoi percolation}\label{ss.NS_for_voro}

The question of existence of exceptional times is intimately related to the notion of \textbf{noise sensitivity}. This notion has been \red{introduced} by Benjamini, Kalai and Schramm in~\cite{benjamini1999noise} and has then been extensively studied (see~\cite{benjamini1999noise,schramm2010quantitative,garban2010fourier} for noise sensitivity results for crossing events of Bernoulli percolation, see also~\cite{garban2014noise} for a book on the subject). Let us recall the definition of noise sensitivity: For each $n \in \N$, equip $\{-1,1\}^n$ with the uniform \red{probability} measure $\Prob_{1/2}^n$, let $\omega_n(0) \sim \Prob_{1/2}^n$, and define the dynamical process $(\omega_n(t))_{t \in \R_+}$ by resampling each coordinate at rate \red{$1$}. Given a sequence of positive integers $(m_n)_{n \in \N}$ that goes to $+\infty$, we say that a sequence of functions $h_n \, : \, \{-1,1\}^{m_n} \rightarrow \{0,1\}$ is noise sensitive if, for every $t \in ]0,+\infty[$,
\[
\Cov \left( h_n(\omega_n(0)),h_n(\omega_n(t)) \right) \underset{n \rightarrow +\infty}{\longrightarrow} 0 \, .
\]
\blue{In other words, $h_n(\omega_n(0))$ and $h_n(\omega_n(t))$ are asymptotically independent.}

\subsubsection{Previous results on noise sensitivity for Voronoi percolation}

Noise sensitivity has already been studied for continuum percolation models:
\begin{itemize}
\item In~\cite{ahlberg2014noise}, the authors prove that the Poisson-Boolean model is noise sensitive.
\item In~\cite{ahlberg2015quenched}, the authors study quenched Voronoi percolation and answer a conjecture from~\cite{benjamini1999noise} related to the notion of noise sensitivity: they prove that, asymptotically almost surely, the quenched probabilities of crossing events do not depend on $\eta$. They also prove quenched and annealed noise sensitivity results for frozen dynamical Voronoi percolation, see Theorem~\ref{t.frozenNS} below.
\item In~\cite{ahlberg2017noise}, the authors prove noise sensitivity results for dynamical models obtained by relocalizing the position of points, see Theorem~\ref{t.AB} below.
\end{itemize}

Below, we let $g_n \, : \, \Omega \rightarrow \{0,1\}$ be the event that there is a black path from left to right in the square $[0,n]^2$. Here and in all the paper, by black path we mean continuous path included in the black region of the plane (a black path is a continuous object).

Let us state the noise sensitivity results from \cite{ahlberg2015quenched} and \cite{ahlberg2017noise}.

\begin{thm}[\cite{ahlberg2015quenched}]\label{t.frozenNS}
The crossing events are noise sensitive for frozen dynamical Voronoi percolation i.e.
\[
\forall t \in ]0,+\infty[, \, \Cov \left( g_n(\omega^{froz}(0)),g_n(\omega^{froz}(t)) \right) \underset{n \rightarrow +\infty}{\longrightarrow} 0 \, .
\]
The crossing events are also a.s. \textbf{quenched noise sensitive} in the sense that
\[
\text{a.s., } \forall t \in ]0,+\infty[, \, \text{\textup{\textbf{Cov}}}^\eta \left( g_n(\omega^{froz}(0)),g_n(\omega^{froz}(t)) \right)  \underset{n \rightarrow +\infty}{\longrightarrow} 0 \, ,
\]
where $\text{\textup{\textbf{Cov}}}^\eta$ is the covariance conditioned on $\eta$. Moreover, there exists a constant $a>0$ such that these annealed and quenched noise sensitivity results also hold if $t=t_n=n^{-a}$.
\end{thm}

\begin{thm}[\cite{ahlberg2017noise}]\label{t.AB}
Consider the Voronoi percolation model in the bounded box $[0,n]^2$ defined like in the present paper except that $\eta$ is a set of $n^2$ independent points uniformly distributed in $[0,n]^2$ rather than a Poisson process in the plane. We still write $\omega$ for the model at parameter $1/2$ and $g_n$ for the crossing event of $[0,n]^2$. Define the two following $\varepsilon$-noises on the Voronoi model: i) resample the localization of each point of $\eta$ with probability $\varepsilon$ but do not resample the colors; ii) resample (at the same time) the localization and the color of each point of $\eta$ with probability $\varepsilon$. The crossing events are sensitive for these two noises, i.e.
\[
\forall \varepsilon \in ]0,1], \, \Cov \left( g_n(\omega),g_n(\omega^\varepsilon) \right)  \underset{n \rightarrow +\infty}{\longrightarrow} 0 \, ,
\]
where $\omega^\varepsilon$ is the $\varepsilon$-perturbation of $\omega$.
\end{thm}


\subsubsection{A sharp noise sensitivity result}

In the present paper, we prove a quantitative noise \red{sensitivity} result. In order to state this theorem, we need to introduce the notion of \textbf{arm events}.
\begin{defi}
Let $0 < r,R <+\infty$ such that $r \leq R$ and $j \in \N^*$. The $j$-arm event $\arm_j(r,R)$ is the event that there are $j$ paths of alternating color in the annulus $[-R,R]^2 \setminus ]-r,r[^2$ from $\partial [-r,r]^2$ to $\partial [-R,R]^2$. We write $\alpha_j^{an}(r,R)$ for the annealed probability of this event i.e.:
\[
\alpha_j^{an}(r,R)=\Pro_{1/2} \left[ \arm_j(r,R) \right] \, .
\]
Also, we write $\alpha_j^{an}(R)=\alpha_j^{an}(1,R)$. If $r > R$, we let $\alpha_j^{an}(r,R)=1$.
\end{defi}
\begin{thm}\label{t.quant_NS_frozen}
Consider frozen dynamical Voronoi percolation and let $g_n$ be the crossing event of $[0,n]^2$. The covariance
\[
\Cov \left( g_n(\omega^{froz}(0)),g_n(\omega^{froz}(t_n)) \right)
\]
goes to $0$ as $n$ goes to $+\infty$ if $t_n n^2 \alpha_4^{an}(n)$ goes to $+\infty$ while this quantity is greater than a positive constant if $t_n n^2 \alpha_4^{an}(n)$ goes to $0$.\footnote{The quantity $n^2 \alpha_4^{an}(n)$ goes to $+\infty$ as $n$ goes to $+\infty$ polynomially fast in $n$ (see Proposition~\ref{p.4}).}
\end{thm}
\blue{As shown in \cite{V1} (see for instance Section 6 therein), $n^2\alpha^{an}_4(n)$ is of the same order as the expected number of pivotal points.\footnote{\blue{For instance, $n^2\alpha^{an}_4(n)$ is of the same order as $\E_{1/2} \left[ \sum_{x \in \eta} \un_{\Piv_x^q(g_n)} \right]$; see Definition \ref{d.piv} for the definition of $\Piv_x^q(g_n)$.}} The interpretation of the above result is the same as in the case of Bernoulli percolation: \red{$g_n(\omega^{froz}(0))$ starts to decorrelate from $g_n(\omega^{froz}(t))$ when the dynamics starts to affect pivotal points.}}
\medskip

We believe that this result also holds for the $\mu$-dynamical processes where $\mu$ is the law of a (non zero) planar Lévy process and for the dynamics considered in Theorem~\ref{t.AB}, but our methods are not quantitative enough to imply it.

\subsection{Some notations}

In all the paper, we will use the following notations: (a) $\grandO{1}$ is a positive bounded function, (b) $\Omega(1)$ is a positive function bounded away from $0$ and (c) if $f$ and $g$ are two non-negative functions, then $f \asymp g$ means $\Omega(1) f \leq g \leq \grandO{1} f$. Also:
\bi 
\item[i)] We write $B(x,r)=x+[-r,r]^2$ for any $x \in \R^2$ and $r \in \R_+$.
\item[ii)] We let $A(x;r,R) = x + [-R,R]^2 \setminus ]-r,r[^2$ for any $x \in \R^2$ and $0< r \leq R < +\infty$. Moreover, we let $A(r,R)=A(0;r,R)$.
\ei
We will also use the following notation, where $(\Omega,\mathcal{F},\Pro)$ is a probability space, $\mathcal{G} \subseteq \mathcal{F}$ is a $\sigma$-field, and $A,B \in \mathcal{F}$ are such that $\Pro \left[ B \right] > 0$:
\[
\Pro \left[ A \mid B, \, \mathcal{G} \right] = \frac{\Pro \left[ A \cap B \mid \mathcal{G} \right]}{\Pro \left[ B \mid \mathcal{G} \right]} \, \un_{\lbrace \Pro [ B \mid \mathcal{G} ] > 0 \rbrace} \, .
\]
Note that, $\Pro [ \: \cdot \mid B ]$-a.s., $\Pro[ A \, | \, B, \, \mathcal{G} ]$ is the conditional probability of $A$ with respect to $\mathcal{G}$ and under $\Pro [ \: \cdot \mid B]$.

\subsection{Estimates for crossing and arm events}\label{ss.arm}

We end the introduction by stating some results on crossing and arm events that we will use throughout this paper and that are proved in~\cite{V1,V2}. As explained above, we see these two papers as the first two of a series of three papers whose final goal is to study the annealed spectral sample (the present work one is the third one). In particular, we will rely a lot on the quasi-multiplicativity property for arm events from \cite{V1} (see Proposition \ref{p.QMulti} below), on the quenched estimates for arm events from \cite{V2} (see Theorem \ref{t.quenched_arm} below), on the study of the events $\widehat{\arm}_j(r,R)$ \blue{defined in} Definition \ref{d.hat}, and on several estimates on arm events proved in these two papers. In these papers, we highly rely on the annealed and quenched box-crossing properties proved in~\cite{ahlberg2015quenched,tassion2014crossing}. Let us first state these two properties.
\begin{defi}
Given $\rho_1,\rho_2 \in ]0,+\infty[$, the crossing event $\cross(\rho_1,\rho_2)$ is the event that there is a black path in the rectangle $[0,\rho_1] \times [0,\rho_2]$ from its left side to its right side.
\end{defi}
By duality, $\Pro_{1/2} \left[\cross(n,n) \right]=1/2$ \blue{(here, we use that almost surely all the vertices of the Voronoi tiling have degree $3$)}. Tassion has proved that the crossing probabilities for other shapes of rectangles are non-degenerate i.e. he has proved the following annealed box-crossing property:
\begin{thm}[\cite{tassion2014crossing}]
For every $\rho \in ]0,+\infty[$ there exists $c=c(\rho) \in ]0,1[$ such that, for every $R \in ]0,+\infty[$,
\[
c \leq \Pro_{1/2} \left[ \cross(\rho R,R) \right] \leq 1-c \, .
\]
\end{thm}
Ahlberg, Griffiths, Morris and Tassion have proved that the quenched crossing probabilities asymptotically do not depend on the environment $\eta$:
\begin{thm}[\cite{ahlberg2015quenched}]\label{t.AGMT}
There exists $\varepsilon > 0$ such that the following holds: For every $\rho \in ]0,+\infty[$ there exists a constant $C=C(\rho) <+\infty$ such that, for every $R \in ]0,+\infty[$,
\[
\Var \left( \Prob^\eta_{1/2} \left[ \cross(\rho R,R) \right] \right) \leq C R^{-\varepsilon} \, .
\]
\end{thm}

\paragraph{The quasi-multiplicativity property.}

Let us now focus on the arm events. In~\cite{V1}, we have proved that the arm events decay polynomially fast:
\begin{equation}\label{e.poly}
\forall j \in \N^*, \, \exists C=C(j) \in [1,+\infty[, \, \forall 1 \leq r \leq R, \, \frac{1}{C}\left(\frac{r}{R}\right)^C \leq \alpha_j^{an}(r,R) \leq C \left(\frac{r}{R}\right)^{1/C} \, .
\end{equation}
Moreover, we have proved that the quantities $\alpha_j^{an}(r,R)$ satisfy a quasi-multiplicativity property:
\begin{prop}[Proposition 1.6 of~\cite{V1}]\label{p.QMulti}
For every $j \in \N^*$, there exists $C=C(j) \in [1,+\infty[$ such that, for every $r_1,r_2,r_3 \in [1,+\infty[$ satisfying $r_1 \leq r_2 \leq r_3$, we have:
\[
\frac{1}{C} \alpha_j^{an}(r_1,r_3) \leq \alpha_j^{an}(r_1,r_2) \alpha_j^{an}(r_2,r_3) \leq  C \alpha_j^{an}(r_1,r_3) \, .
\]
\end{prop}

\paragraph{Quenched estimates on arm events.}

In~\cite{V2} we have studied quenched arm events (by following a strategy from~\cite{ahlberg2015quenched}) and we have roughly proved that with high probability the quenched probabilities do not depend on the environment $\eta$ (up to a constant). In particular, we have proved the following result, where we use the notation
\[
\widetilde{\alpha}_j(r,R) = \sqrt{\E \left[ \Prob_{1/2}^\eta \left[ \arm_j(r,R) \right]^2 \right]} \, .
\]
(We also write $\widetilde{\alpha}_j(R)=\widetilde{\alpha}_j(1,R)$.)
\begin{thm}[Theorem~1.4 of~\cite{V2}]\label{t.quenched_arm}
For every $j \in \N^*$, there exists $C=C(j)<+\infty$ such that, for every $1 \leq r \leq R  < +\infty$,
\[
\alpha_j^{an}(r,R) \leq \widetilde{\alpha}_j(r,R) \leq C \alpha_j^{an}(r,R) \, .
\]
\end{thm}
In Subsections~\ref{ss.why} and~\ref{ss.preli}, we explain why the quenched estimate Theorem~\ref{t.quenched_arm} is central in the present paper (see also the proofs - written in Subsection~\ref{ss.excep} - of Theorems~\ref{t.main_frozen} and~\ref{t.main_levy}).

\paragraph{A spatial independence result for arm events.}

In~\cite{V1}, we have studied the events $\widehat{\arm}_j(r,R)$ that are useful to apply spatial independence arguments:
\begin{defi}\label{d.hat}
If $j \in \N^*$ and $1 \leq r \leq R < +\infty$, we let\footnote{In  Section~\ref{s.ann}, we define spectral objects $\widehat{h}(S)$ where $h$ is a function. The hat symbol in $\widehat{\arm}_j(r,R)$ is not at all a spectral notation, but we have chosen to keep this notation (that was used in~\cite{V1,V2}) to facilitate references. We hope that it will not confuse the reader.}
\[
\widehat{\arm}_j(r,R) = \left\lbrace \Pro_{1/2} \left[ \arm_j(r,R) \cond \omega \cap A(r,R) \right] > 0 \right\rbrace \, .
\]
(Remember that $A(r,R)$ is the annulus $[-R,R]^2 \setminus ]-r,r[^2$.)
\end{defi}
Note that the events $\widehat{\arm}_j(r,R)$ are measurable with respect to $\omega \cap A(r,R)$ while the events $\arm_j(r,R)$ are not. We have the following estimate:
\begin{prop}[Propositions~2.4 and~D.8 of~\cite{V1} together with Theorem~1.4 of \cite{V2}]\label{p.hat}
For every $j \in \N^*$, there exists $C=C(j) \in [1,+\infty[$ such that, for every $1 \leq r \leq R < +\infty$,
\[
\alpha^{an}_j(r,R) \leq \Pro_{1/2} \left[ \widehat{\arm}_j(r,R) \right] \leq \sqrt{ \E \left[ \Prob_{1/2}^\eta \left[ \widehat{\arm}_j(r,R) \right]^2 \right]} \leq C \alpha_j^{an}(r,R) \, .
\]
\end{prop}
In the present work, we will use a lot the events $ \widehat{\arm}_j(r,R)$, see in particular Subsection~\ref{ss.findelathese} where we prove a quasi-multiplicativity property for coupled Voronoi configurations by following~\cite{garban2010fourier} and~\cite{V1}. This subsection is the most technical part of the paper and the study of events analogue to $\widehat{\arm}_j(r,R)$ (and the proof of an analogue of Proposition~\ref{p.hat} for coupled Voronoi configurations, see for instance Lemma \ref{l.fin4}) will be crucial in order to overcome the problems resulting from spatial dependencies.

\paragraph{Other estimates on arm events.}

Let us end this section by stating estimates for some specific values of $j$. We have the following estimates on the $4$-arm event:
\begin{prop}[Propositions~6.1 and~7.2 of~\cite{V2}]\label{p.4}
For every $\varepsilon > 0$, there exists $C=C(\varepsilon)<+\infty$ such that, for every $1 \leq r  \leq R < +\infty$,
\[
\alpha_4^{an}(r,R) \leq C \sqrt{\alpha_2^{an}(r,R)}\left( \frac{r}{R} \right)^{1-\varepsilon} \, .
\]
Moreover, there exists $c>0$ such that, for every $1 \leq r  \leq R < +\infty$,
\[
\alpha_4^{an}(r,R) \geq c \left( \frac{r}{R} \right)^{2-c} \frac{1}{\alpha_1^{an}(r,R)} \, .
\]
\end{prop}
By the (annealed) FKG property (see Lemma~14 of Chapter~8 of~\cite{bollobas2006critical}), the first part of Proposition~\ref{p.4} implies the following property.
\begin{cor}\label{c.4}
For every $\varepsilon > 0$, there exists $C=C(\varepsilon)<+\infty$ such that, for every $1 \leq r  \leq R < +\infty$:
\[
\alpha_4^{an}(r,R) \leq C \alpha_1^{an}(r,R) \left( \frac{r}{R} \right)^{1-\varepsilon} \, .
\]
\end{cor}
We also have the following estimates on the ``universal'' arm events:
\begin{prop}[Proposition~2.7 of~\cite{V1}]\label{p.univ}
For every $j \in \N^*$, let $\arm_j^+(r,R)$ denote the $j$-arm event in the half-plane\footnote{The definition is the same as $\arm_j(r,R)$ except that we ask that the paths are included in the half-plane.} and let $\alpha_j^{an,+}(r,R) = \Pro_{1/2} \left[ \arm_j^+(r,R) \right]$. We have:
\[
\alpha_2^{an,+}(r,R) \asymp r/R \text{ and } \alpha_3^{an,+}(r,R) \asymp \alpha_5^{an}(r,R) \asymp (r/R)^2 \, .
\]
\end{prop}

Let us finally note that~\eqref{e.poly}, the quasi-multiplicativity property Proposition \ref{p.QMulti}, Theorem~\ref{t.quenched_arm} and Proposition~\ref{p.hat} also apply to arm events in the half-plane (and the proofs are exactly the same).

\subsection{Organization of the paper}

\blue{The paper is organized as follows:
\begin{itemize}
\item In Section \ref{s.ann}, we introduce the annealed spectral sample, which is the main object of the paper. Moreover, we state the main spectral estimates of the paper (i.e.\ we state a clustering property as well as estimates on the full lower tail of the annealed spectral sample) and we use these estimates to prove the theorems stated in the introduction: the existence of exceptional times (Theorems \ref{t.main_frozen}, \ref{t.main_levy} and \ref{t.Brownian}) and the quantitative noise sensitivity result (Theorem \ref{t.quant_NS_frozen}).
\item In Section \ref{s.spectral}, we prove the spectral estimates stated in Section \ref{s.ann}. In this section, we are highly inspired by the methods from \cite{garban2010fourier} (and by the extension of these methods from \cite{GV}) and we use the results from \cite{V1,V2} to extend these methods to Voronoi percolation. In particular, we prove a quasi-multiplicativity property for the annealed spectral sample in Subsection \ref{ss.findelathese}. The notions of annealed and quenched pivotal events (Definitions \ref{d.piv} and \ref{d.piv_hat}) play a central role in Section \ref{s.spectral}.
\item The paper includes five appendices. In one of them - Appendix \ref{s.4hp} - we prove a result that we use in Section \ref{s.spectral} about the $4$-arm event conditioned on the configuration in a half-plane.
\end{itemize}}

\paragraph{Acknowledgments:} I would like to thank Christophe Garban for many helpful discussions and for his comments on earlier versions of the manuscript and Avelio Sepúlveda for help with the theory of Lévy processes. I would also like to thank Vincent Tassion for fruitful discussions. Finally, I am extremely grateful to an anonymous referee for his or her careful reading and many very helpful comments. This work was supported by the ERC grant Liko No 676999.

\section{The annealed spectal sample of Voronoi percolation}\label{s.ann}

\subsection{Definition of the annealed spectral sample}\label{ss.annealed}

In order to define the annealed spectral sample, we need to recall the definition of the spectral sample of Boolean functions from~\cite{garban2010fourier}. To this purpose, we first need to recall what is the Fourier decomposition of Boolean functions. Consider a countable set $E$ and equip the set $\Omega_E:=\{ -1,1 \}^E$ with the product ``uniform'' measure $\Prob^E_{1/2} = \left( \frac{\delta_1+\delta_{-1}}{2} \right)^{\otimes E}$. For every $S$ finite subset of $E$, define the following function on $\Omega_E$:
\begin{equation}\label{e.fourier_function}
\chi_S^E \, : \, \omega_E \in \Omega_E \mapsto \prod_{i \in S} \omega_E(i) \, .
\end{equation}
The functions $\chi_S^E$ form an orthonormal set of $L^2(\Omega_E,\Prob^E_{1/2})$. If $h$ is a function from $\Omega_E$ to $\R$ that depends on only finitely many coordinates, we can decompose $h$ on this orthonormal set:
\[
h = \sum_{S \text{ finite subset of } E} \widehat{h}(S) \chi_S^E \, ,
\]
where $\widehat{h}(S)=\Ex^E_{1/2} \left[ h \chi_S \right]$ (in particular, $\widehat{h}(S)=0$ if there exists $i \in S$ such that $h$ does not depend on the value of the coordinate $i$). The coefficients $\widehat{h}(S)$ are called the Fourier coefficients of $h$. The idea to use this decomposition in order to prove noise sensitivity results goes back to~\cite{benjamini1999noise}. Let $\omega_E(0)\sim \Prob_{1/2}^E$ and define the dynamical process $(\omega_E(t))_{t \in \R_+}$ by resampling each coordinate at rate $1$. The Fourier basis diagonalises this dynamics:
\[
\forall S,S' \text{ finite subsets of } E, \, \E \left[ \chi_S^E(\omega_E(0)) \chi_{S'}^E(\omega_E(t)) \right] = \un_{S=S'}e^{-t|S|} \, .
\]
As a result,
\begin{equation}\label{e.cov_Fourier}
\Cov \left( h(\omega_E), h(\omega_E(t)) \right) = \sum_{\emptyset \neq S \text{ finite subset of }E} \widehat{h}(S)^2 e^{-t|S|} \, .
\end{equation}
In~\cite{garban2010fourier}, the authors introduce a geometrical object: the spectral sample.
\begin{defi}[\cite{garban2010fourier}]\label{d.spec}
Let $h$ be a measurable function from $\Omega_E$ to $\R$ that is not the zero function and that depends only on a finite subset $F \subset E$. The spectral sample of $h$ is a random variable with values in the subsets of $F$ whose law $\widehat{\Pro}_h$ is given by
\[
\forall S \subseteq F, \, \widehat{\Pro}_h \left[ \{ S \} \right] = \frac{\widehat{h}(S)^2}{\sum_{S' \subseteq F} \widehat{h}(S')^2} = \frac{\widehat{h}(S)^2}{\Ex^E_{1/2} \left[ h^2 \right]} \, .
\]
Moreover, the un-normalized spectral measure $\widehat{\Q}_h$ is defined by $\widehat{\Q}_h \left[ \{ S \} \right] = \widehat{h}(S)^2$.
\end{defi}
With this notion, proving noise sensitivity of a Boolean function (at least for non-degenerate functions) is equivalent to proving that the cardinality of the spectral sample is \textbf{large or empty with high probability}.
\medskip

Let us now go back to the model of Voronoi percolation and introduce an annealed version of the spectral sample. Remember the definition of the sets $\Omega'$ and $\Omega$ from Subsection~\ref{ss.main}. We need three other notations: (a) for every measurable function $h$ from $\Omega$ to $\R$  and for every $\overline{\eta} \in \Omega'$, we write $h^{\overline{\eta}}$ for the restriction of $h$ to $\Omega_{\overline{\eta}} = \{ -1, 1 \}^{\overline{\eta}}$; (b) we write $S \subseteq_f E$ if $S$ is a finite subset of $E$, (c) we let $\mathcal{F}'$ be the (classical) $\sigma$-algebra on $\Omega'$ defined in Subsection~\ref{ss.main}.

\begin{defi}[The annealed spectral sample]\label{d.annealed_spectral_sample}
Let $h$ be a bounded measurable function from $\Omega $ to $\R$ which is not the zero function and assume that a.s. $h^\eta$ depends on finitely many points of $\eta$. An \textbf{annealed spectral sample} of $h$ is a random variable $\mathcal{S}^{an}_h$ with values in $\Omega'$ whose distribution $\widehat{\Pro}_h^{an}$ is defined by
\[
\forall A \in \mathcal{F}', \, \widehat{\Pro}_h^{an} \left[ A \right] = \frac{\E \left[ \sum_{S \subseteq_f \eta, \, S \in A} \widehat{h^\eta}(S)^2 \right]}{\E \left[ \sum_{S \subseteq_f \eta} \widehat{h^\eta}(S)^2 \right]} = \frac{\E \left[ \sum_{S \subseteq_f \eta, \, S \in A} \widehat{h^\eta}(S)^2 \right]}{\E \left[ h^2 \right]} \, ,
\]
where the coefficients $(\widehat{h^\eta}(S))_{S \subseteq_f \eta}$ are the Fourier coefficients of $h^\eta$. Also, we define the un-normalized measure $\widehat{\Q}_h^{an}$ on $\Omega'$ as follows:
\[
\widehat{\Q}_h^{an} \left[ A \right] = \E \left[ \sum_{S \subseteq_f \eta, \, S \in A} \widehat{h^\eta}(S)^2 \right] \, ,
\]
i.e.
\[
\widehat{\Q}_h^{an} \left[ A \right] = \E \left[ \widehat{\Q}_{h^\eta} \left[ A \cap \Omega_\eta \right] \right]
\]
where $\widehat{\Q}_{h^\eta}$ is the measure on the finite subsets of $\Omega_\eta$ from Definition~\ref{d.spec}.
\end{defi}

\blue{(See Appendix \ref{a.meas} for the proof that $\sum_{S \subseteq_f \eta, \, S \in A} \widehat{h^\eta}(S)^2$ is a measurable quantity.)} Note that the annealed spectral sample is a continuous point process \red{which is a.s. finite}. The following is the analogue of~\eqref{e.cov_Fourier}.
\begin{lem}\label{l.cov_frozen}
Take $h$ as in Definition~\ref{d.annealed_spectral_sample} and let $(\omega^{froz}(t))_{t \geq 0}$ be a frozen dynamical Voronoi percolation. Then, for all $t \geq 0$,
\[
\E \left[ h(\omega^{froz}(0)) h(\omega^{froz}(t)) \right] = \sum_{k \in \N} \widehat{\Q}_{h}^{an} \left[ |S| = k \right] e^{-kt} \, .
\]
\end{lem}
\begin{proof}
Let $\eta$ be the underlying point configuration. If $S \subseteq_f \eta$, we write $\chi_S^\eta$ for the Fourier function defined on $\Omega_\eta$ as in~\eqref{e.fourier_function}. We have
\begin{eqnarray*}
\E \left[ h(\omega^{froz}(0)) h(\omega^{froz}(t)) \right] & = & \E \left[ \left( \sum_{S \subseteq_f \eta} \widehat{h^\eta}(S) \chi_S^\eta(\omega^{froz}(0)) \right) \times \left( \sum_{S \subseteq_f \eta} \widehat{h^\eta}(S) \chi_S^\eta(\omega^{froz}(t)) \right) \right]\\
& = & \E \left[ \sum_{S,S' \subseteq_f \eta} \widehat{h^\eta}(S) \widehat{h^\eta}(S') \chi_S^\eta(\omega^{froz}(0)) \chi_{S'}^\eta(\omega^{froz}(t)) \right]\\
& = & \E \left[ \E \left[ \sum_{S,S' \subseteq_f \eta} \widehat{h^\eta}(S) \widehat{h^\eta}(S') \chi_S^\eta(\omega^{froz}(0)) \chi_{S'}^\eta(\omega^{froz}(t)) \cond \eta \right] \right]\\
& = & \E \left[ \sum_{S \subseteq_f \eta} \widehat{h^\eta}(S)^2 e^{-t|S|} \right]\\
& = & \sum_{k \in \N} \widehat{\Q}_{h}^{an} \left[ |S| = k \right] e^{-kt} \, .
\end{eqnarray*}
\end{proof}

Remember that in this paper we study two dynamical processes. Lemma \ref{l.cov_frozen} links the annealed spectral sample to the frozen dynamics. We now explain how we can use the annealed spectral sample in order to study the $\mu$-dynamical Voronoi percolation processes, where $\mu$ is the distribution of a Lévy process. The formula from Lemma~\ref{l.cov_frozen} comes from the fact that the Fourier basis diagonalises the dynamics defined by resampling the values of the bits independently from each other. We do not have such a property for the $\mu$-dynamical process. The same kind of difficulty arose in the study of exclusion dynamics, see~\cite{broman2013exclusion,GV}. The following lemma is inspired by Lemma~7.1 of~\cite{broman2013exclusion} and by Lemma~4.1 of~\cite{GV}. Remember that, if $\mu$ is the distribution of a planar Lévy process, we let $(\omega^\mu(t))_{t \in \R_+}$ be a $\mu$-dynamical Voronoi percolation process and we let $(\eta^\mu(t))_{t \in \R_+}$ be the underlying (non-colored) point process. \blue{For any deterministic finite subset of the plane $s(0)$, we define a (càdlàg) process $(s^{\mu}(t))_{t \in \R}$ in $\Omega'$ by letting each point of $s(0)$ move independently according to a Lévy process of law $\mu$. Note that the time set is $\R$; $(s^{\mu}(-t))_{t \geq 0}$ is the time reversal process of a process of law $\mu$ and $(s^{\mu}(-t))_{t \geq 0} \perp \!\!\! \perp (s^{\mu}(t))_{t \geq 0}$.}

\begin{lem}\label{l.cov_moving}
\blue{Take $h$ as in Definition~\ref{d.annealed_spectral_sample}, let $\mu$ be the distribution of a planar Lévy process and let $\mathcal{F}'$ be the (classical) $\sigma$-algebra defined on $\Omega'$ in Subsection~\ref{ss.main}. Let $(I_i)_{i \geq 1}$ be a partition of $\N=\{0,1,2,\cdots\}$ and, for each $i \in \{1,2,\cdots\}$, let $B_i = \{ \overline{\eta} \in \Omega' \, : \, |\overline{\eta}| \in I_i\} \in \mathcal{F}'$. Moreover, let $A_1,A_2,\cdots \in \mathcal{F}'$ \red{such that} $A_i \subseteq B_i$ for every $i \geq 1$. Also, for each $i \geq 1$ and each $t \geq 0$, let $\delta_1(i,t), \delta_2(i) > 0$ be such that the three following properties hold:
\begin{itemize}
\item[i)] $\underset{s(0)\in A_i}{\max} \, \Pro \left[ s^{\mu}(t) \in A_i \right] \leq \delta_1(i,t) \,$,
\item[ii)] $\underset{s(0)\in A_i}{\max} \, \Pro \left[ s^{\mu}(-t) \in A_i \right] \leq \delta_1(i,t) \,$,
\item[iii)] $\widehat{\Q}_h^{an} \left[ A_i \right] \geq (1- \delta_2(i))\widehat{\Q}_h^{an} \left[ B_i \right] \,$.
\end{itemize}
Then, for every $t \geq 0$,
\[
\E \left[ h(\omega^\mu(0)) h(\omega^\mu(t)) \right] \leq \E \left[ \Ex_{1/2}^\eta \left[ h \right]^2  \right] + \sum_{i=1}^{+\infty} \widehat{\Q}_h^{an} \left[ B_i \right] (\delta_1(i,t)+ 2 \sqrt{\delta_2(i)}) \, .
\]}
\end{lem}
\begin{proof}
\blue{We refer to Appendix \ref{a.meas} for all the measurability issues in this proof.} To simplify the notations, we write $\eta(t)=\eta^\mu(t)$. \blue{For every $S \subseteq \eta(0)$, we let $S_t$ be the corresponding subset of $\eta(t)$ (this set is a.s. well defined because there is a.s. no collusion of points in the process and no two particles jump at the same time, see Appendix \ref{a.meas} for more details). Note that, for every $i$, every $S \subseteq \eta(0)$ and every $t \geq 0$, $S \in B_i$ if and only if $S_t \in B_i$.} The quantity $\E \left[ h(\omega^\mu(0)) h(\omega^\mu(t)) \right]$ equals
\begin{align*}
& \E \left[ \left( \sum_{S \subseteq_f \eta(0)} \widehat{h^{\eta(0)}}(S) \chi_S^{\eta(0)}(\omega^\mu(0)) \right) \times \left( \sum_{S \subseteq_f \eta(t)} \widehat{h^{\eta(t)}}(S) \chi_S^{\eta(t)}(\omega^\mu(t)) \right) \right]\\
& = \E \left[ \E \left[ \sum_{S \subseteq_f \eta(0), S' \subseteq_f \eta(t)} \widehat{h^{\eta(0)}}(S)\widehat{h^{\eta(t)}}(S') \chi_S^{\eta(0)}(\omega^\mu(0)) \chi_{S'}^{\eta(t)}(\omega^\mu(t)) \cond (\eta(s))_{s \geq 0} \right] \right]\\
& = \E \left[ \sum_{S \subseteq_f \eta(0), S' \subseteq_f \eta(t)} \widehat{h^{\eta(0)}}(S)\widehat{h^{\eta(t)}}(S') \E \left[ \chi_S^{\eta(0)}(\omega^\mu(0)) \chi_{S'}^{\eta(t)}(\omega^\mu(t)) \cond (\eta(s))_{s \geq 0} \right] \right]\\
& = \E \left[ \sum_{S \subseteq_f \eta(0), S' \subseteq_f \eta(t)} \widehat{h^{\eta(0)}}(S) \widehat{h^{\eta(t)}}(S') \un_{S'=S_t } \right]\\
& = \E \left[ \sum_{S \subseteq_f \eta(0)} \widehat{h^{\eta(0)}}(S) \widehat{h^{\eta(t)}}(S_t) \right]\\
& = \E \left[ \widehat{h^{\eta(0)}}(\emptyset) \widehat{h^{\eta(t)}}(\emptyset) + \sum_{i=1}^{+\infty} \sum_{S \subseteq_f \eta(0), \, S \in B_i} \widehat{h^{\eta(0)}}(S) \widehat{h^{\eta(t)}}(S_t) \right]\\
& = \E \left[ \widehat{h^{\eta(0)}}(\emptyset)  \widehat{h^{\eta(t)}}(\emptyset) \right] + \sum_{i=1}^{+\infty} \E \left[ \sum_{S \subseteq_f \eta(0), \, S \in B_i} \widehat{h^{\eta(0)}}(S) \widehat{h^{\eta(t)}}(S_t) \right].
\end{align*}
The last equality follows from dominated convergence \blue{(indeed, by the Cauchy-Schwarz inequality, for any $B \in \mathcal{F}'$ we have $\left| \sum_{S \subseteq_f \eta(0), \, S \in B} \widehat{h^{\eta(0)}}(S) \widehat{h^{\eta(t)}}(S_t) \right| \leq \Ex^\eta [ h^2 ] \leq \parallel h \parallel_\infty^2$; recall that $h$ is bounded)}.

Since by the Cauchy-Schwarz inequality we have $\E \left[ \widehat{h^{\eta(0)}}(\emptyset)  \widehat{h^{\eta(t)}}(\emptyset) \right] \leq \E \left[ \Ex_{1/2}^\eta \left[ h \right]^2 \right]$, it is now sufficient to prove that, for every $i \geq 1$,
\begin{equation}\label{e.only_for_i}
\E \left[ \sum_{S \subseteq_f \eta(0), \, S \in B_i} \widehat{h^{\eta(0)}}(S) \widehat{h^{\eta(t)}}(S_t) \right] \leq \widehat{\Q}_h^{an} \left[ B_i \right] (\delta_1(i,t)+ 2 \sqrt{\delta_2(i)}) \, .
\end{equation}
Let us divide the above sum into three sums:
\begin{multline*}
\E \left[ \sum_{S \subseteq_f \eta(0), S \in B_i} \widehat{h^{\eta(0)}}(S) \widehat{h^{\eta(t)}}(S_t) \right] = \E \left[\sum_{S \subseteq_f \eta(0), \, S \in A_i} \widehat{h^{\eta(0)}}(S) \widehat{h^{\eta(t)}}(S_t) \un_{S_t \in A_i} \right]\\
+ \E \left[ \sum_{S \subseteq_f \eta(0), \, S \in A_i} \widehat{h^{\eta(0)}}(S) \widehat{h^{\eta(t)}}(S_t) \un_{S_t \notin A_i} \right]+ \E \left[ \sum_{S \subseteq_f \eta(0), \, S \in B_i \setminus A_i} \widehat{h^{\eta(0)}}(S) \widehat{h^{\eta(t)}}(S_t) \right] \, .
\end{multline*}
Let us write $\Sigma_1$, $\Sigma_2$ and $\Sigma_3$ for the three terms of the right-hand-side of the above equality and let us first deal with $\Sigma_1$. By the Cauchy-Schwarz inequality (applied to the counting measure and then to $\E$), we have
\begin{align*}
\Sigma_1 & \leq \E \left[ \left( \sum_{S \subseteq_f \eta(0) , \, S \in A_i} \widehat{h^{\eta(0)}}(S)^2 \un_{S_t \in A_i} \right)^{1/2} \left( \sum_{S \subseteq_f \eta(0), \, S \in A_i} \widehat{h^{\eta(t)}}(S_t)^2 \un_{S_t \in A_i} \right)^{1/2} \right]\\
& \leq \left( \E \left[ \sum_{S \subseteq_f \eta(0), \, S \in A_i} \widehat{h^{\eta(0)}}(S)^2 \un_{S_t \in A_i} \right] \E \left[ \sum_{S \subseteq_f \eta(0), \, S \in A_i} \widehat{h^{\eta(t)}}(S_t)^2 \un_{S_t \in A_i} \right] \right)^{1/2}.
\end{align*}
\blue{We have
\begin{align*}
\E \left[ \sum_{S \subseteq_f \eta(0), \, S \in A_i} \widehat{h^{\eta(0)}}(S)^2 \un_{S_t \in A_i} \right] & = \E \left[ \sum_{S \subseteq_f \eta(0), \, S \in A_i} \widehat{h^{\eta(0)}}(S)^2 \Pro \left[ S_t \in A_i \cond \eta(0) \right] \right]\\
& \leq \delta_1(i,t) \widehat{\Q}_h^{an} \left[ A_i \right]
\end{align*}
by i). Given a set $S' \subseteq \eta(t)$, let $S'_0$ be the corresponding subset of $\eta(0)$. We have
\begin{align*}
\E \left[ \sum_{S \subseteq_f \eta(0), \, S \in A_i} \widehat{h^{\eta(t)}}(S_t)^2 \un_{S_t \in A_i} \right] & = \E \left[ \sum_{S' \subseteq_f \eta(t), \, S' \in A_i} \widehat{h^{\eta(t)}}(S')^2 \un_{S_0' \in A_i} \right]\\ & = \E \left[ \sum_{S' \subseteq_f \eta(t), \, S' \in A_i} \widehat{h^{\eta(t)}}(S')^2 \Pro \left[ S_0' \in A_i \cond \eta(t) \right] \right]\\
& \leq \delta_1(i,t) \widehat{\Q}_h^{an} \left[ A_i \right]
\end{align*}
by ii). As a result,
\[
\Sigma_1 \leq \sqrt{ \left( \delta_1(i,t) \, \widehat{\Q}_h^{an} \left[ A_i \right] \right)^2}  \leq \delta_1(i,t) \, \widehat{\Q}_h^{an} \left[ B_i \right] \, .
\]
}Let us now deal with the terms $\Sigma_2$ and $\Sigma_3$. By applying the Cauchy-Schwarz inequality once again, we obtain that
\begin{align*}
\Sigma_2 & \leq \left( \E \left[ \sum_{S \subseteq_f \eta(0), \, S \in A_i} \widehat{h^{\eta(0)}}(S)^2 \right] \, \E \left[ \sum_{S \subseteq_f \eta(0), \, S \in A_i} \widehat{h^{\eta(t)}}(S_t)^2 \un_{S_t \notin A_i} \right] \right)^{1/2}\\
& \leq \left( \widehat{\Q}_h^{an} \left[ A_i \right] \, \widehat{\Q}_{h}^{an} \left[ B_i \setminus A_i \right] \right)^{1/2} \leq \sqrt{\delta_2(i)} \, \widehat{\Q}_h^{an} \left[ B_i \right] \, .
\end{align*}
\blue{(In the second inequality, we have used that $S \in A_i \Rightarrow S \in B_i \Rightarrow S_t \in B_i$.)} By the same calculations, we prove that $\Sigma_3 \leq \sqrt{\delta_2(i)} \, \widehat{\Q}_h^{an} \left[ B_i \right]$, which implies~\eqref{e.only_for_i} and ends the proof.
\end{proof}

\subsection{Why don't we study a quenched spectral sample?}\label{ss.why}

Why do we need to define an \textbf{annealed} spectral sample? At first sight, it may seem to be easier to work with a quenched spectral sample, which could be defined exactly as in~\cite{garban2010fourier}. However, the quenched model is not translation invariant (which is very important in~\cite{garban2010fourier}) and we do not have quantitative enough quenched quasi-multiplicativity properties to be able to follow the strategy of Garban, Pete and Schramm at the quenched level. This is why we have chosen to introduce an annealed analogue of the spectral sample. It is important to keep in mind that this is a \textbf{continuous} point process. In particular, we will study events of the kind:
\[
\{ \mathcal{S}_h^{an} \cap B  \neq \emptyset = \mathcal{S}_h^{an} \cap W \} \, ,
\]
where $B$ and $W$ are two Borel subsets of the plane. Events of the kind
\[
\{ \mathcal{S}_h^{an} \cap B'  \neq \emptyset = \mathcal{S}_h^{an} \cap W' \} \, ,
\]
where $B'$ and $W'$ are subsets of $\eta$ (that seem at first sight to be the natural analogues of the events studied in~\cite{garban2010fourier}) would not make any sense since we have not coupled the spectral sample with $\eta$.
\medskip

However, we will also need to work at the quenched level to apply discrete Fourier techniques. Remember that the un-normalized measure of the spectral sample is
\[
\widehat{\Q}^{an}_h \left[ \cdot \right] = \E \left[ \widehat{\Q}_{h^\eta} \left[ \cdot \right] \right] \, .
\]
The strategy will consist in applying discrete Fourier results to $\widehat{\Q}_{h^\eta}$ and then deducing results for $\widehat{\Q}^{an}_h$. The technical difficulties will come from the \textbf{multiple passages from quenched to annealed measures}. To overcome these difficulties, the key result will be Theorem~\ref{t.quenched_arm}. Let us recall that this theorem says that
\[
\widetilde{\alpha}_j(r,R) := \sqrt{\E \left[ \Prob_{1/2}^\eta \left[ \arm_j(r,R) \right]^2 \right]} \asymp \alpha_j^{an}(r,R) \, .
\]
In other words,
\[
\Var \left( \Prob_{1/2}^\eta \left[ \arm_j(r,R) \right] \right) \leq \grandO{1} \Pro_{1/2} \left[ \arm_j(r,R) \right]^2 = \grandO{1} \alpha_j^{an}(r,R)^2 \, .
\]
Therefore, this theorem roughly means that the arm events do not depend too much on the environment. Thanks to this result, we will not lose too much each time we go from quenched to annealed measures. For a more precise explanation of the importance of Theorem~\ref{t.quenched_arm} in the present work, we refer to Subsection~\ref{ss.preli}.

\subsection{Main results on the annealed spectral sample}

Let us now state the main results on the annealed spectral sample, which are the analogues of results from~\cite{garban2010fourier,GV} and which will enable us to prove that there exist exceptional times. As explained in \cite{garban2010fourier}, the expected size of the spectral sample of the crossing event for Bernoulli percolation is $n^2\alpha_4(n)$ since the spectral sample has the same one-dimensional marginal as the pivotal set. By the same argument together with the arm event estimates from \cite{V1,V2} (see Subsection \ref{ss.arm}), one can obtain that the expected size of the annealed spectral sample of the crossing event of Voronoi percolation is $n^2\alpha_4^{an}(n)$. As a result, the following theorem can be seen as an estimate on the full lower tail of the annealed spectral sample of Voronoi percolation.
\begin{thm}\label{t.Spec_sample_g_n}
For every $n \in \N^*$, let $g_n$ denote the indicator function of the crossing event of $[0,n]^2$. There exists $C<+\infty$ such that for every $r \in [1,n]$ we have
\[
\Pro \left[0<| \mathcal{S}^{an}_{g_n}| < r^2\alpha_4^{an}(r) \right] \leq C \left( \frac{n}{r} \alpha_4^{an}(r,n) \right)^2 + \frac{C}{n} \, . 
\]
\end{thm}
\begin{rem}
The term $\frac{C}{n}$ is not present in the analogous result from~\cite{garban2010fourier}. We will see that this term comes from the contribution of the boxes outside of $[0,n]^2$. Actually, since it is conjectured that $\left( \frac{n}{r} \alpha_4^{an}(r,n) \right)^2 = (r/n)^{1/2+\petito{1}}$, the result should be true without this term.
\end{rem}
It is not difficult to deduce Theorem~\ref{t.quant_NS_frozen} from Theorem~\ref{t.Spec_sample_g_n}. As one can see below, the quenched estimate Theorem~\ref{t.AGMT} by Ahlberg, Griffiths, Morris and Tassion is also a key result of the proof of Theorem \ref{t.quant_NS_frozen}.
\begin{proof}[Proof of Theorem~\ref{t.quant_NS_frozen} using Theorem~\ref{t.Spec_sample_g_n}]
The proof in the case $t_n n^2 \alpha_4^{an}(n) \ll 1$ does not rely on spectral estimates and is the purpose of Appendix~\ref{a.t_n} so let us only consider the case where $t_n n^2\alpha_4(n)$ goes to $+\infty$. Remember that the renormalisation constant in the definition of the distribution $\widehat{\Pro}_{g_n}^{an}$ of $\mathcal{S}_{g_n}^{an}$ is $\E_{1/2} \left[g_n \right]= \Pro_{1/2} \left[\cross(n,n) \right](=1/2)$. Hence, Lemma~\ref{l.cov_frozen} implies that
\begin{multline*}
\Cov \left( g_n(\omega^{foz}(0)),g_n(\omega^{froz}(t_n)) \right) = \sum_{k \geq 1} \Pro \left[ | \mathcal{S}^{an}_{g_n}| = k\right]  \Pro_{1/2} \left[\cross(n,n) \right] e^{-kt_n}\\
+ \Pro \left[ | \mathcal{S}^{an}_{g_n}| = 0 \right]  \Pro_{1/2} \left[\cross(n,n) \right] -  \Pro_{1/2} \left[\cross(n,n) \right]^2 \, .
\end{multline*}
Since $\Pro \left[ | \mathcal{S}^{an}_{g_n}| = 0 \right] \Pro_{1/2} \left[\cross(n,n) \right] = \E \left[ \Prob_{1/2}^\eta \left[ \cross(n,n) \right]^2 \right]$, the quenched estimate Theorem~\ref{t.AGMT} implies that
\[
\Pro \left[ | \mathcal{S}^{an}_{g_n}| = 0 \right]  \Pro_{1/2} \left[\cross(n,n) \right] -  \Pro_{1/2} \left[\cross(n,n) \right]^2
\]
goes to $0$ as $n$ goes to $+\infty$. It is thus sufficient to prove that
\begin{equation}\label{e.QUANT_DANS_PR}
\sum_{k \geq 1} \Pro \left[ | \mathcal{S}^{an}_{g_n}| = k\right]  \Pro_{1/2} \left[\cross(n,n) \right] e^{-kt_n} \left( = \frac{1}{2}\sum_{k \geq 1} \Pro \left[ | \mathcal{S}^{an}_{g_n}| = k\right]  e^{-kt_n} \right)
\end{equation}
goes to $0$ as $n$ goes to $+\infty$. This is actually a direct consequence of Theorem~\ref{t.Spec_sample_g_n}. For details, we refer to Section~8.1 of~\cite{garban2010fourier} where Garban, Pete and Schramm prove the analogue of~\eqref{e.QUANT_DANS_PR} for Bernoulli percolation by using the analogue of Theorem~\ref{t.Spec_sample_g_n}. In the said article, the authors use the following properties of the probabilities of arm events of Bernoulli percolation $\alpha_j(\cdot,\cdot)$: i) they decay polynomially fast, ii) they satisfy the quasi-multiplicativity property, iii) $\Omega(1) (n/m)^{2-\Omega(1)} \leq \alpha_4(n,m)\leq \grandO{1} (n/m)^{1+\Omega(1)}$ if $1 \leq n \leq m < +\infty$. All these properties also hold for the quantities $\alpha_j^{an}(\cdot,\cdot)$ (see Subsection~\ref{ss.arm}).
\end{proof}

Let us now state the main results about the annealed spectral sample of the $1$-arm event. Below, we let $f_R$ denote the indicator function of the $1$-arm event i.e. $f_R=\un_{\arm_1(1,R)}$.

\begin{thm}\label{t.Spec_sample_f_R}
There exists $C<+\infty$ such that, if $R \in [1,+\infty[$ and $r \in [1,R]$, then
\[
\Pro \left[0<| \mathcal{S}^{an}_{f_R}| \leq r^2\alpha_4^{an}(r) \right] \leq C \alpha_1^{an}(r,R) \, . 
\]
\end{thm}
In Subsection~\ref{ss.excep}, we explain how to deduce Theorem~\ref{t.main_frozen} from Theorem~\ref{t.Spec_sample_f_R}.

\medskip

The following result shows a \textbf{clustering effect} for the annealed spectral sample. In this theorem, we estimate the probability that a small spectral mass is far from the origin (see Subsection 2.5 of \cite{GV} for a discussion about the analogous result in Bernoulli percolation).
\begin{thm}\label{t.Spec_sample_f_R_bis}
There exists $\varepsilon_0>0$ and $C<+\infty$ such that, for every $1 \leq r \leq r_0 \leq R/2<+\infty$,
\[
\Pro \left[ 0<|\mathcal{S}^{an}_{f_R} \setminus [-r_0,r_0]^2 | < r^2\alpha_4^{an}(r) \right] \leq C \alpha_1^{an}(r_0,R) \left( \frac{r_0}{r} \right)^{1-\varepsilon_0} \alpha^{an}_4(r,r_0) \, .
\]
\end{thm}
In Subsection~\ref{ss.excep}, we explain how to deduce Theorem~\ref{t.main_levy} from Theorem~\ref{t.Spec_sample_f_R_bis}. We will actually rely on the following corollary of Theorem~\ref{t.Spec_sample_f_R_bis}:
\begin{cor}\label{c.f_R}
Let $\varepsilon_0$ be the constant of Theorem~\ref{t.Spec_sample_f_R_bis}. There exists a constant $C<+\infty$ such that, for all $1 \leq r \leq r_0 < +\infty$ and all $R \in [1,+\infty[$,
\[
\Pro \left[ |\mathcal{S}^{an}_{f_R}| < r^2 \alpha_4^{an}(r), \, \mathcal{S}^{an}_{f_R} \nsubseteq [-r_0,r_0]^2 \right] \leq C \, \alpha_1^{an}(r_0,R) \left( \frac{r_0}{r} \right)^{1-\varepsilon_0} \alpha_4^{an}(r,r_0) \, .
\]
\end{cor}
\begin{proof}[Proof of Corollary~\ref{c.f_R} using Theorem~\ref{t.Spec_sample_f_R_bis}]
If $r_0 \geq 2R$ we have:
\begin{eqnarray*}
\Pro \left[ \mathcal{S}^{an}_{f_R} \nsubseteq [-r_0,r_0]^2 \right] & = & \frac{\E \left[ \widehat{\Q}_{f_R^\eta} \left[ S \nsubseteq [-r_0,r_0]^2 \right] \right]}{\E_{1/2} \left[ f_R \right]}\\
& \leq & \frac{\Pro \left[ f_R^\eta \text{ depends on some points outside of } [-r_0,r_0]^2 \right]}{\E_{1/2} \left[ f_R \right]} \, .
\end{eqnarray*}
The result now follows from the fact that $\E_{1/2} \left[ f_R \right]=\alpha_1^{an}(R)$ decays polynomially fast in $R \leq r_0$ while $\Pro \left[ f_R^\eta \text{ depends on some points outside of } [r_0,r_0]^2 \right]$ decays super-exponentially fast in $r_0$. Indeed, if $f_R^\eta$ depends on some points outside of $[-r_0,r_0]^2$ then there is a Voronoi cell that intersects both $B(0,R) \subseteq B(0,r_0/2)$ and $\partial B(0,r_0)$, which has probability less than $\grandO{1} \exp \left( -\Omega(1) r_0^2 \right)$ by simple properties of Poisson point processes.\\

If $r_0 \leq R/2$, then the result is a direct consequence of Theorem~\ref{t.Spec_sample_f_R_bis} since
\[
\{ |\mathcal{S}^{an}_{f_R}| < r^2 \alpha_4^{an}(r), \, \mathcal{S}^{an}_{f_R} \nsubseteq [-r_0,r_0]^2 \} \subseteq \{ 0<|\mathcal{S}^{an}_{f_R} \setminus [-r_0,r_0]^2 | < r^2\alpha_4^{an}(r) \} \, .
\]
If $r_0 \in [R/2,2R]$, then this is direct consequence of the quasi-multiplicativity property (and of~\eqref{e.poly}) and of the result for $r_0=R/2$ since
\[
\{ |\mathcal{S}^{an}_{f_R}| < r^2 \alpha_4^{an}(r), \, \mathcal{S}^{an}_{f_R} \nsubseteq [-r_0,r_0]^2 \} \subseteq \{ |\mathcal{S}^{an}_{f_R}| < r^2 \alpha_4^{an}(r), \, \mathcal{S}^{an}_{f_R} \nsubseteq [-R/2,R/2]^2 \} \, .
\]
This ends the proof of the corollary.
\end{proof}

\subsection{Proofs of existence of exceptional times}\label{ss.excep}

In this subsection, we prove the results of existence of exceptional times by using Theorems~\ref{t.Spec_sample_f_R} and~\ref{t.Spec_sample_f_R_bis}. \textbf{In this section, we assume that the reader has read Section~8 of~\cite{garban2010fourier} and Section~4 of~\cite{GV}}\footnote{In Section $4$ of \cite{GV}, the reader does not have to read the proofs of Lemmas 4.1 and 4.4 and can stop just before the paragraph ``The constant $\alpha_0=217/816$''.} where results of existence of exceptional times are proved by using analogues of Theorems~\ref{t.Spec_sample_f_R} and~\ref{t.Spec_sample_f_R_bis}.
\medskip

We start with the following lemma that takes its roots in~\cite{olle1997dynamical}. Let $\mu$ be the distribution of a planar Lévy process, and let $\left( \omega(t) \right)_{t \geq 0}$ be either a frozen dynamical Voronoi percolation process or a $\mu$-dynamical Voronoi percolation process. Remember that $f_R = \un_{\arm_1(1,R)}$ and let
\[
X_R = \int_0^1 f_R \left( \omega(t) \right) dt \, .
\]
\begin{lem}\label{l.secondmoment}
Assume that there exists a constant $C < + \infty$ such that for all $R \in [1,+\infty[$ we have:
\[
\E \left[ X_R^2 \right] \leq C \E \left[ X_R \right]^2.
\]
Then \textit{a.s.} there are exceptional times at which there is an unbounded black component.
\end{lem}
Lemma~\ref{l.secondmoment} is proved in Appendix~\ref{a.2nd}. Let us now use this lemma to show that Theorems~\ref{t.Spec_sample_f_R} (resp. Theorem~\ref{t.Spec_sample_f_R_bis}) implies Theorem~\ref{t.main_frozen} (resp. Theorem~\ref{t.main_levy}). To this purpose, let us study $\E \left[ X_R \right]$ and $\E \left[ X_R^2 \right]$. First note that:
\begin{equation}\label{e.avant_fub}
\E \left[ X_R \right] = \alpha_1^{an}(R) \, .
\end{equation}
and (by Fubini):
\begin{eqnarray}
\E \left[ X_R^2 \right] & = & \int_0^1 \int_0^1 \E \left[ f_R(\omega(s)) f_R(\omega(t)) \right] ds dt \nonumber\\
& \leq & 2 \int_0^1 \E \left[ f_R(\omega(0)) f_R(\omega(t)) \right] dt \, .\label{e.fub}
\end{eqnarray}
\begin{proof}[Proof of Theorem~\ref{t.main_frozen} using Theorem~\ref{t.Spec_sample_f_R}]
By Lemmas~\ref{l.cov_frozen} and~\ref{l.secondmoment} (and by~\eqref{e.avant_fub} and~\eqref{e.fub}), it is sufficient to show that
\[
\Pro \left[ |\mathcal{S}^{an}_{f_R}|=0 \right]  + \int_0^1 \sum_{k \geq 1} \Pro \left[ |\mathcal{S}^{an}_{f_R}|=k \right] e^{-kt} dt \leq \grandO{1} \alpha_1^{an}(R) \, .
\]
(Indeed, the renormalisation constant in the distribution $\widehat{\Pro}^{an}_{f_R}$ of $\mathcal{S}_{f_R}^{an}$ is $\E_{1/2} \left[ f_R \right] = \alpha_1^{an}(R)$, this is why there is $\alpha_1^{an}(R)$ instead of $\alpha_1^{an}(R)^2$ on the right-hand-side.)\\

Let us first explain why $\Pro \left[ |\mathcal{S}^{an}_{f_R}|=0 \right] \leq \grandO{1} \alpha_1^{an}(R)$. To this purpose, note that:
\[
\Pro \left[ |\mathcal{S}^{an}_{f_R}|=0 \right] = \frac{\E \left[ \widehat{f_R^\eta}(\emptyset)^2 \right]}{\alpha_1^{an}(R)} = \frac{\widetilde{\alpha}_1(R)^2}{\alpha_1^{an}(R)} \, ,
\]
where the notation $\widetilde{\alpha}_1(R)$ comes from~Subsection~\ref{ss.arm}. The fact that the above is at most of the order of $\alpha^{an}_1(R)$ is given by Theorem~\ref{t.quenched_arm}.
\medskip

It is thus sufficient to prove that
\[
\int_0^1 \sum_{k \geq 1} \Pro \left[ |\mathcal{S}^{an}_{f_R}|=k \right] e^{-kt} dt \leq \grandO{1} \alpha_1^{an}(R) \, ,
\]
which is a consequence of Theorem~\ref{t.Spec_sample_f_R}. The proof of the analogous estimate (by using the analogue of Theorem~\ref{t.Spec_sample_f_R}) for Bernoulli percolation on $\Z^2$ is written in Section~9 of~\cite{garban2010fourier} (see also Section~6 of Chapter~XI of~\cite{garban2014noise}). In the said article, the authors use the following properties of the probabilities of arm events of Bernoulli percolation $\alpha_j(\cdot,\cdot)$: i) they decay polynomially fast, ii) they satisfy the quasi-multiplicativity property, iii) $\alpha_1(r,R)^{-1} (r/R)^{2-\Omega(1)} \leq \grandO{1} \alpha_4(r,R)$. All these properties hold for the quantities $\alpha_j^{an}(\cdot,\cdot)$ (see Subsection~\ref{ss.arm}), so the proof adapts readily.
\end{proof}

\begin{proof}[Proof of Theorem~\ref{t.Brownian} using Theorem~\ref{t.Spec_sample_f_R}]
If we follow the proof of Lemma~\ref{l.cov_frozen} and the beginning of the proof of Lemma~\ref{l.cov_moving} we obtain that, if $(\omega(t))_{t \in \R_+}$ is the process from Theorem~\ref{t.Brownian}, then
\[
\E \left[ f_R(\omega(0)) f_R(\omega(t)) \right] = \E \left[ \sum_{S \subseteq_f \eta(0)} \widehat{f_R^{\eta(0)}}(S)\widehat{f_R^{\eta(t)}}(S_t) e^{-t|S|} \right] \, .
\]
(where for every $S \subseteq \eta(0)$, $S_t$ is the corresponding subset of $\eta(t)$). By the Cauchy-Schwarz inequality (applied twice), this is less than or equal to
\begin{multline*}
\E \left[ \sqrt{\sum_{S \subseteq_f \eta(0)} \widehat{f_R^{\eta(0)}}(S)^2 e^{-t|S|}} \sqrt{\sum_{S \subseteq_f \eta(0)} \widehat{f_R^{\eta(t)}}(S_t)^2 e^{-t|S|}} \right]\\
\leq \E \left[ \sum_{S \subseteq_f \eta(0)} \widehat{f_R^{\eta(0)}}(S)^2 e^{-t|S|}\right] = \E \left[ f_R(\omega^{froz}(0)) f_R(\omega^{froz}(t)) \right] \, ,
\end{multline*}
where $(\omega^{froz}(t))_{t \in \R_+}$ is the frozen dynamical process. As a result, the proof of Theorem~\ref{t.Brownian} is the same as the proof of Theorem~\ref{t.main_frozen} provided that Lemma~\ref{l.secondmoment} also applies in the case of $(\omega(t))_{t \in \R_+}$. It indeed does, and the proof is exactly the same as for the $\mu$-dynamical processes (see Appendix~\ref{a.2nd}).
\end{proof}

\begin{proof}[Proof of Theorem~\ref{t.main_levy} using using Theorem~\ref{t.Spec_sample_f_R} and Corollary~\ref{c.f_R}]
For every $i \in \N^*$, let $A_i,B_i \subseteq \Omega'$ be defined as follows, where $\beta \in ]1,+\infty[$ will be chosen later:
\[
B_i = \{ \overline{\eta} \in \Omega' \, : \, |\overline{\eta}| \in [2^i,2^{i+1}-1] \}
\]
and
\[
A_i = \{ \overline{\eta} \in B_i \, : \, \overline{\eta} \subseteq [-2^{i\beta},2^{i\beta}]^2 \} \subseteq B_i \, .
\]
\red{As above Lemma \ref{l.cov_moving}, for any deterministic finite subset of the plane $s(0)$, we define a process $(s^{\mu}(t))_{t \in \R}$ by letting each point of $s(0)$ move independently according to a Lévy process of law $\mu$. Let $\delta_1(i,t)$ and $\delta_2(i)$ be defined by
\[
\delta_1(i,t) = \max \left\lbrace  \underset{s(0)\in A_i}{\max} \, \Pro \left[ s^{\mu}(t) \in A_i \right] , \underset{s(0)\in A_i}{\max} \, \Pro \left[ s^{\mu}(-t) \in A_i \right] \right\rbrace 
\]}
and
\[
\delta_2(i) = 1 - \frac{\widehat{\Pro}^{an}_{f_R}\left[A_i\right]}{ \widehat{\Pro}^{an}_{f_R}\left[B_i\right]} \, .
\]
By Lemmas~\ref{l.cov_moving} and~\ref{l.secondmoment}, it is sufficient to prove that:
\begin{equation}
\int_0^1 \left( \frac{\E \left[ \Ex_{1/2}^\eta \left[ f_R \right]^2 \right]}{\alpha_1^{an}(R)} + \sum_{i \geq 1} \widehat{\Pro}^{an}_{f_R} \left[ B_i \right] (\delta_1(i,t) + 2\sqrt{\delta_2(i)} ) \right) dt  \leq \grandO{1} \alpha_1^{an}(R) \, .
\end{equation}
By Theorem~\ref{t.AGMT}, we have $\E \left[ \Ex_{1/2}^\eta \left[ f_R \right]^2 \right] = \widetilde{\alpha}_1(R)^2 \leq \grandO{1} \alpha^{an}_1(R)^2$ so it only remains to estimate $\int_0^1 \sum_{i \geq 1} \widehat{\Pro}^{an}_{f_R} \left[ B_i \right] (\delta_1(i,t) + 2\sqrt{\delta_2(i)} ) dt $. Let $X \sim \mu$. Remember that we have assumed that, for each $L\in [1,+\infty[$ and each $t \in [0,1]$, $\Pro \left[ ||X_t||_2 \geq L \right] \geq c t L^{-\alpha}$ for some $c>0$. This implies that
\begin{equation}\label{e.easypourlevy}
\forall i \geq 1, \, \forall t \in [0,1], \, \delta_1(i,t) \leq \exp(-c't (2^i)^{1-\alpha\beta} )
\end{equation}
for some $c'>0$. It thus remains to prove that, if $\alpha$ is sufficiently small, then we can choose $\beta \in ]1,+\infty[$ so that we both have
\begin{equation}\label{e.fin_1}
\int_0^1  \sum_{i \geq 1} \widehat{\Pro}^{an} \left[ B_i \right] \exp(-c' t (2^i)^{1-\alpha\beta} ) dt \leq \grandO{1} \alpha_1^{an}(R)
\end{equation}
and
\begin{equation}\label{e.fin_2}
\sum_{i \geq 1} \widehat{\Pro}^{an} \left[ B_i \right] \sqrt{\delta_2(i)} \leq \grandO{1} \alpha_1^{an}(R) \, .
\end{equation}
The quantity $\widehat{\Pro}^{an} \left[ B_i \right]$ can be estimated by using Theorem~\ref{t.Spec_sample_f_R} and the quantity
\[
\widehat{\Pro}^{an} \left[ B_i \right] \delta_2(i) = \Pro \left[ | \mathcal{S}_{f_R}^{an}| \in [2^i,2^{i+1}-1], \, \mathcal{S}_{f_R}^{an} \nsubseteq [-2^{\beta i},2^{\beta i}]^2 \right]
\]
can be estimated by using Corollary~\ref{c.f_R}.
\medskip

In Section~4 of~\cite{GV}, we have proved with Christophe Garban that if $\alpha$ is sufficiently small then there exists $\beta \in ]1,+\infty[$ such that the analogues of~\eqref{e.fin_1} and~\eqref{e.fin_2} hold for Bernoulli percolation. In the said article, we have used the analogues of Theorem~\ref{t.Spec_sample_f_R} and Corollary~\ref{c.f_R} and the following properties of the probabilities of arm events: i) they decay polynomially fast, ii) they satisfy the quasi-multiplicativity property. iii) $\Omega(1) \alpha_1(r,R)^{-1}(r/R)^{2-\Omega(1)} \leq \alpha_4(r,R) \leq \grandO{1} (r/R) \alpha_1(r,R)$. By the results of Subsection~\ref{ss.arm}, the probabilities of arm events in Voronoi percolation satisfy all these properties, except maybe $\alpha_4(r,R) \leq \grandO{1} (r/R) \alpha_1(r,R)$. However, Corollary~\ref{c.4} implies that $\alpha_4^{an}(r,R) \leq \grandO{1} (r/R)^{1-\varepsilon} \alpha_1^{an}(r,R)$ for any $\varepsilon>0$. Actually, this slightly weaker property would have been enough in~\cite{GV} (and would not have required any change in the proof). We refer to Section~4 of~\cite{GV} for more details (let us note that the reader does not have to read the proofs of Lemmas~4.1 and 4.4 of \cite{GV} since they are the analogues of Lemma~\ref{l.cov_moving} and of~\eqref{e.easypourlevy} respectively; moreover, the reader can stop just before the paragraph ``The constant $\alpha_0 = 217/816$'').
\end{proof}

\section{Proofs of the spectral estimates}\label{s.spectral}

\subsection{Preliminary results and pivotal sets}\label{ss.preli}

In this subsection, we state some preliminary results (these results illustrate the importance of Theorem~\ref{t.quenched_arm} in the study of the annealed spectral sample). We first state a result from~\cite{garban2010fourier}:
\begin{lem}[Consequence of Lemma~2.2 of~\cite{garban2010fourier}]\label{l.sample_leq_piv}
Let $E$ be a countable set and let $h \, : \, \Omega_E=\{-1,1\}^E \rightarrow \{0,1\}$ be a function that depends on finitely many coordinates. Then, for any $G \subseteq_f E$ we have
\[
\widehat{\Q}_h \left[ S \cap G \neq \emptyset \right] \leq 4 \Prob_{1/2}^E \left[ \Piv_G^E(h) \right] \text{ and } \widehat{\Q}_h \left[ \emptyset \neq S \subseteq G \right] \leq 4 \Prob_{1/2}^E \left[ \Piv^E_G(h)\right]^2 \, ,
\]
where
\[
\Piv_G^E(h) = \{ \omega_E \in \{-1,1\}^E  \, : \, \exists \omega_E' \in \{-1,1\}^E , \, (\omega_E')_{|G^c}=(\omega_E)_{|G^c} \text{ and } h(\omega'_E) \neq h(\omega_E) \} \, .
\]
\end{lem}
In order to state a consequence of Lemma~\ref{l.sample_leq_piv} for the annealed spectral sample, we need the two following definitions of pivotal events (that come from~\cite{V1}):
\begin{defi}\label{d.piv}
Let $A$ be an event measurable with respect to the colored configuration $\omega$. Also, let $D$ be a bounded Borel subset of the plane.
\bi 
\item Let $\overline{\omega} \in \Omega$ and let $\overline{\eta} \in \Omega'$ be the underlying (non-colored) point configuration. The subset $D$ is said \textbf{quenched-pivotal} for $\overline{\omega}$ and $A$ if there exists $\overline{\omega}' \in \lbrace -1,1 \rbrace^{\overline{\eta}}$ such that $\overline{\omega}$ and $\overline{\omega}'$ coincide on $\overline{\eta} \cap D^c$ and $\un_A(\overline{\omega}') \neq \un_A(\overline{\omega})$. We write $\Piv^q_D(A)$ for the event that $D$ is quenched-pivotal for $A$. Moreover, if we work under the probability measure $\Prob_{1/2}^\eta$ and if $x \in \eta$, we let $\Piv_x^q(A)$ be the event that changing the color of $x$ modifies $\un_A$.
\item We define the event $\Piv_D(A)$ that $D$ is \textbf{annealed-pivotal} for $A$ as follows:
\[
\Piv_D(A) := \{ \Pro \left[ A \, | \, \omega \setminus D \right] \in ]0,1[ \} \, .
\]
\ei
\end{defi}
Note that we have $\Pro \left[ \Piv^q_D(A) \setminus \Piv_D(A) \right]=0$ for any $A$ and $D$ as above. If $h$ is a measurable function from $\Omega$ to $\{ 0,1\}$, we write $\Piv_D(h)=\Piv_D(h^{-1}(1))$ and $\Piv_D^q(h)=\Piv_D^q(h^{-1}(1))$.\\ 

Let $h$ be a measurable function from $\Omega$ to $\{ 0,1\}$ such that a.s. $h^\eta$ depends on finitely many points of $\eta$. Remember that the annealed non-normalized spectral measure is $\widehat{\Q}^{an}_h \left[ \cdot \right] = \E \left[ \widehat{\Q}_{h^\eta} \left[ \cdot \right] \right]$. Thus, Lemma~\ref{l.sample_leq_piv} implies that, for every $D$ bounded Borel subset of the plane,
\[
\widehat{\Q}_h^{an} \left[ S \cap D \neq \emptyset \right] \leq 4\Pro_{1/2} \left[ \Piv_D^q(h) \right]\leq 4 \Pro_{1/2}\left[ \Piv_D(h) \right]
\]
and
\[
\widehat{\Q}_h^{an} \left[ \emptyset \neq S \subseteq D \right] \leq 4 \E \left[ \Prob_{1/2}^\eta \left[ \Piv_D^q(h) \right]^2 \right] \leq 4 \E \left[ \Prob_{1/2}^\eta \left[ \Piv_D(h) \right]^2 \right] \, .
\]
In the case where $h=g_n$ (which is the crossing event of $[0,n]^2$), we have proved (see respectively\footnote{The results of Lemma~D.13 of~\cite{V1} are stated in the case $h$ is the $1$-arm event $f_R$ but the proof of the analogous results in the case $h=g_n$ is exactly the same.} Lemmas~4.5 and~see Lemma D.13~\cite{V1}) that, if $B$ is a $1 \times 1$ box included in $[0,n]^2$ and at distance at least $n/3$ from the sides of this square, then
\[
\Pro_{1/2} \left[ \Piv_B^q(g_n) \right] \asymp \alpha_4^{an}(n)
\]
and
\[
\E \left[ \Prob_{1/2}^\eta \left[ \Piv_B(g_n) \right]^2 \right] \leq \grandO{1} \widetilde{\alpha}_4(n)^2 \, .
\]
Theorem~\ref{t.quenched_arm} enables to compare the quantities $\alpha_4^{an}(n)$ and  $\widetilde{\alpha}_4(n)^2$ that appear naturally in the study of the annealed spectral sample as one can see in the estimates above. This will be crucial for us, for instance in the $1^{\text{st}}$ step of the proof written in Subsection~\ref{ss.proofs_spec}.\\

Before writing the proofs of the main results on the annealed spectral sample (i.e. Theorems~\ref{t.Spec_sample_g_n},~\ref{t.Spec_sample_f_R} and~\ref{t.Spec_sample_f_R_bis}), let us state another lemma that links the spectral sample to the pivotal sets. To this purpose, we need the following definition:
\begin{defi}
Let $E$ be a countable set and $h \, : \, \Omega_E \rightarrow \{0,1\}$ a function that depends on finitely many coordinates. If $n\in \N^*$ and $J_1,\cdots,J_n$ are mutually disjoint finite subsets of $E$, we say that $J_1,\cdots,J_n$ are jointly pivotal for $h$ and some $\omega_E \in \Omega_E$ if for every $j_0 \in \{1,\cdots,n \}$ there exists $\omega'_E \in \Omega_E$ such that $\omega_E$ and $\omega_E'$ coincide outside of $\cup_{j=1}^n J_j$ and $\omega_E' \in \Piv_{J_{j_0}}^E(h)$. We write $JP_{J_1,\cdots,J_n}^E(h)$ for the event that $J_1,\cdots,J_n$ are jointly pivotal.
\end{defi}
\begin{lem}[Lemma~2.2 of~\cite{garban2010fourier} for $n=1$, Lemma~5.7 of~\cite{GV} for the general case]\label{l.jp}
Let $E$ be a countable set and $h \, : \, \Omega_E \rightarrow \{0,1\}$ a function that depends on finitely many coordinates. Also, let $n\in \N^*$ and let $J_1,\cdots,J_n$ be mutually disjoint finite subsets of $E$. If $W  \subseteq E$ satisfies $W \cap J_i=\emptyset$ for every $i \in \{1,\cdots,n\}$, then
\[
\widehat{\Q}_h \left[ \forall i \in \{1,\cdots,n\}, \, S \cap J_i \neq \emptyset = S \cap W \right] \leq 4^n \Ex_{1/2}^E \left[ \Prob_{1/2}^E \left[ JP_{J_1,\cdots,J_n}^E(h) \cond (\omega_E)_{|W^c} \right]^2 \right] \, .
\]
\end{lem}
\begin{rem}
If $n=0$ we have the following (see~(2.9) of~\cite{garban2010fourier}):
\[
\widehat{\Q}_h \left[ S \cap W = \emptyset \right] = \Ex_{1/2}^E \left[ \Ex_{1/2}^E \left[ h(\omega_E) \cond (\omega_E)_{|W^c} \right]^2 \right] \, .
\]
\end{rem}

\subsection{Proof of the main estimates on the annealed spectral sample}\label{ss.proofs_spec}

In this subsection, we write the proofs of Theorems~\ref{t.Spec_sample_g_n},~\ref{t.Spec_sample_f_R} and~\ref{t.Spec_sample_f_R_bis}. The proof of these three theorems follows the general method in three steps from~\cite{garban2010fourier} (which is also used in~\cite{GV}). \textbf{In this section, we assume that the reader has read Sections~4,~5,~6 and~7 of~\cite{garban2010fourier} and Section~5 of~\cite{GV}}.\\

We start with a $0^{\text{th}}$ step in order to deal with the spectral mass of $g_n$ and $f_R$ which is outside of $[0,n]^2$ and $[-R,R]^2$ respectively.

\paragraph{Step 0.} In this step, we prove the following estimate:
\begin{lem}\label{l.0}
We have
\[
\Pro \left[ \mathcal{S}^{an}_{g_n} \setminus [0,n]^2 \neq \emptyset \right] \leq \grandO{1} \frac{1}{n} \text{  \textup{and}  }\Pro \left[ \mathcal{S}^{an}_{f_R} \setminus [-R,R]^2 \neq \emptyset \right] \leq \grandO{1}\frac{1}{R}\, .
\]
\end{lem}
\begin{proof}
Let $(B_k)_{k \in \N}$ be an enumeration of the $1 \times 1$ boxes of the grid $\Z^2$ that are not included in $[0,n]^2$ (respectively $[-R,R]^2$). Then, by the first part of Lemma~\ref{l.sample_leq_piv},
\begin{eqnarray*}
\Pro \left[ \mathcal{S}^{an}_{g_n} \setminus [0,n]^2 \neq \emptyset \right] & \leq &  \sum_{k \in \N} \frac{\E  \left[ \widehat{\Q}_{g_n^\eta} \left[ S \cap B_k \neq \emptyset \right] \right] }{\E_{1/2}\left[ g_n \right]}\\
& \leq & \sum_{k \in \N} 4 \frac{ \Pro_{1/2} \left[ \Piv_{B_k}^q(g_n) \right]}{\Pro_{1/2} \left[ \cross(n,n) \right]}\\
& \leq & \sum_{k \in \N} 8 \Pro_{1/2} \left[ \Piv_{B_k}(g_n) \right] \, ,
\end{eqnarray*}
and similarly
\[
\Pro \left[ \mathcal{S}^{an}_{f_R} \setminus [-R,R]^2 \neq \emptyset \right] \leq \sum_{k \in \N} \frac{\Pro_{1/2} \left[ \Piv_{B_k}(f_R) \right]}{\alpha_1^{an}(R)} \, .
\]
In Section~4 of~\cite{V1}, we have proved estimates for these sums (by using the computation of the $3$-arm event in the half-plane). By following Section~4.3 of~\cite{V1}, one obtains that $\sum_{k \in \N} \Pro_{1/2} \left[ \Piv_{B_k}(g_n) \right] \leq \grandO{1} n^{-1}$ and $\sum_{k \in \N} \Pro_{1/2} \left[ \Piv_{B_k}(f_R) \right] \leq \grandO{1} \alpha_1^{an}(R)R^{-1}$.
\end{proof}

\paragraph{Step 1.} The first step is a combinatorial step where we estimate $\Pro \left[ |\mathcal{S}^{an}_h|=k \right]$ for $k$ sufficiently small (where $h$ is the crossing event $g_n$ or the $1$-arm event $f_R$). For every $S \subseteq \R^2$ and $r \in ]0,+\infty[$, we write $S(r)$ for the set of $r \times r$ boxes of the grid $r\Z^2$ that intersect $S$. We will use the three following estimates to prove Theorems~\ref{t.Spec_sample_g_n},~\ref{t.Spec_sample_f_R} and~\ref{t.Spec_sample_f_R_bis} respectively.
\begin{prop}\label{p.GPScomb1}
There exists $\theta<+\infty$ such that, for every $1 \leq r \leq n <+\infty$ and every $k \in \N^*$,
\[
\Pro \left[ |\mathcal{S}_{g_n}^{an}(r)|=k, \, \mathcal{S}_{g_n}^{an} \subseteq [0,n]^2 \right] \leq 2^{\theta \log^2(k+2)} \left( \frac{n}{r} \alpha_4^{an}(r,n) \right)^2 \, .
\]
\end{prop}
\begin{prop}\label{p.GPScomb2}
There exists $\theta<+\infty$ such that, for every $1 \leq r \leq R <+\infty$ and every $k \in \N^*$,
\[
\Pro \left[ |\mathcal{S}_{f_R}^{an}(r)|=k, \, \mathcal{S}_{f_R}^{an} \subseteq [-R,R]^2 \right] \leq  2^{\theta \log^2(k+2)} \alpha_1^{an}(r,R) \, .
\]
\end{prop}
\begin{prop}\label{p.GVcomb}
There exist $\varepsilon_1>0$ and $\theta<+\infty$ such that, for every $1 \leq r \leq r_0 \leq R/2 <+\infty$ and every $k \in \N^*$,
\[
\Pro \left[ |\mathcal{S}_{f_R}^{an}(r) \setminus [-r_0,r_0]^2|=k, \, \mathcal{S}_{f_R}^{an} \subseteq [-R,R]^2 \right] \leq 2^{\theta \log^2(k+2)} \alpha_1^{an}(r_0,R) \left( \frac{r_0}{r} \right)^{1-\varepsilon_1} \alpha_4^{an}(r,r_0) \, .
\]
\end{prop}


Let us prove these three propositions. The proofs follow~\cite{garban2010fourier} and~\cite{GV}. We first need to prove some \textbf{annulus-structures} estimates. To this purpose, let us define three different notions of ``annulus-structures''. In these definitions, the $A_i$'s are annuli of the form $A_i=A(x_i;\rho_1(i),\rho_2(i)):=x_i+[-\rho_2(i),\rho_2(i)]^2 \setminus ]-\rho_1(i),\rho_1(i)[^2$ ($\rho_1(i)$ is called the inner radius of $A_i$ and $\rho_2(i)$ is called its outer radius).
\begin{itemize}
\item An annulus structure for $g_n$ is a collection of mutually disjoint annuli $\mathcal{A}=\{A_1,\cdots,A_l\}$. An annulus $A_i$ is called \textbf{interior} if it is contained in $[0,n]^2$, \textbf{side} if it is centered at a point of $\partial [0,n]^2$ and is at distance at least its outer radius from the other sides, and \textbf{corner} if it is centered at a corner of $[0,n]^2$ and its outer radius is at most $n/2$. We assume that each $A_i$ is an annulus of one of these kinds. For each $i \in \{1,\cdots,l \}$, we let $h(A_i)$ be the annealed probability of the $4$-arm event in $A_i$ if $A_i$ is interior, of the $3$-arm event in $A_i \setminus [0,n]^2$ if it is a side annulus, and of the $2$-arm event in $A_i \setminus [0,n]^2$ if it is a corner annulus.
\item An annulus structure for $f_R$ is a collection $\mathcal{A}=\{A_1,\cdots,A_l;r_\mathcal{A}\}$ where $r_\mathcal{A} \in \R_+$ and $A_1,\cdots,A_l$ are mutually disjoint annuli such that, for each $i \in \{1,\cdots,l\}$, $A_i$ does not intersect $[-r_\mathcal{A},r_\mathcal{A}]^2$. We define interior, side and corner annuli similarly as above except that the box $[0,n]^2$ is replaced by $[-R,R]^2$ and that we ask that the inner boxes of these annuli do not contain $0$. We also need the notion of centered annuli: an annulus $A_i$ is called \textbf{centered} if it is centered at $0$ and included in $[-R,R]^2$. We assume that each $A_i$ is an annulus of one of these kinds. For each $i \in \{1,\cdots,l \}$, if $A_i$ is either interior, side or corner, we define $h(A_i)$ as above; if $A_i$ is centered, we let $h(A_i)$ be the annealed probability of the $1$-arm event in $A_i$.
\item An $r_0$-decorated annulus structure for $f_R$ is a collection of mutually disjoint annuli $\mathcal{A}=\{A_1,\cdots,A_l,A_{l+1},\cdots,A_m \}$ such that $\{ A_1,\cdots,A_l;r_0 \}$ is an annulus structure for $f_R$ and $A_{l+1}, \cdots, A_m$ are centered at a point of $\partial [-r_0,r_0]^2$ and have outer radius less than $r_0/2$. The annuli $A_{l+1},\cdots,A_m$ are called \textbf{$r_0$-annuli}. For each $i \in \{1,\cdots,l \}$, we define $h(A_i)$ as above. For each $i \in \{l+1,\cdots,m\}$, we let $h^{\varepsilon_2}(A_i)$ be the annealed probability of the $4$-arm event in $A_i$ times $(\rho_1(i)/\rho_2(i))^{\varepsilon_2}$, where $\varepsilon_2$ is the constant of Lemma~\ref{l.half}.
\item If $\mathcal{A}$ is an annulus structure of $g_n$ or $f_R$ and if $S \subseteq \R^2$, we say that $S$ is \textbf{compatible} with $\mathcal{A}$ if $S$ intersects the inner square of all the \textbf{non-centered} annuli and if $S$ does not intersect any of the annuli.
\item If $\mathcal{A}$ is an $r_0$-decorated annulus structure and if $S \subseteq \R^2$, we say that $S$ is compatible with $\mathcal{A}$ if $S$ intersects the inner square of all the \textbf{non-centered} annuli and if, for every $i \in \{1,\cdots,m\}$, $S \cap (A_i \setminus [-r_0,r_0]^2)=\emptyset$ (note that $A_i \cap [-r_0,r_0]^2 \neq \emptyset$ if and only if $A_i$ is an $r_0$-annulus).
\end{itemize}
We have the following result:
\begin{lem}\label{l.key_struct}
In the case of an annulus structure of $g_n$, we have
\begin{equation}\label{e.comb1}
\widehat{\Q}^{an}_{g_n} \left[ S \text{ is compatible with } \mathcal{A} \right] \leq \prod_{i=1}^l \grandO{1} h(A_i)^2 \, .
\end{equation}
In the case of an annulus structure of $f_R$, we have
\begin{equation}\label{e.comb2}
\widehat{\Q}^{an}_{f_R} \left[ S \text{ is compatible with } \mathcal{A} \right] \leq \grandO{1} \alpha_1^{an}(r_\mathcal{A}) \prod_{i=1}^l \grandO{1} h(A_i)^2 \, .
\end{equation}
In the case of an $r_0$-decorated annulus structure of $f_R$, we have
\begin{equation}\label{e.comb3}
\widehat{\Q}^{an}_{f_R} \left[ S \text{ is compatible with } \mathcal{A} \right] \leq \grandO{1} \alpha_1^{an}(r_0) \prod_{i=1}^l \grandO{1} h(A_i)^2 \prod_{i=l+1}^m \grandO{1} h^{\varepsilon_2}(A_i) \, .
\end{equation}
\end{lem}

In the following proof and in other parts of the present paper, we use the following definition that will help us to deal with spatial independence properties (note that $\Piv_{D_1}^{D_2}(h)$ is measurable with respect to $\omega \cap D_2$).
\begin{defi}\label{d.piv_hat}
If $D_1,D_2$ are two bounded Borel subsets of $\R^2$ and if $h$ is a measurable function from $\Omega$ to $\{0,1\}$, we let
\[
\Piv_{D_1}^{D_2}(h)= \left\lbrace \Pro_{1/2} \left[ \Piv_{D_1}(h) \cond \omega \cap D_2 \right] > 0 \right\rbrace \, .
\]
\end{defi}

\begin{proof}[Proof of Lemma \ref{l.key_struct}]
Let us first deal with~\eqref{e.comb1}. Let $W = \cup_{i=1}^l A_i$, let $B_i$ denote the inner box of $A_i$, and let $B_1',\cdots,B_{l'}'$ be all the inner boxes that do not contain any other inner box. Note that
\[
\{ S \text{ is compatible with } \mathcal{A} \} = \{ B_1' \cap S \neq \emptyset, \cdots, B_{l'}' \cap S \neq \emptyset, W \cap S = \emptyset \} \, .
\]
Remember that $\widehat{\Q}^{an}_{g_n} \left[ \cdot \right] = \E \left[ \widehat{\Q}_{g_n^\eta} \left[ \cdot \right] \right]$. Lemma~\ref{l.jp} implies that
\begin{multline*}
\widehat{\Q}^{an}_{g_n} \left[ S \text{ is compatible with } \mathcal{A} \right] \leq 4^{l'} \E \left[ \Prob_{1/2}^\eta \left[ JP^\eta_{B_1',\cdots,B_{l'}'}(g_n) \cond \omega \setminus W \right]^2 \right]\\
\leq 4^{l} \E \left[ \Prob_{1/2}^\eta \left[ JP^\eta_{B_1',\cdots,B_{l'}'}(g_n) \cond \omega \setminus W \right]^2 \right] \, .
\end{multline*}
Note furthermore that
\[
JP^\eta_{B_1',\cdots,B_{l'}'}(g_n) \subseteq \cap_{i=1}^l \Piv^{A_i}_{B_i}(g_n) \, .
\]
Since $\Piv_{B_i}^{A_i}(g_n)$ is measurable with respect to $\omega \cap A_i$ and since the $A_i$'s are mutually disjoint, spatial independence of Poisson processes imply that
\[
\widehat{\Q}^{an}_{g_n} \left[ S \text{ is compatible with } \mathcal{A} \right] \leq \E \left[ \prod_{i=1}^l  4 \Prob_{1/2}^\eta \left[ \Piv_{B_i}^{A_i}(g_n) \right]^2 \right] = \prod_{i=1}^l 4 \E \left[ \Prob_{1/2}^\eta \left[ \Piv_{B_i}^{A_i}(g_n) \right]^2 \right] \, .
\]
Let $\rho_1(i)$ and $\rho_2(i)$ be the inner and outer radii of $A_i$. In~\cite{V1} (Lemma~D.13), we have proved that, if $A_i$ is an interior annulus, then\footnote{The results of Lemma~D.13 of \cite{V1} are actually stated for the function $f_R$ instead of $g_n$, but the proof of the analogous results for $g_n$ is exactly the same.}
\[
\E \left[ \Prob_{1/2}^\eta \left[ \Piv_{B_i}^{A_i}(g_n) \right]^2 \right] \leq \grandO{1} \E \left[ \Prob_{1/2}^\eta \left[ \arm_4(\rho_1(i),\rho_2(i)) \right]^2 \right].
\]
Moreover, Theorem~\ref{t.quenched_arm} implies that
\[
\E \left[ \Prob_{1/2}^\eta \left[ \arm_4(\rho_1(i),\rho_2(i)) \right]^2 \right] \asymp h(A_i)^2 \, .
\]
The case of side and corner annuli is treated exactly in the same way. This ends the proof of~\eqref{e.comb1}.\\

Let us now prove~\eqref{e.comb2}. We still write $W = \cup_{i=1}^l A_i$, and we let $B'_1,\cdots,B'_{l'}$ be the inner boxes of the non-centered annuli such that $B'_i$ does not contain any other inner box. Note that $JP^\eta_{B_1',\cdots,B_{l'}'}(f_R)$ is included in
\[
\widehat{\arm}_1(1,r_\mathcal{A}) \cap \bigcap_{i=1}^l \Piv^{A_i}_{B_i}(f_R) \, ,
\]
where $\widehat{\arm}_1(\cdot,\cdot)$ is the event defined in Definition~\ref{d.hat}. By Proposition~\ref{p.hat},
\[
\E \left[ \Prob_{1/2}^\eta \left[ \widehat{\arm}_1(1,r_{\mathcal{A}}) \cond \omega \setminus W \right]^2 \right] = \Pro \left[ \widehat{\arm}_1(1,r_{\mathcal{A}})  \right] \leq \grandO{1} \alpha_1^{an}(r_\mathcal{A}) \, .
\]
Now, the proof is the same as the proof of~\eqref{e.comb1} so we let the details to the reader.\\

Let us now prove~\eqref{e.comb3}. In this case, we let $W = \cup_{i=1}^m A_i \setminus [-r_0,r_0]^2$. Moreover, we let $B_1',\cdots,B_{m'}'$ be the inner boxes of the non-centered annuli such that $B'_i$ does not contain any other inner box. Note that
\[
\{ S \text{ is compatible with } \mathcal{A} \} = \{ B_1'\cap S \neq \emptyset , \cdots, B_{m'}' \cap S \neq \emptyset, W \cap S = \emptyset \}
\]
and
\[
JP^\eta_{B_1',\cdots,B_{m'}'}(f_R) \subseteq \widehat{\arm}_1(1,r_0/2) \cap \bigcap_{i=1}^m \Piv^{A_i}_{B_i}(f_R)  \, .
\]
Note also that $A_i \cap [-r_0/2,r_0/2]^2 = \emptyset$ for every $i \in \{1,\cdots,m\}$. The fact that $S$ may intersect the sets $A_i \cap [-r_0,r_0]^2$ adds a new difficulty. If we follow the proof of~\eqref{e.comb2}, it only remains to deal with the annuli that intersect $[-r_0,r_0]^2$ - i.e. the $r_0$-annuli - and prove that for these annuli we have
\[
\E \left[ \Prob_{1/2}^\eta \left[ \Piv_{B_i}^{A_i}(f_R) \cond \omega \setminus W \right]^2 \right] \leq \grandO{1} \alpha_4^{an}(\rho_1(i),\rho_2(i)) \left( \frac{\rho_1(i)}{\rho_2(i)} \right)^{\varepsilon_2} \, ,
\]
where $\varepsilon_2>0$ is the constant of Lemma~\ref{l.half}. To prove this, first note that, if $A_i$ is an $r_0$-annulus, then
\[
\E \left[ \Prob_{1/2}^\eta \left[ \Piv_{B_i}^{A_i}(f_R) \cond \omega \setminus W \right]^2 \right] = \E \left[ \Prob_{1/2}^\eta \left[ \Piv_{B_i}^{A_i}(f_R) \cond \omega \cap [-r_0,r_0]^2 \right]^2 \right]
\]
since $\Piv_{B_i}^{A_i}(f_R)$ is measurable with respect to $\omega \cap A_i$. Let $H_i$ be the half plane (or one of the half planes) that contains $[-r_0,r_0]^2$, whose boundary contains the center of $A_i$, and whose boundary is parallel to the $x$ or $y$ axis. Note that we have
\[
\E \left[ \Prob_{1/2}^\eta\left[ \Piv_{B_i}^{A_i}(f_R) \cond \omega \cap [-r_0,r_0]^2 \right]^2 \right] \leq \E \left[ \Prob_{1/2}^\eta \left[ \Piv_{B_i}^{A_i}(f_R) \cond \omega \cap H_i \right]^2 \right] \, .
\]
Therefore, we can end the proof by applying Lemma~\ref{l.half}.
\end{proof}

\begin{proof}[Proof of Propositions~\ref{p.GPScomb1},~\ref{p.GPScomb2} and~\ref{p.GVcomb}]
Once Lemma~\ref{l.key_struct} is proved, the proofs of Propositions~\ref{p.GPScomb1} and~\ref{p.GPScomb2} (respectively Proposition~\ref{p.GVcomb}) are exactly the same as the proofs of Propositions~4.1 and~4.7 of~\cite{garban2010fourier} (respectively Proposition~5.3 of~\cite{GV}) if one uses the estimates on arm events from Subsection~\ref{ss.arm}. Therefore, we refer to these papers for the proofs.
\end{proof}

\paragraph{Step 2.} We first need to define a random object: $\mathcal{Z}_r$ is a random union of $1 \times 1$ boxes of the grid $\Z^2$ defined as follows: each box of this grid is included in $\mathcal{Z}_r$ with probability $(r^2 \alpha_4^{an}(r))^{-1}$, independently of the other boxes.
\begin{prop}\label{p.QM_spec1}
Let $1 \leq r \leq n < +\infty$ and let $\mathcal{S}^{an}_{g_n}$ be a spectral sample of $g_n$ independent of $\mathcal{Z}_r$. Let $B$ be a box of radius $r$ and let $B'$ be the concentric box with radius $r/3$. Assume that $B' \subseteq [0,n]^2$. Also, let $W$ be a Borel subset of the plane such that $W \cap B = \emptyset$. There exist two absolute constants $\overline{r} <+\infty$ and $a > 0$ such that, if $r \geq \overline{r}$, then
\[
\Pro \left[ \mathcal{S}_{g_n}^{an} \cap B' \cap \mathcal{Z}_r \neq \emptyset \cond \mathcal{S}_{g_n}^{an} \cap B \neq \emptyset = \mathcal{S}_{g_n}^{an} \cap W \right] \geq a \, .
\]
\end{prop}
\begin{prop}\label{p.QM_spec2}
Let $1 \leq r \leq R < +\infty$ and let $\mathcal{S}^{an}_{f_R}$ be a spectral sample of $f_R$ independent of $\mathcal{Z}_r$. Let $B$ be a box of radius $r$ and let $B'$ be the concentric box with radius $r/3$. Assume that $B' \subseteq [-R,R]^2$ and $B \cap [-4r,4r]^2 =\emptyset$. Also, let $W$ be a Borel subset of the plane such that $W \cap B = \emptyset$. There exist two absolute constants $\overline{r} <+\infty$ and $a > 0$ such that, if $r \geq  \overline{r}$, then
\[
\Pro \left[ \mathcal{S}_{f_R}^{an} \cap B' \cap \mathcal{Z}_r \neq \emptyset \cond \mathcal{S}_{f_R}^{an} \cap B \neq \emptyset = \mathcal{S}_{f_R}^{an} \cap W \right] \geq a \, .
\]
\end{prop}

\begin{rem}
With exactly the same proofs, we can also obtain the analogues of Propositions~\ref{p.QM_spec1} and~\ref{p.QM_spec2} where $B$ and $B'$ are not squares but rectangles of shape not too degenerate. More precisely, these propositions hold with the same hypotheses except that we can ask for instance that $B$ is of the form $x+[0,\rho_1] \times [0,\rho_2]$ and $B' \subseteq B$ is of the form $x'+[0,\rho_1'] \times [0,\rho_2']$, where $\frac{r}{2} \leq \rho_1,\rho_2 \leq 2r$, $\frac{r}{6} \leq \rho_1',\rho_2' \leq r$ and $B'$ is at distance at least $\frac{r}{6}$ from the sides of $B$.
\end{rem}

Propositions~\ref{p.QM_spec1} and~\ref{p.QM_spec2} follow from a second and a first moment estimates. Let us fix two boxes $B$ and $B'$ like in these propositions. In Subsection~5.7 of~\cite{garban2010fourier}, the authors explain how one can adapt the proof of their estimate for the crossing event (of Bernoulli percolation on $\Z^2$ or on the triangular lattice) in order to obtain an estimate for the one-arm event. By the same observations, one can adapt the proof of our Proposition~\ref{p.QM_spec1} in order to obtain Proposition~\ref{p.QM_spec2}, so we only write the proof of Proposition~\ref{p.QM_spec1}.

Below, $r$, $n$, $B$, $B'$ and $W$ are the quantities and sets from Proposition~\ref{p.QM_spec1}.

\paragraph{A second moment estimate.}

\begin{lem}\label{l.second_moment_spec}
Let $\square_1$ and $\square_2$ be two $1 \times 1$ squares of the grid $\Z^2$ included in $B'$. We have
\begin{multline*}
\Pro \left[ \mathcal{S}^{an}_{g_n} \cap \square_1 \neq \emptyset , \, \mathcal{S}^{an}_{g_n} \cap \square_2 \neq \emptyset , \, \mathcal{S}^{an}_{g_n} \cap W = \emptyset \right]\\
\leq \grandO{1} \alpha_4^{an}\left(\text{\textup{dist}}(\square_1,\square_2) \right) \alpha_4^{an}(r) \E \left[ \Prob_{1/2}^\eta \left[ \Piv_B(g_n) \cond \omega \setminus W \right]^2 \right] \, ,
\end{multline*}
where $\text{\textup{dist}}(\square_1,\square_2)$ is the Euclidean distance between the two squares.
\end{lem}
\begin{proof}
All the annuli considered in this proof are of the form $A(x;\rho_1,\rho_2)$. Let $A_1$ (respectively $A_2$) be an annulus co-centered with $\square_1$ (respectively $\square_2$) of inner-radius $1$ and of outer radius $\text{\textup{dist}}(\square_1,\square_2)/4$. Also, let $A_3$ be the annulus centered at a point at distance $\text{\textup{dist}}(\square_1,\square_2)/2$ from both $\square_1$ and $\square_2$, of inner radius $2\text{\textup{dist}}(\square_1,\square_2)$ and of outer radius $r/3$. Note that $JP^\eta_{\square_1,\square_2}(g_n)$ is included in $\Piv_{\square_1}^{A_1}(g_n) \cap \Piv_{\square_2}^{A_2}(g_n) \cap \Piv_{\square_3}^{A_3}(g_n) \cap \Piv_B(g_n)$, where $\square_3$ is the inner square of $A_3$. By spatial independence, by using that $W$ does not intersect $B$, and by Lemma~\ref{l.jp}, we have
\begin{align*}
& \widehat{\Q}^{an}_{g_n} \left[ S \cap \square_1 \neq \emptyset , \, S \cap \square_2 \neq \emptyset , \, S \cap W = \emptyset \right]\\
& \leq 4^2 \, \E \left[ \Prob_{1/2}^\eta \left[ \Piv_B(g_n) \cond \omega \setminus W \right]^2 \right] \prod_{i=1}^3 \E \left[ \Prob_{1/2}^\eta \left[ \Piv_{\square_i}^{A_i}(g_n) \right] \right]\\
& = 4^2 \, \E \left[ \Prob_{1/2}^\eta \left[ \Piv_B(g_n) \cond \omega \setminus W \right]^2 \right] \prod_{i=1}^3 \Pro_{1/2} \left[ \Piv_{\square_i}^{A_i}(g_n) \right] \, .
\end{align*}
By Lemma~4.9 of~\cite{V1}, we have
\[
\prod_{i=1}^3 \Pro_{1/2} \left[ \Piv_{\square_i}^{A_i}(g_n) \right] \leq  \grandO{1} \alpha_4^{an}\left( \text{\textup{dist}}(\square_1,\square_2) \right)^2 \alpha_4^{an}(\text{\textup{dist}}(\square_1,\square_2),r) \, .
\]
By the quasi-multiplicativity property, we have $ \alpha_4^{an}\left( \text{\textup{dist}}(\square_1,\square_2) \right)^2 \alpha_4^{an}(\text{\textup{dist}}(\square_1,\square_2),r) \leq \grandO{1} \alpha_4^{an}(r) \alpha_4^{an}(\text{\textup{dist}}(\square_1,\square_2))$. This ends the proof since the distribution of $\mathcal{S}^{an}_{g_n}$ is
\[
\frac{\widehat{\Q}^{an}_{g_n} \left[ \cdot \right]}{\Pro_{1/2} \left[\cross(n,n) \right]}=2 \widehat{\Q}^{an}_{g_n} \left[ \cdot \right] \, .
\]
\end{proof}

\paragraph{A first moment estimate.}

\begin{lem}\label{l.first_moment_spec}
Let $\square$ be a $1 \times 1$ square of the grid $\Z^2$ included in $B'$. We have
\[
\Pro \left[ \mathcal{S}^{an}_{g_n} \cap \square \neq \emptyset = \mathcal{S}^{an}_{g_n} \cap W \right] \geq \Omega(1) \alpha_4^{an}(r)  \E \left[ \Prob_{1/2}^\eta \left[ \Piv_B(g_n) \cond \omega \setminus W \right]^2 \right]  \, .
\]
\end{lem}

Before proving Lemma~\ref{l.first_moment_spec}, let us explain how to deduce Proposition~\ref{p.QM_spec1} from Lemmas~\ref{l.second_moment_spec} and~\ref{l.first_moment_spec}.
\begin{proof}[Proof of Proposition~\ref{p.QM_spec1}]
Remember that, if $S \subseteq \R^2$ and $r \in \R_+$, $S(r)$ denotes the set of all $r \times r$ squares of the grid $r\Z^2$ that intersect $S$. Let $Y= |\mathcal{S}^{an}_{g_n}(1) \cap B' \cap \mathcal{Z}_r | \un_{\mathcal{S}^{an}_{g_n} \cap W = \emptyset }$. Lemma~\ref{l.first_moment_spec} implies that (by definition of $\mathcal{Z}_r$ and since it is independent of the annealed spectral sample)
\[
\E \left[ Y \right] \geq \Omega(1) \E \left[ \Prob_{1/2}^\eta \left[ \Piv_B(g_n) \cond \omega \setminus W \right]^2 \right] \, .
\]
Moreover, Lemma~\ref{l.second_moment_spec} implies that (by distinguishing between diagonal and non-diagonal terms)
\begin{multline*}
\E \left[ Y^2 \right] \leq \grandO{1} \E \left[ \Prob_{1/2}^\eta \left[ \Piv_B (g_n) \cond \omega \setminus W \right]^2 \right]\\
+ \E \left[ \Prob_{1/2}^\eta \left[ \Piv_B(g_n) \cond \omega \setminus W \right]^2 \right] \frac{1}{\alpha_4^{an}(r)^2r^4} \sum \alpha_4^{an}\left( \text{\textup{dist}}(\square_1,\square_2) \right) \alpha_4^{an}(r)  \, ,
\end{multline*}
where the sum is over every pair of $1 \times 1$ squares $\square_1 \neq \square_2$ of the grid $\Z^2$ that are included in $B'$. By the quasi-multiplicativity property, we have
\begin{multline*}
\sum \alpha_4^{an}\left( \text{\textup{dist}}(\square_1,\square_2) \right) \alpha_4^{an}(r)   \leq \grandO{1} r^2 \sum_{k=0}^{\log_2(r)} 2^{2k} \alpha_4^{an}(2^k) \alpha_4^{an}(r)\\
\leq  \grandO{1} r^2 \alpha_4^{an}(r)^2 \sum_{k=0}^{\log_2(r)} 2^{2k} \alpha_4^{an}(2^k,r)^{-1} \, .
\end{multline*}
By using the fact that (for all $r' \leq r$) $\alpha_4^{an}(r',r) \geq \grandO{1} (r'/r)^{2-\Omega(1)}$, it is not difficult to see that the above is at most
\[
r^4 \alpha_4^{an}(r)^2 \, .
\]
As a result,
\[
\E \left[ Y^2 \right] \leq \grandO{1} \E \left[ \Prob_{1/2}^\eta \left[ \Piv_B(g_n) \cond \omega \setminus W \right]^2 \right] \, .
\]
Next, note that (by Lemma~\ref{l.jp} with $n=1$)
\begin{multline*}
\Pro \left[ \mathcal{S}^{an}_{g_n} \cap B \neq \emptyset = \mathcal{S}_{g_n}^{an} \cap W \right] = \frac{\E \left[\widehat{\Q}_{g_n^\eta} \left[ S \cap B \neq \emptyset = S \cap W \right] \right]}{\Pro_{1/2}\left[ \cross(n,n) \right]} \\
\hspace{5em} \leq  8 \E \left[ \Prob_{1/2}^\eta \left[ \Piv_B^q(g_n) \cond \omega \setminus W \right]^2 \right]\\
\leq 8 \E \left[ \Prob_{1/2}^\eta \left[ \Piv_B(g_n) \cond \omega \setminus W \right]^2 \right] \, .
\end{multline*}
As a result,
\[
\Pro \left[ \mathcal{S}_{g_n}^{an} \cap B' \cap \mathcal{Z}_r \neq \emptyset \cond \mathcal{S}_{g_n}^{an} \cap B \neq \emptyset = \mathcal{S}_{g_n}^{an} \cap W \right] \geq \frac{\Pro \left[ Y > 0 \right]}{8 \E \left[ \Prob_{1/2}^\eta \left[ \Piv_B(g_n) \cond \omega \setminus W \right]^2 \right]} \, .
\]
By the Cauchy-Schwarz inequality, we have
\[
\Pro \left[ Y > 0 \right] \geq \frac{\E \left[Y \right]^2}{\E \left[Y^2\right]} \geq \Omega(1) \E \left[ \Prob_{1/2}^\eta \left[ \Piv_B(g_n) \cond \omega \setminus W \right]^2 \right] \, .
\]
This ends the proof.
\end{proof}

Let us start the proof of Lemma~\ref{l.first_moment_spec}. Let $\square$ be a $1 \times 1$ square of the grid $\Z^2$ included in $B'$. When $|\eta \cap \square | = 1$, we let $x$ be the only element of this set. We have
\begin{eqnarray*}
\Pro \left[ \mathcal{S}^{an}_{g_n} \cap \square \neq \emptyset = \mathcal{S} \cap W \right] & =  & \frac{1}{\Pro_{1/2} \left[ \cross(n,n) \right] } \E \left[ \widehat{\Q}_{g_n^\eta} \left[ S \cap \square \neq \emptyset = S \cap W \right] \right]\\
& \geq &  2 \E \left[ \un_{|\square \cap \eta| =1} \widehat{\Q}_{g_n^\eta} \left[ x \in S, \, S \cap W = \emptyset \right] \right] \, .
\end{eqnarray*}
By Lemma~2.1 of~\cite{garban2010fourier}, on the event $\{ |\eta \cap \square | =1 \}$, we have
\[
\widehat{\Q}_{g_n^\eta} \left[ x \in S, \, S \cap W = \emptyset \right] = \Ex_{1/2}^\eta \left[ \Ex_{1/2}^\eta \left[ \chi^\eta_{\{x\}} g_n^\eta \cond \omega \setminus \left( W \cup \{x \} \right) \right]^2 \right]  \, .
\]
Since $g_n^\eta$ is increasing and only takes values $0$ and $1$, we have
\[
\Ex_{1/2}^\eta \left[ \chi^\eta_{\{x\}} g_n^\eta \cond \omega \setminus \{x \} \right] = \frac{1}{2} \un_{\Piv_x^q(g_n)} \, .
\]
As a result,
\begin{eqnarray*}
\Ex_{1/2}^\eta \left[ \chi^\eta_{\{x\}} g_n^\eta \cond \omega \setminus \left( W \cup \{x \} \right) \right]  & = & \Ex_{1/2}^\eta \left[ \Ex_{1/2}^\eta \left[ \chi^\eta_{\{x\}} g_n^\eta \cond \omega \setminus \{x \} \right] \cond \omega \setminus W \right]\\
& = & \frac{1}{2} \Prob^\eta_{1/2} \left[ \Piv_x^q(g_n) \cond \omega \setminus W \right] \, .
\end{eqnarray*}
As observed at the beginning of Subsection~5.3 of~\cite{garban2010fourier}, we have the following: let $\omega',\omega'' \sim \Pro_{1/2}$ be two configurations that have the same underlying point process $\eta$ and satisfy i) $\omega'_{|\eta \setminus W} = \omega''_{|\eta \setminus W}$ and ii) conditionally on $\eta$, $\omega'_{|\eta \cap W}$ is independent of $\omega''$ and $\omega''_{|\eta \cap W}$ is independent of $\omega'$. Then,
\[
\Ex^\eta_{1/2} \left[ \Prob^\eta_{1/2} \left[ \Piv_x^q(g_n) \cond \omega \setminus W \right]^2 \right] = \Pro \left[ \omega', \omega'' \in \Piv_x^q(g_n) \cond \eta \right] \, .
\]
As a result,
\[
\Pro \left[ \mathcal{S}^{an}_{g_n} \cap \square \neq \emptyset = \mathcal{S} \cap W \right] \geq \E \left[ \un_{|\eta \cap \square|=1} \Pro \left[ \omega', \omega'' \in \Piv_x^q(g_n) \cond \eta \right] \right] \, .
\]
Therefore, in order to prove Lemma~\ref{l.first_moment_spec}, it is sufficient to show that
\[
\E \left[ \un_{|\eta \cap \square|=1} \Pro \left[ \omega', \omega'' \in \Piv_x^q(g_n) \cond \eta \right] \right] \geq \Omega(1) \alpha_4^{an}(r) \E \left[\Prob^\eta_{1/2} \left[ \Piv_B(g_n) \cond \omega \setminus W \right]^2 \right] \, .
\]
If we use the observation at the beginning of Subsection~5.3 of~\cite{garban2010fourier} once again, we obtain that
\[
\E \left[ \Prob^\eta_{1/2} \left[ \Piv_B(g_n) \cond \omega \setminus W \right]^2 \right] = \Pro \left[ \omega', \omega'' \in \Piv_B(g_n) \right] \, .
\]
As a result, it is enough to prove that
\begin{equation}\label{e.findelathese}
\E \left[ \un_{|\eta \cap \square|=1} \Pro \left[ \omega', \omega'' \in \Piv_x^q(g_n) \cond \eta \right] \right] \geq \Omega(1) \alpha_4^{an}(r) \Pro \left[ \omega', \omega'' \in \Piv_B(g_n) \right] \, ,
\end{equation}
which is a quasi-multiplicativity type estimate. The proof of~\eqref{e.findelathese} (so the end of the proof of Lemma \ref{l.first_moment_spec}) is written in Subsection~\ref{ss.findelathese} by relying on quasi-multiplicativity arguments from~\cite{schramm2010quantitative},~\cite{garban2010fourier} and~\cite{V1} and by studying precisely events of the form
\[
\{ \omega', \omega'' \in \widehat{\arm}_4(r,R) \} \, ,
\]
where $\widehat{\arm}_4(r,R)$ is the event from Definition \ref{d.hat}; the study of such events will be crucial in order to overcome problems resulting from spatial dependencies.

\paragraph{Step 3.} \blue{The last step in~\cite{garban2010fourier} is the proof of a large deviation result, which is \cite[Proposition~6.1]{garban2010fourier}. We refer to Section 6 of \cite{garban2010fourier} for the statement of this result.}
\medskip

Let us end this section by explaining how to combine the Steps $0,\cdots,3$ in order to obtain Theorems~\ref{t.Spec_sample_g_n},~\ref{t.Spec_sample_f_R} and~\ref{t.Spec_sample_f_R_bis}.
\begin{proof}[Proof of Theorems~\ref{t.Spec_sample_g_n},~\ref{t.Spec_sample_f_R} and~\ref{t.Spec_sample_f_R_bis}]
Remember that, if $S \subseteq \R^2$ and $r \in \R_+$, we write\footnote{\blue{In the present proof, we only use the notation $S(r)$ with $r=1$. However, it is hidden in the reference to \cite{garban2010fourier} and \cite{GV} that we need to tile with $r \times r$ boxes and use Steps 1 and 2 with any $r$. The ``$r=1$'' from the present proof comes from the fact that, for any $r$, the random variable $\mathcal{Z}_r$ from Step 2 is a random subset of $1 \times 1$ boxes.}} $S(r)$ for the set of all $r \times r$ squares of the grid $r\Z^2$ that intersect $S$. If we use the results of Steps $1,\cdots,3$ above and if we follow Section~7 of~\cite{garban2010fourier} and Subsection~5.2.2 of~\cite{GV} (where analogues of Theorems~\ref{t.Spec_sample_g_n},~\ref{t.Spec_sample_f_R} and~\ref{t.Spec_sample_f_R_bis} are proved), we obtain that
\[
\Pro \left[ 0  < |\mathcal{S}^{an}_{g_n}(1)| < r^2\alpha_4^{an}(r), \, \mathcal{S}^{an}_{g_n} \subseteq [0,n]^2 \right] \leq \grandO{1} \left( \frac{n}{r} \alpha_4^{an}(r,n) \right)^2 \, ;
\]
\[
\Pro \left[ 0  < |\mathcal{S}^{an}_{f_R}(1)| < r^2\alpha_4^{an}(r), \, \mathcal{S}^{an}_{f_R} \subseteq [-R,R]^2 \right] \leq \grandO{1} \alpha_1^{an}(r,R) \, ;
\]
and
\begin{multline*}
\Pro \left[ 0  < |\mathcal{S}^{an}_{f_R}(1) \setminus [-r_0,r_0]^2| < r^2\alpha_4^{an}(r), \, \mathcal{S}^{an}_{f_R} \subseteq [-R,R]^2 \right]\\
\leq \grandO{1} \alpha_1^{an}(r_0,R) \left( \frac{r_0}{r} \right)^{1-\varepsilon_1} \alpha_4^{an}(r,r_0) \, .
\end{multline*}
Note that (for any $h$) $|\mathcal{S}_h^{an}|\geq |\mathcal{S}_h^{an}(1)|$, so these three results also hold with $|\mathcal{S}_h^{an}|$ instead of $|\mathcal{S}_h^{an}(1)|$. As a result, it only remains to deal with the events $\{ \mathcal{S}_{g_n}^{an} \nsubseteq [0,n]^2 \}$ and $\{ \mathcal{S}_{f_R}^{an} \nsubseteq [-R,R]^2 \}$. Lemma~\ref{l.0} implies that
\[
\Pro \left[ \mathcal{S}^{an}_{g_n} \nsubseteq [0,n]^2 \right] \leq \grandO{1}/n \, ,
\]
which ends the proof of Theorem~\ref{t.Spec_sample_g_n} (i.e. the result concerning $g_n$). Concerning the function $f_R$, Lemma~\ref{l.0} implies that
\[
\Pro \left[ \mathcal{S}^{an}_{f_R} \nsubseteq [-R,R]^2 \right] \leq  \grandO{1} \frac{1}{R} \, .
\]
Therefore, in order to prove Theorems~\ref{t.Spec_sample_f_R} and~\ref{t.Spec_sample_f_R_bis} it only remains to prove that, if $\varepsilon_0>0$ is sufficiently small, then $1/R$ is
\bi 
\item[i)] less than $\grandO{1} \alpha_1^{an}(r,R)\,$;
\item[ii)] less than $\grandO{1} \alpha_1^{an}(r_0,R) \left( \frac{r_0}{r} \right)^{1-\varepsilon_0} \alpha_4^{an}(r,r_0)\,$.
\ei
Item~i) comes for instance from the fact that the probability of the $1$-arm event is greater than the probability of the $2$-arm event in the half-plane which has exponent $1$ (see Proposition~\ref{p.univ}). Let us end the proof by showing Item~ii). By Proposition~\ref{p.4}, there exists $\varepsilon > 0$ such that
\[
\alpha_4^{an}(r,r_0) \geq \varepsilon \left( \frac{r}{r_0} \right)^{2-\varepsilon} \, .
\]
We assume that $\varepsilon_0 \leq \varepsilon$. We obtain that
\[
\alpha_1^{an}(r_0,R) \left( \frac{r_0}{r} \right)^{1-\varepsilon_0} \alpha_4^{an}(r,r_0) \geq \varepsilon \alpha_1^{an}(r_0,R) \frac{r}{r_0}\geq \Omega(1) \frac{r_0}{R}\cdot \frac{r}{r_0} \geq \Omega(1) \frac{1}{R} \, .
\]
This ends the proof of Item~ii).
\end{proof}

\subsection{Quasi-multiplicativity properties for the annealed spectral sample}\label{ss.findelathese}

In this subsection, we prove~\eqref{e.findelathese}, which implies Lemma~\ref{l.first_moment_spec}. We refer to Step~2 of Subsection~\ref{ss.proofs_spec} for the notations used in the present subsection. The proof of~\eqref{e.findelathese} relies on quasi-multiplicativity estimates. \textbf{In this section, we assume that the reader has read Section~5 of~\cite{garban2010fourier}} (where the authors prove the analogue of~\eqref{e.findelathese} for Bernoulli percolation) \textbf{and Section~7.1 of~\cite{V1}} (where the quasi-multiplicativity property is proved for Voronoi percolation). Let us first define some events introduced in~\cite{V1} (as an important step of the proof, we will show that, if one conditions on the fact that $\omega$ and $\omega'$ satisfy a $j$-arm event between scales $r$ and $R$, then the events defined in the three definitions below are typically satisfied at scales $r$ and $R$; remember that $\omega$ and $\omega'$ are defined in Subsection~\ref{ss.proofs_spec}; we will recall their definition below).
\begin{defi}\label{d.dense}
If $\delta \in ]0,1[$ and $D$ is a bounded Borel subset of the plane, we write $\dense_\delta(D)$ for the event that, for every $u \in D$, there exists $x \in \eta \cap D$ such that $||x-u||_2 \leq \delta \cdot \text{diam}(D)$.
\end{defi}
\begin{defi}\label{d.qbc}
In this definition, we introduce the event $\qbc_\delta^\gamma(D)$. This event is roughly the event that the quenched crossing probabilities are not degenerate for all quads $Q \subseteq D$ satisfying: i) the diameter of $Q$ is larger than $\delta \cdot \text{diam}(D)$ and ii) two opposite sides of $Q$ are at distance larger than $\delta \cdot \text{diam}(D)$.
\begin{itemize}
\item A \textbf{quad} is a topological rectangle, i.e. a subset of the plane homeomorphic to a closed disc with two disjoint non empty segments on its boundary. If $Q$ is a quad, the event $\cross(Q)$ is the event that there is a black path included in $Q$ that connects one segment to the other.
\item Let $D$ be a subset of the plane and let $\delta \in ]0,1[$. We denote by $\mathcal{Q}'_\delta(D)$ the set of all quads $Q \subseteq D$ which are drawn on the grid $(\delta \, \diam(D)) \cdot \Z^2$ (i.e. whose sides are included in the edges of $(\delta \, \diam(D)) \cdot \Z^2$ and whose corners are vertices of $(\delta \, \diam(D)) \cdot \Z^2$). Also, we denote by $\mathcal{Q}_\delta(D)$ the set of all quads $Q \subseteq D$ such that there exists a quad $Q' \in \mathcal{Q}'_\delta(D)$ satisfying $\cross(Q') \subseteq \cross(Q)$.
\item By~\cite{V1} (Proposition~2.13), there exists a constant $C<+\infty$ such that the following holds: For every $\delta \in ]0,1[$ and every $\gamma \in ]0,+\infty[$, there exists $c=c(\delta,\gamma) \in ]0,1[$ such that, for every $D$ Borel subset of the plane that satisfies $\text{diam}(D) \geq \delta^{-2}/100$, we have
\[
\Pro \left[ \forall Q \in \mathcal{Q}_\delta(D) , \, \Prob_{1/2}^\eta \left[ \cross(Q) \right] \geq c \right] \geq 1 - C\text{diam}(D)^{-\gamma} \, .
\]
We define the event $\qbc_\delta^\gamma(D)$ (for ``Quenched Box Crossings'') as follows:
\[
\qbc_\delta^\gamma(D) = \{ \forall Q \in \mathcal{Q}_\delta(D) , \, \Prob_{1/2}^\eta \left[ \cross(Q) \right] \geq c(\delta,\gamma)  \} \, .
\]
\end{itemize}
\end{defi}
Let us now define events that the interfaces at scale $r$ or $R$ are well separated.
\begin{defi}\label{d.gi}
Let $\delta \in ]0,1/100[$ and $1 \leq r \leq R < +\infty$. Also, let $\beta_1,\cdots,\beta_k$ denote the interfaces between black and white from $\partial [-r,r]^2$ to $\partial [-R,R]^2$ (which are drawn on the edges of the Voronoi cells) and let $z_i^{ext}$ (respectively $z_i^{int}$) denote the end point of $\beta_i$ on $\partial [-R,R]^2$ (respectively on $\partial [-r,r]^2$). We write $s^{ext}(r,R)$ for the infimum of $||z_i^{ext}-z_j^{ext}||_2$ (for $i\neq j$) and $s^{int}(r,R)$ for the infimum of $||z_i^{int}-z_j^{int}||_2$ (for $i \neq j$). We define the two following events (where $\gi$ means ``Good Interfaces''):
\[
\gi^{ext}_\delta(R)= \left\lbrace s^{ext}(3R/4,R) \geq 10\delta R \right\rbrace 
\]
and
\[
\gi^{int}_\delta(r)= \left\lbrace s^{int}(r,3r/2) \geq 10\delta r \right\rbrace \, .
\]
(In~\cite{V1}, these events are denoted by $\widetilde{\gi}^{ext}_\delta(R)$ and $\widetilde{\gi}^{int}_\delta(r)$.)
\end{defi}
\bigskip

We need several lemmas to prove~\eqref{e.findelathese}.
In these lemmas, we use the following notations:
\bi
\item[(a)] $W$ is a Borel subset of the plane,
\item[(b)] $\omega',\omega'' \sim \Pro_{1/2}$ are two configurations that have the same underlying point process $\eta$ and satisfy i) $\omega'_{|\eta \setminus W} = \omega''_{|\eta \setminus W}$ and ii) conditionally on $\eta$, $\omega'_{|\eta \cap W}$ is independent of $\omega''$ and $\omega''_{|\eta \cap W}$ is independent of $\omega'$.
\ei
Moreover, for any $1 \leq r \leq R < +\infty$, we write
\begin{equation}
\beta^{an,W}_4(r,R) = \Pro \left[ \omega',\omega'' \in \arm_4(r,R) \right]
\end{equation}
and
\begin{equation}\label{e.lanotationbetahat}
\widehat{\beta}^{an,W}_4(r,R) = \Pro \left[ \omega',\omega'' \in \widehat{\arm}_4(r,R) \right]  \, ,
\end{equation}
where the event $\widehat{\arm}_4(r,R)$ is the event from Definition~\ref{d.hat} (remember that $\widehat{\arm}_4(r,R)$ is measurable with respect to the configuration restricted to the annulus $B(0,R) \setminus B(0,r)$). We also define the following events analogous to $\widehat{\arm}_4(r,R)$:
\[
\widehat{\arm}_4^{ext}(r,R) = \left\lbrace \Pro_{1/2} \left[ \arm_4(r,R) \cond \omega \cap B(0,R) \right] > 0 \right\rbrace \, ,
\]
\[
\widehat{\arm}^{int}_4(r,R) = \left\lbrace \Pro_{1/2} \left[ \arm_4(r,R) \cond \omega \setminus B(0,r) \right] > 0 \right\rbrace 
\]
and we let
\[
\widehat{\beta}^{an,ext,W}_4(r,R) = \Pro \left[ \omega',\omega'' \in \widehat{\arm}^{ext}_4(r,R) \right]  \, ,\; \widehat{\beta}^{an,int,W}_4(r,R) = \Pro \left[ \omega',\omega'' \in \widehat{\arm}^{int}_4(r,R) \right]  \, .
\]
We refer to the proof of Lemma \ref{l.fin3} for an illustration of how we will combine estimates on $\widehat{\beta}^{an,ext,W}_4(r,R)$ and $\widehat{\beta}^{an,int,W}_4(r,R)$ to obtain estimates on the quantity $\widehat{\beta}^{an,W}_4(r,R)$.

Note that we have
\begin{multline*}
\beta^{an,W}_4(r,R) \leq \beta_4^{an,ext,W}(r,R)  \leq \widehat{\beta}^{an,W}_4(r,R) \text{ and similarly }\\ \beta^{an,W}_4(r,R) \leq  \beta_4^{an,int,W}(r,R) \leq \widehat{\beta}^{an,W}_4(r,R) \, .
\end{multline*}
A very important step in the proof of the quasi-multiplicativity type result~\eqref{e.findelathese} is that the above inequalities are actually equalities up to constants. Thanks to this - which roughly implies that the event $\{ \omega',\omega'' \in \arm_4(r,R) \}$ depends essentially on the configuration in $B(0,R) \setminus B(0,r)$ - we will overcome most of the problems resulting from spatial dependencies.

In order to prove \eqref{e.findelathese}, we first need to show four lemmas. Before stating them, let us say that, in his or her reading of Lemmas \ref{l.fin0} to \ref{l.fin4}, the reader can keep in mind that the goal of the present subsection is to prove a quasi-multiplicativity property of the kind
\[
\forall 1 \leq r_1 \leq r_2 \leq r_3 <+\infty, \, \widehat{\beta}^{an,W}_4(r_1,r_3) \asymp \beta^{an,W}_4(r_1,r_3) \asymp \beta^{an,W}_4(r_1,r_2)\beta^{an,W}_4(r_2,r_3) \, .
\]
We will not write exactly the proof of such an inequality since we are only interested in proving \eqref{e.findelathese}, but it is easier to follow the general ideas by keeping this in mind rather than \eqref{e.findelathese}.

We begin with the following lemma.
\begin{lem}\label{l.fin0}
Let $E$ and $F$ be two Borel subsets of the plane and consider two configurations $\widetilde{\omega}',\widetilde{\omega}'' \sim \Pro_{1/2}$ such that: i) $\widetilde{\omega}'=\widetilde{\omega}''$ outside of $E \cup F$, ii) $\widetilde{\omega}'$ and $\widetilde{\omega}''$ are independent of each other in $E$ and iii) in $F \setminus E$, $\widetilde{\omega}'$ and $\widetilde{\omega}''$ have the same underlying point process $\eta$ but are independent conditionally on this point process. Then, for every event $A$, the quantity
\[
\Pro \left[ \widetilde{\omega}',\widetilde{\omega}'' \in A \right]
\]
is non-increasing in $E$ and $F$.
\end{lem}
\begin{proof}
This comes from the fact that
\[
\Pro \left[ \widetilde{\omega}',\widetilde{\omega}'' \in A \right]=\E \left[ \Pro_{1/2} \left[ A \cond \eta \setminus E, \, \omega \setminus (E \cup F) \right]^2 \right]
\]
and that the $\sigma$-algebra $\sigma \left(\eta \setminus E, \omega \setminus (E \cup F) \right)$ is non-increasing in $E$ and $F$.
\end{proof}

The following lemma is the analogue of Lemma~5.7 of~\cite{garban2010fourier}. 
\begin{lem}\label{l.fin1}
For any $\varepsilon > 0$, there exist $\overline{r}=\overline{r}(\varepsilon) < +\infty$, $c=c(\varepsilon)>0$ and an absolute constant $d \in ]0,+\infty[$ (in particular, $\overline{r}$, $c$ and $d$ are independent of $W$) such that the following holds:
\begin{itemize}
\item For any $r_0 \geq 1$ and $r \geq r_0 \vee \overline{r}$ that satisfy
\begin{equation}\label{e.fin11}
\widehat{\beta}^{an,ext,W}_4(r_0,4r) \geq \varepsilon \widehat{\beta}^{an,ext,W}_4(r_0,r)  \, ,
\end{equation}
we have
\begin{equation}\label{e.concl1}
\forall r' \geq r, \, \beta^{an,W}_4(r_0,r') \geq c \widehat{\beta}^{an,ext,W}_4(r_0,r) \left( \frac{r}{r'} \right)^d \,.
\end{equation}
\item For any $r_0 \geq \overline{r}$ and any $r \in [\overline{r},r_0]$ that satisfy
\begin{equation}\label{e.fin12}
\widehat{\beta}^{an,int,W}_4(r/4,r_0) \geq \varepsilon \widehat{\beta}^{an,int,W}_4(r,r_0)  \, ,
\end{equation}
we have
\begin{equation}\label{e.concl2}
\forall r' \in [\overline{r},r] , \, \beta^{an,W}_4(r',r_0) \geq c \widehat{\beta}^{an,int,W}_4(r,r_0) \left( \frac{r'}{r} \right)^d \,.
\end{equation}
\end{itemize} 
\end{lem}
It seems much harder to prove the analogue of Lemma \ref{l.fin1} with all the $\widehat{\beta}^{an,ext,W}_4$'s and $\widehat{\beta}^{an,int,W}_4$'s replaced by $\beta^{an,W}_4$ (see for instance the footnote \ref{footnote}).
\begin{proof}[Proof of Lemma \ref{l.fin1}]
We only prove the first item since the proof of the second one is the same. Let $\varepsilon>0$, $1 \leq r_0 \leq r <+\infty$, and assume that \eqref{e.fin11} holds. Also, let
\[
X' = \Pro \left[ \omega' \in \widehat{\arm}^{ext}_4(r_0,4r) \cond \omega' \cap B(0,r) \right]
\]
and
\[
X'' = \Pro \left[ \omega'' \in \widehat{\arm}^{ext}_4(r_0,4r) \cond \omega'' \cap B(0,r) \right] \, .
\]
We have
\begin{multline}\label{e.avant14}
\Pro \left[  \omega',\omega'' \in \widehat{\arm}^{ext}_4(r_0,4r) \cond \omega' \cap B(0,r), \, \omega'' \cap B(0,r)  \right]\\
\leq \Pro \left[  \omega' \in \widehat{\arm}^{ext}_4(r_0,4r) \cond \omega' \cap B(0,r), \, \omega'' \cap B(0,r)  \right]\\
\hspace{6em}\wedge \Pro \left[  \omega'' \in \widehat{\arm}^{ext}_4(r_0,4r) \cond \omega' \cap B(0,r), \, \omega'' \cap B(0,r)  \right] \\
 = X' \wedge X'' =:\widetilde{X}  \, .
\end{multline}
As a result,
\[
\E \left[ \widetilde{X} \right] \geq \widehat{\beta}^{an,ext,W}_4(r_0,4r) \geq \varepsilon \widehat{\beta}^{an,ext,W}_4(r_0,r) \hspace{1em}\text{   (by \eqref{e.fin11}).}
\]
Now, note that\footnote{\label{footnote}Here, it is important to consider the event $\widehat{\arm}^{ext}_4(r_0,4r)$ and not the event $\arm_4(r_0,4r)$.} $\{ \widetilde{X} > 0 \} \subseteq \{ \omega',\omega'' \in \widehat{\arm}^{ext}_4(r_0,r) \}$. As a result,
\begin{equation}\label{e.14}
\E \left[ \widetilde{X} \cond \omega',\omega'' \in \widehat{\arm}_4^{ext}(r_0,r) \right] \geq \varepsilon \, .
\end{equation}
Now, as in Section~7.1 of~\cite{V1}, we define the following event that will help us to extend the arms to a larger scale (see Definitions~\ref{d.dense},~\ref{d.qbc} and~\ref{d.gi} for the notations; remember that, if $1 \leq \rho_1 \leq \rho_2$, we let $A(\rho_1,\rho_2) = [-\rho_2,\rho_2]^2 \setminus ]-\rho_1,\rho_1[^2$):
\begin{multline}\label{e.Gext}
G^{ext}_\delta(\rho) = \gi^{ext}_\delta(\rho) \cap \dense_\delta(A(\rho/2,2\rho)) \cap \qbc_\delta^1(A(3\rho/4,3\rho/2)) \\
\cap \left\lbrace \Pro \left[ \qbc^1_{1/100}(A(\rho,4\rho)) \cap \dense_{1/100}(A(\rho,4\rho))  \cond \eta \cap A(\rho/2,2\rho) \right] \geq 3/4 \right\rbrace
\end{multline}
(the last event in the intersection is defined this way so that the event $G^{ext}_\delta(\rho)$ is measurable with respect to $\omega \cap A(\rho/2,2\rho)$). For some technical reasons we will also need the following event where we rather control the interfaces at scale $9\rho/10$:
\begin{multline*}
\widetilde{G}^{ext}_\delta(\rho) = \gi^{ext}_\delta(9\rho/10) \cap \dense_\delta(A(\rho/2,2\rho)) \cap \qbc_\delta^1(A(3\rho/4,3\rho/2)) \\
\cap \left\lbrace \Pro \left[ \qbc^1_{1/100}(A(\rho,4\rho)) \cap \dense_{1/100}(A(\rho,4\rho))  \cond \eta \cap A(\rho/2,2\rho) \right] \geq 3/4 \right\rbrace \, .
\end{multline*}
By Lemma~7.4 of~\cite{V1}, there exists $a > 0$ such that, for every $\delta  \in ]0,1/1000[$, there exists $\overline{R}=\overline{R}(\delta)$ such that, if $\rho \geq \overline{R}(\delta)$, then
\begin{equation}\label{e.cacontinue3}
\Pro_{1/2} \left[ G^{ext}_\delta(\rho) \right], \, \Pro \left[ \widetilde{G}^{ext}_\delta(\rho) \right] \geq 1 - \frac{1}{a} \delta^a \, .
\end{equation}

By Lemma~7.6 of~\cite{V1}, there exists an absolute constant $\delta_0 \in ]0,1/1000[$ such that, for every $\delta \in ]0,1/1000[$, there exist two constants $c'=c'(\delta)>0$ and $\overline{r}=\overline{r}(\delta)<+\infty$ such that, for every $r_1 \geq \overline{r} \vee 4r_0$, we have
\[
\Pro_{1/2} \left[ \arm_4(r_0,4r_1) \cap G^{ext}_{\delta_0}(4r_1) \right] \geq c' \Pro_{1/2} \left[ \arm_4(r_0,r_1) \cap G^{ext}_{\delta}(r_1) \right] \, .
\]
With exactly the same proof, if we assume furthermore that $r_1 \geq 4r$, we even have the following point-wise inequality:
\begin{multline}\label{e.cacontinue1}
\Pro_{1/2} \left[ \arm_4(r_0,4r_1) \cap G^{ext}_{\delta_0}(4r_1) \cond \omega \cap B(0,r) \right]\\
\geq c' \Pro \left[ \arm_4(r_0,r_1) \cap G^{ext}_{\delta}(r_1) \cond \omega \cap B(0,r) \right]  \, .
\end{multline}
Now, note that $\widehat{\arm}^{ext}_4(r_0,r_1) \cap \dense_\delta(A(r_1/2,2r_1)) \subseteq \arm_4(r_0,9r_1/10)$ and remember that we have $\dense_\delta(A(r_1/2,2r_1)) \subseteq \widetilde{G}^{ext}_\delta(r_1)$. As a result, still with exactly the same proof as Lemma~7.6 of~\cite{V1} (and maybe by decreasing $\delta_0$ and $c'$ and increasing $\overline{r}$), we obtain that furthermore, for every $r_1 \geq 4r \vee \overline{r}$,\footnote{Note that it seems much harder to prove the analogue of \eqref{e.cacontinue2} with  $\widetilde{G}^{ext}_\delta(r_1)$ replaced by $G^{ext}_\delta(r_1)$. Indeed, $\widehat{\arm}^{ext}_4(r_0,r_1) \cap \widetilde{G}^{ext}_\delta(r_1)$ implies that $\arm_4(r_0,9r_1/5)$ holds but not that $\arm_4(r_0,r_1)$ does, so having information about the well-separateness of interfaces that end on $\partial B(0,r_1)$ does not help us, but having information about those that end on $\partial B(0,9r_1/10)$ does.}
\begin{multline}\label{e.cacontinue2}
\Pro_{1/2} \left[ \arm_4(r_0,4r_1) , \, G^{ext}_{\delta_0}(4r_1) \cond \omega \cap B(0,r) \right]\\
\geq c' \Pro \left[ \widehat{\arm}^{ext}_4(r_0,r_1), \, \widetilde{G}^{ext}_\delta(r_1) \cond \omega \cap B(0,r) \right] \, .
\end{multline}

As in~\cite{garban2010fourier}, we let $\widetilde{\omega}'$ and $\widetilde{\omega}''$ be two configurations with distribution $\Pro_{1/2}$ that satisfy i) $\widetilde{\omega}' \cap B(0,r) = \omega' \cap B(0,r)$ and $\widetilde{\omega}'' \cap B(0,r) = \omega'' \cap B(0,r)$; ii) $\widetilde{\omega}'$ is independent of $\widetilde{\omega}''$ outside of $B(0,r)$. In particular, the underlying point processes $\widetilde{\eta}'$ and $\widetilde{\eta}''$ (of $\widetilde{\omega}'$ and $\widetilde{\omega}''$) are independent of each other outside of $B(0,r)$.
\medskip

By \eqref{e.cacontinue3}, \eqref{e.cacontinue1} and \eqref{e.cacontinue2} respectively, we can fix four constants $\delta_0 \in ]0,1/1000[$, $c'>0$, $\delta \in ]0,1/1000[$ and $\overline{r}<+\infty$ such that the three following estimates hold as soon as $r \geq \overline{r}$ and $r_1 \geq 4r$:
\begin{equation}\label{e.15}
\Pro_{1/2} \left[ \widetilde{G}^{ext}_\delta(4r) \right] \geq 1 - \frac{\varepsilon}{2} \, ;
\end{equation}
\begin{multline}\label{e.17}
\Pro_{1/2} \left[ \widetilde{\omega}',\widetilde{\omega}'' \in \arm_4(r_0,4r_1) \cap G^{ext}_{\delta_0}(4r_1) \right]\\
\geq (c')^2 \Pro_{1/2} \left[ \widetilde{\omega}',\widetilde{\omega}'' \in \arm_4(r_0,r_1) \cap G^{ext}_{\delta_0}(r_1) \right] \, ;
\end{multline}
\begin{multline}\label{e.16}
\Pro_{1/2} \left[ \arm_4(r_0,16r) \cap G^{ext}_{\delta_0}(16r) \cond \ \omega \cap B(0,r) \right]\\
\geq c' \Pro \left[ \widehat{\arm}^{ext}_4(r_0,4r) \cap \widetilde{G}^{ext}_{\delta}(4r) \cond \omega \cap B(0,r) \right] \, .
\end{multline}

Let us now combine all these estimate to prove the lemma. Note that
\begin{multline*}
\Pro \left[ \widetilde{\omega}' \in \widehat{\arm}^{ext}_4(r_0,4r) \cap \widetilde{G}^{ext}_\delta(4r) \cond \omega' \cap B(0,r) \right]\\
\geq \Pro \left[  \widetilde{\omega}' \in \widehat{\arm}^{ext}_4(r_0,4r) \cond \omega' \cap B(0,r)  \right]\\
\hspace{10em}- \Pro \left[ \widetilde{\omega}' \notin \widetilde{G}^{ext}_\delta(4r) \cond \omega' \cap B(0,r) \right]
\\
\geq \widetilde{X} - \frac{\varepsilon}{2} \text{ (by }\eqref{e.15}).
\end{multline*}
As a result,~\eqref{e.16} implies that
\[
\Pro \left[ \widetilde{\omega}' \in \arm_4(r_0,16r) \cap G^{ext}_{\delta_0}(16r) \cond \omega' \cap B(0,r) \right]\\
\geq c'\left(\widetilde{X} - \frac{\varepsilon}{2} \right)_+ \, .
\]
Since $\widetilde{\omega}'$ and $\widetilde{\omega}''$ are independent conditionally on $\omega' \cap B(0,r)$ and $\omega'' \cap B(0,r)$, the above implies that
\[
\Pro \left[ \widetilde{\omega}', \widetilde{\omega}'' \in \arm_4(r_0,16r) \cap G^{ext}_{\delta_0}(16r) \cond \omega' \cap B(0,r), \, \omega'' \cap B(0,r) \right]\\
\geq \left(c'\left(\widetilde{X} - \frac{\varepsilon}{2} \right)_+\right)^2 \, .
\]
This implies that
\begin{align}
& \Pro \left[ \widetilde{\omega}', \widetilde{\omega}'' \in \arm_4(r_0,16r) \cap G^{ext}_{\delta_0}(16r)  \right] \nonumber\\
& = \Pro \left[ \widetilde{\omega}', \widetilde{\omega}'' \in \arm_4(r_0,16r) \cap G^{ext}_{\delta_0}(16r) , \, \omega',\omega'' \in \widehat{\arm}^{ext}_4(r_0,r) \right]\nonumber\\
& \geq \E \left[ \un_{\omega',\omega'' \in \widehat{\arm}^{ext}_4(r_0,r)} \left( c' \left(\widetilde{X}-\frac{\varepsilon}{2} \right)_+ \right)^2 \right]\nonumber\\
& \geq \Pro \left[ \omega',\omega'' \in \widehat{\arm}^{ext}_4(r_0,r) \right] \frac{(c' \varepsilon)^2}{4} \text{ (by }\eqref{e.14} \text{ and Jensen's inequality).} \label{e.oui}
\end{align}
By the above and~\eqref{e.17},
\[
\Pro \left[ \widetilde{\omega}', \widetilde{\omega}'' \in \arm_4(r_0,r')   \right] \geq \frac{(c'\varepsilon)^2}{4} (c')^{2\lceil \log_4(r'/16r) \rceil} \Pro \left[ \omega',\omega'' \in \widehat{\arm}^{ext}_4(r_0,r) \right] \, .
\]
This ends the proof since $\Pro \left[ \omega',\omega'' \in \widehat{\arm}^{ext}_4(r_0,r) \right] = \widehat{\beta}^{an,ext,W}_4(r_0,r)$ and since, by Lemma~\ref{l.fin0},
\begin{eqnarray*}
\Pro \left[ \widetilde{\omega}', \widetilde{\omega}'' \in \arm_4(r_0,r')  \right] & \leq & \beta^{an}_4(r_0,r') \, .
\end{eqnarray*}
\end{proof}

\begin{rem}\label{r.lextensionavecwidetilde}
In this remark, we use the notations from the proof and we note that \eqref{e.cacontinue1} implies that, if $\delta \in ]0,1/1000[$, $1 \leq r_0 \leq r < +\infty$, and $r_1 \geq \overline{r}(\delta) \vee 4r$, then
\begin{multline*}
\Pro \left[ \widetilde{\omega}',\widetilde{\omega}'' \in \arm_4(r_0,4r_1) \cond \widetilde{\omega}' \cap B(0,r),\widetilde{\omega}'' \cap B(0,r)  \right]\\
\geq (c'(\delta))^2 \Pro \left[ \widetilde{\omega}',\widetilde{\omega}'' \in \arm_4(r_0,r_1) \cap G^{ext}_{\delta}(r_1) \cond \widetilde{\omega}' \cap B(0,r), \widetilde{\omega}'' \cap B(0,r) \right] \, .
\end{multline*}
This comes from the fact that $\widetilde{\omega}'$ and $\widetilde{\omega}''$ are independent conditionally on $\omega' \cap B(0,r)$ and $\omega'' \cap B(0,r)$ and that $\widetilde{\omega}'$ and $\widetilde{\omega}''$ are equal to $\omega'$ and $\omega''$ respectively in $B(0,r)$. We will use analogous estimates at the end of the present subsection.
\end{rem}
Let us now prove that the hypotheses of Lemma \ref{l.fin1} hold for some $\varepsilon>0$.
\begin{lem}\label{l.fin2}
There exist two absolute constants $\overline{\varepsilon}>0$ and $\overline{R}<+\infty$ such that:
\begin{itemize}
\item For every $r_0 \geq \overline{R}$ and every $\rho \geq r_0$, we have\footnote{Note that in the analogous result Lemma 5.8 of \cite{garban2010fourier}, the hypotheses are $r_0 > 0$ and $\rho \geq r_0 \vee \overline{R}$. However, we don't need this stronger result and the hypothesis $r_0 \geq \overline{R}$ simplifies the proof. More precisely, thanks to this hypothesis, we don't need to extend the arms to scale $1$, which is very technical for Voronoi percolation (see \cite{V1}).}
\begin{equation}\label{e.fin2}
\widehat{\beta}^{an,ext,W}_4(r_0,4\rho) \geq \overline{\varepsilon} \widehat{\beta}^{an,ext,W}_4(r_0,\rho)  \, .
\end{equation}
\item For every $r_0 \geq \overline{R}$ and every $\rho \in [\overline{R},r_0]$, we have
\begin{equation}
\widehat{\beta}^{an,int,W}_4(\rho/4,r_0) \geq \overline{\varepsilon} \widehat{\beta}^{an,int,W}_4(\rho,r_0)  \, .
\end{equation}
\end{itemize}
\end{lem}
\begin{proof}
We only prove the first item since the proof of the second one is the same. First note that by Lemma~\ref{l.fin1} we have the following for any $\varepsilon >0$: If $r_0\geq 1$, if $r \geq r_0 \vee \overline{r}$ (where $\overline{r}=\overline{r}(\varepsilon)$ is the constant from Lemma~\ref{l.fin1}), and if
\begin{equation}\label{e.debut_rec}
\widehat{\beta}^{an,ext,W}_4(r_0,4r) \geq \varepsilon \widehat{\beta}^{an,ext,W}_4(r_0,r) \, ,
\end{equation}
then for any $r' \geq r$ we have
\[
\widehat{\beta}^{an,ext,W}_4(r_0,r') \geq \beta^{an,W}_4(r_0,r') \geq c \left( \frac{r}{r'} \right)^d  \widehat{\beta}^{an,ext,W}_4(r_0,r)
\]
(where $c=c(\varepsilon)$ and $d$ are the constants from Lemma~\ref{l.fin1}). As explained in the beginning of the proof of Lemma~5.8 of~\cite{garban2010fourier}, one can deduce from this that there exists $h > 0$ such that, if~\eqref{e.debut_rec} holds for some $\varepsilon \in ]0,h[$, some $r_0 \geq 1$ and some $r \geq r_0 \vee \overline{r}(\varepsilon)$, then
\[
\widehat{\beta}^{an,ext,W}_4(r_0,4\rho) \geq \overline{\varepsilon} \widehat{\beta}^{an,ext,W}_4(r_0,\rho) 
\]
holds for some $\overline{\varepsilon} = \overline{\varepsilon}(\varepsilon)$ and for every $\rho \geq r$.

As a result, it is sufficient to prove that there exists $\varepsilon > 0$ such that, for every $r_0 \geq 1$, there exists $r \geq r_0$ satisfying both i) \eqref{e.fin2} for any $\rho \in [r_0,r]$ and ii) \eqref{e.debut_rec}. This is actually a direct consequence of the fact that (by~\eqref{e.poly} and Lemma \ref{l.fin0}) there exists an absolute constant $c' \in ]0,1[$ such that, for any $r_0 \geq 1$ and any $r_1 \in [r_0,4r_0]$,
\[
c' \leq \alpha_4^{an}(r_0,r_1)^2 \leq \widehat{\beta}^{an,ext,W}_4(r_0,r_1) \leq 1 \, . 
\]
This ends the proof.
\end{proof}
\begin{lem}\label{l.fin3}
There exist two absolute constants $\varepsilon>0$ and $\overline{R}<+\infty$ such that, for every $\overline{R} \leq r \leq R < +\infty$, we have
\[
\widehat{\beta}^{an,W}_4(r,R) \leq \varepsilon \widehat{\beta}^{an,W}_4(r/4,4R)  \, .
\]
\end{lem}
\begin{proof}
Note that the event $\widehat{\arm}_4(r,R) \cap \dense_{1/100}(A(R/2,2R))$ is included in $\widehat{\arm}_4^{int}(r,R/2)$ where the event $\dense_{1/100}(A(R/2,2R))$ is the event from Definition~\ref{d.dense}. As a result,
\begin{equation}\label{e.finoupas}
\widehat{\beta}^{an,W}_4(r,R) \leq \Pro \left[ \neg \dense_{1/100}(A(R/2,2R)) \right] + \widehat{\beta}^{an,int,W}_4(r,R/2) \, .
\end{equation}
Moreover, $\widehat{\beta}^{an,int,W}_4(r,R/2) \geq \alpha_4^{an}(r,R)^2$, so $\widehat{\beta}^{an,int,W}_4(r,R/2)$ is greater than some polynomial of $r/R$. Since the probability of $\neg \dense_{1/100}(A(R/2,2R))$ decays to $0$ super-polynomially fast,~\eqref{e.finoupas} implies that $\widehat{\beta}^{an,W}_4(r,R) \leq \frac{1}{2}\widehat{\beta}^{an,int,W}(r,R/2)$ if $R$ is sufficiently large.\\

Now, note that, by Lemmas~\ref{l.fin1} and~\ref{l.fin2}, we have: i) for any $r_0$ sufficiently large and any $\rho \geq r_0$, the quantities $\widehat{\beta}^{an,ext,W}_4(r_0,\rho)$ and $\beta^{an,W}_4(r_0,8\rho)$ are of the same order; ii) for any $\rho$ sufficiently large and any $r_0 \geq \rho$, the quantities $\widehat{\beta}^{an,int,W}_4(\rho/8,r_0)$ and $\beta^{an,W}_4(\rho,r_0)$ are of the same order. As a result, if $r$ is sufficiently large and if $R \geq r$, then
\begin{eqnarray*}
\widehat{\beta}^{an,W}_4(r,R) & \leq & \frac{1}{2} \widehat{\beta}^{an,int,W}_4(r,R/2)\\
& \leq & \grandO{1} \beta^{an,W}_4(r/4,R/2)\\
& \leq & \grandO{1} \widehat{\beta}^{an,ext,W}_4(r/4,R/2)\\
& \leq & \grandO{1} \beta^{an,W}_4(r/4,4R)\\
& \leq & \grandO{1} \widehat{\beta}^{an,W}_4(r/4,4R) \, ,
\end{eqnarray*}
which ends the proof.
\end{proof}

\begin{lem}\label{l.fin4}
There exist two absolute constants $C \in ]0, +\infty[$ and $\overline{R}<+\infty$ such that:
\begin{itemize}
\item For any $\overline{R} \leq r' \leq r \leq R \leq R'<+\infty$, we have
\[
\beta^{an,W}_4(r',R') \geq \frac{1}{C} \widehat{\beta}^{an,W}_4(r,R) \left( \frac{r'R}{rR'} \right)^C \, .
\]
\item If $\overline{R} \leq r \leq R <+\infty$, if $G$ is an event measurable with respect to the configuration outside of $A(r,R)$ satisfying $\Pro \left[ G \right] \geq 1-\frac{1}{C}$, and if $D_{r,R}= B(0,r) \cup B(0,R)^c$, then
\begin{equation}\label{e.dgksdofsjojrf}
\widehat{\beta}_4^{an,W}(r,R) \leq C \Pro \left[ \widetilde{\omega}',\widetilde{\omega}'' \in \arm_4(r/16,16R) \cap G \right] \, ,
\end{equation}
where $\widetilde{\omega}'$ and $\widetilde{\omega}''$ are two configurations with distribution $\Pro_{1/2}$ that satisfy i) $\widetilde{\omega}'= \widetilde{\omega}''$ outside of $D_{r,R} \cup W$; ii) $\widetilde{\omega}'$ and $\widetilde{\omega}''$ are independent of each other in $D_{r,R}$ and iii) in $W \setminus D_{r,R}$, $\widetilde{\omega}'$ and $\widetilde{\omega}''$ have the same underlying point process $\eta$ but are independent conditionally on this point process.
\end{itemize}
\end{lem}
\begin{proof}
If one uses the results of Lemma~\ref{l.fin3} (which gives an ``inwards and outwards'' analogue of the hypotheses~\eqref{e.fin11} and~\eqref{e.fin12} of Lemma~\ref{l.fin1}) and if one follows the proof of Lemma~\ref{l.fin1} (by studying $X' = \Pro \left[ \widetilde{\omega}' \in \widehat{\arm}_4(r/4,4R) \cond \widetilde{\omega}' \cap A(r,R) \right]$ and the analogous variable $X''$ and by extending the arms inwards and outwards at the same time) then one obtains that there exists $C' \in ]0,+\infty[$ such that, if $r'$ is sufficiently large and if $r' \leq r \leq R \leq R' < +\infty$, then
\begin{equation}\label{e.avecretr'}
\Pro \left[ \widetilde{\omega}', \widetilde{\omega}'' \in \arm_4(r',R')   \right] \geq \frac{1}{C} (r'/r)^C (R/R')^C \Pro \left[ \omega',\omega'' \in \widehat{\arm}_4(r,R) \right]  \, .
\end{equation}
Since (by Lemma \ref{l.fin0}) the left hand side is at most $\Pro \left[ \omega',\omega'' \in\arm_4(r',R')   \right]= \beta^{an,W}_4(r',R')$ and since the right hand side is $\frac{1}{C} \widehat{\beta}^{an,W}_4(r,R) \left( \frac{r'R}{rR'} \right)^C$, this gives the first result. Let us now prove the second result of the lemma. To this purpose, first note that
\begin{eqnarray*}
\Pro \left[ \widetilde{\omega}', \widetilde{\omega}'' \in \arm_4(r/16,16R) \setminus G \right] & \leq & \Pro \left[ \widetilde{\omega}', \widetilde{\omega}'' \in \widehat{\arm}_4(r,R) \setminus G \right]\\
& = & \Pro \left[ \widetilde{\omega}', \widetilde{\omega}'' \in \widehat{\arm}_4(r,R) \right] \Pro \left[ \neg G \right]\\
& \leq & \Pro \left[ \omega', \omega'' \in \widehat{\arm}_4(r,R) \right] \Pro \left[ \neg G \right] \text{ (by Lemma }\ref{l.fin0}\text{)} \, .
\end{eqnarray*}
The second result of the lemma is a direct consequence of the above estimate and of \eqref{e.avecretr'} with $r'=r/16$ and $R'=16R$ (and maybe by increasing the constant $C$).
\end{proof}

We are now in shape to prove~\eqref{e.findelathese} i.e. Lemma~\ref{l.first_moment_spec}.
\begin{proof}[Proof of Lemma~\ref{l.first_moment_spec}]
Let $1 \leq r \leq n < +\infty$, let $B$ be a box of radius $r$, let $B'$ the concentric box with radius $r/3$, let $W$ be a Borel subset of the plane such that $W \cap B = \emptyset$, and let $\square$ be a $1 \times 1$ square of the grid $\Z^2$ which is included in $B'$. We assume that $B' \subseteq [0,n]^2$. We recall that $\omega',\omega'' \sim \Pro_{1/2}$ are two configurations that have the same underlying point process $\eta$ and satisfy i) $\omega'_{|\eta \setminus W} = \omega''_{|\eta \setminus W}$ and ii) conditionally on $\eta$, $\omega'_{|\eta \cap W}$ is independent of $\omega''$ and $\omega''_{|\eta \cap W}$ is independent of $\omega'$. In Step 2 of Subsection \ref{ss.proofs_spec} we have seen that, in order to prove Lemma~\ref{l.first_moment_spec}, it is sufficient to prove~\eqref{e.findelathese} i.e. that the following holds if $r$ is sufficiently large:
\[
\E \left[ \un_{|\eta \cap \square|=1} \Pro \left[ \omega', \omega'' \in \Piv_x^q(g_n) \cond \eta \right] \right] \geq \Omega(1) \alpha_4^{an}(r) \Pro \left[ \omega', \omega'' \in \Piv_B(g_n) \right] \, ,
\]
where $x$ is the (random) point such that $\eta \cap \square =\{x\}$ when $|\eta \cap \square |=1$. 

We assume that $r \geq \overline{R}$ where $\overline{R}$ is the constant of Lemma \ref{l.fin4}. We use the following notations: i) $y$ is the center of $B$, ii) $d_0$ is the distance between $y$ and the closest side of $[0,n]^2$, iii) $y_0$ is the projection of $y$ on this side, iv) $d_1$ is the distance between $y_0$ and the closest corner of $[0,n]^2$, v) $y_1$ is this closest corner. We assume that $r \leq d_0/100$, $d_0 \leq d_1/100$ and $d_1 \leq n/100$. The other cases are treated similarly. Without loss of generality, we also assume that $y_0$ lies on the bottom side and that $y_1$ is the left-bottom corner. See Figure \ref{fig:sets}.
\medskip

Below, we use the notation $\widehat{\beta}_4^{an,W}(\cdot,\cdot)$ from \eqref{e.lanotationbetahat} and we use the analogous notations $\widehat{\beta}_3^{an,+,W}(\cdot,\cdot)$ and $\widehat{\beta}_2^{an,++,W}(\cdot,\cdot)$ for the $3$-arm event in the half-plane and the $2$-arm event in the quarter plane. Lemmas~\ref{l.fin1},~\ref{l.fin2},~\ref{l.fin3} and~\ref{l.fin4} are also true for these quantities (and the proofs are exactly the same). We also use the notations $\widehat{\beta}_4^{an,W}(z;\cdot,\cdot)$, $\widehat{\beta}_3^{an,+,W}(z;\cdot,\cdot)$ and $\widehat{\beta}_2^{an,++,W}(z;\cdot,\cdot)$ for the probabilities of the same events but translated by $z$ (but with $B$ and $W$ kept fixed, so these quantities do depend on $z$). Finally, we use the notation $\arm_4(z;\rho_1,\rho_2)$ (respectively $\widehat{\arm}_4(z;\rho_1,\rho_2)$) to denote the event $\arm_4(\rho_1,\rho_2)$ (respectively $\widehat{\arm}_4(\rho_1,\rho_2)$) translated by some $z \in \R^2$. We use the analogous notations for the other arm events.

\begin{figure}[h!]
\centering
\includegraphics[scale=0.8]{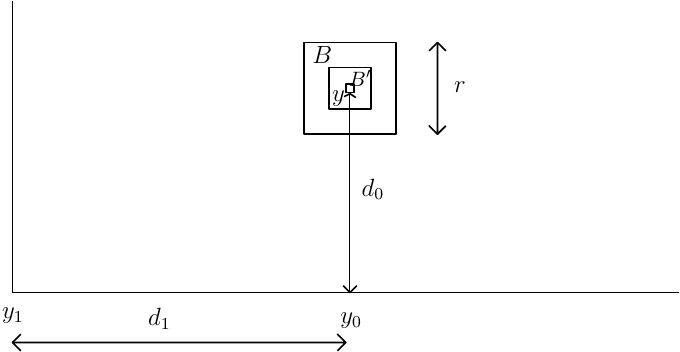}
\caption{The points $y,y_0,y_1$ and the lengths $r,d_0,d_1$.\label{fig:sets}}
\end{figure}

The desired result is a direct consequence of the two following claims. 
\begin{claim}\label{cl.1!!}
We have
\[
\Pro \left[ \omega', \omega'' \in \Piv_B(g_n) \right] \leq \grandO{1} \widehat{\beta}_4^{an,W}(y;r,d_0) \widehat{\beta}_3^{an,+,W}(y_0;d_0,d_1) \widehat{\beta}_2^{an,++,W}(y_1;d_1,n) \, . 
\]
\end{claim}
\begin{claim}\label{cl.2!!}
We have
\begin{multline*}
\alpha_4^{an}(r) \widehat{\beta}_4^{an,W}(y;r,d_0) \widehat{\beta}_3^{an,+,W}(y_0;d_0,d_1) \widehat{\beta}_2^{an,++,W}(y_1;d_1,n) \\
\leq \grandO{1} \E \left[ \un_{|\eta \cap \square|=1} \Pro \left[ \omega', \omega'' \in \Piv_x^q(g_n) \cond \eta \right] \right] \, .
\end{multline*}
\end{claim}
\begin{proof}[Proof of Claim~\ref{cl.1!!}]
To prove this claim, we use the notations for pivotal events from Definition~\ref{d.piv_hat}. Note that we have
\[
\Piv_B(g_n) \subseteq \Piv_B^{A(y;r,d_0/10)}(g_n) \cap \Piv_{B(y_0,10d_0)}^{A(y_0;10d_0,d_1/10)}(g_n) \cap \Piv_{B(y_1,10d_1)}^{A(y_1;10d_1,n/10)}(g_n)
\]
and that the three events on the right-hand side are independent. Hence, it is sufficient to prove the three following estimates
\begin{equation}\label{e.4braspourpivooot}
\Pro \left[ \omega',\omega'' \in \Piv_B^{A(y;r,d_0/10)}(g_n) \right] \leq \grandO{1} \widehat{\beta}_4^{an,W}(y;r,d_0) \, ;
\end{equation}
\[
\Pro \left[ \omega',\omega'' \in \Piv_{B(y_0,10d_0)}^{A(y_0;10d_0,d_1/10)}(g_n) \right] \leq \grandO{1} \widehat{\beta}_3^{an,+,W}(y_0;d_0,d_1) \, ;
\]
\[
\Pro \left[ \omega',\omega'' \in \Piv_{B(y_1,10d_1)}^{A(y_1;10d_1,n/10)}(g_n) \right] \leq \grandO{1} \widehat{\beta}_2^{an,++,W}(y_1;d_1,n) \, .
\]
We only write the proof of \eqref{e.4braspourpivooot} since the other proofs are the same. We use the following notation: for any $\rho \in [r,d_0/10]$ we let $\dense(\rho,d_0/10) = \dense_{1/100}(A(y;\rho/2,2\rho)) \cap \dense_{1/100}( A(y;d_0/20,d_0/5))$. Note that, for every $\rho \in [r,d_0/10]$,
\begin{equation}\label{e.decomposition_piv_dense_arm}
\Piv_B^{A(y;r,d_0/10)}(g_n) \subseteq \widehat{\arm}_4(y;2\rho,d_0/20) \cup \left( \Piv_B^{A(y;r,d_0/10)}(g_n) \setminus \dense(\rho,d_0/10) \right) \, .
\end{equation}
Note furthermore that $\widehat{\arm}_4(y;2\rho,d_0/20)$ is measurable with respect to $\omega \cap A(y;2\rho,d_0/20)$ while $\dense(\rho,d_0/10)$ is measurable with respect to $\eta \cap \left( A(y;\rho/2,2\rho) \cup A(y;d_0/20,d_0/5) \right)$. Note also that $\Pro \left[ \neg \dense(\rho,d_0/10) \right] \leq \grandO{1} \exp(-\Omega(1)\rho^2)$. As a result, \eqref{e.decomposition_piv_dense_arm} applied to $\rho = r,2r,\cdots$ implies that $\Pro \left[ \omega',\omega'' \in \Piv_B^{A(y;r,d_0/10)}(g_n) \right]$ is at most
\begin{multline*}
\grandO{1} \widehat{\beta}^{an,W}_4(y;2r,d_0/20) \\
+\sum_{k=1}^{\lfloor \log_2(d_0/20r) \rfloor} \grandO{1} \exp\left(-\Omega(1)(2^kr)^2 \right) \widehat{\beta}^{an,W}_4(y;2^{k+1} r,d_0/20) + \grandO{1} \exp\left(-\Omega(1)d_0^2 \right)\, .
\end{multline*}
The first part of Lemma~\ref{l.fin4} implies that the above is at most $\grandO{1} \widehat{\beta}_4^{an,W}(y;r,d_0)$. This ends the proof of \eqref{e.4braspourpivooot}.
\end{proof}

\begin{proof}[Proof of Claim~\ref{cl.2!!}]
Let $D = \big(  B(y_0,8d_0) \setminus B(y,d_0/8) \big) \cup \big(  B(y_1,8d_1) \setminus B(y_0,d_1/8) \big) \cup B(y_1,n/8)^c$ and let $\widehat{\omega}'$ and $\widehat{\omega}''$ be two configurations 
with law $\Pro_{1/2}$ such that
\begin{itemize}
\item[i)] $\widehat{\omega}'=\widehat{\omega}''$ outside of $W \cup D$;
\item[ii)] $\widehat{\omega}'$ and $\widehat{\omega}''$ are independent of each other in $D$;
\item[iii)] in $W \setminus D$, $\widehat{\omega}'$ and $\widehat{\omega}''$ have the same underlying point process $\eta$ but are independent conditionally on this point process.
\end{itemize}
Remember that $B$ is the box of radius $r$ centered at $y$ and that $W \cap B=\emptyset$. See Figure \ref{fig:grey}.

\begin{figure}[h!]
\centering
\includegraphics[scale=0.7]{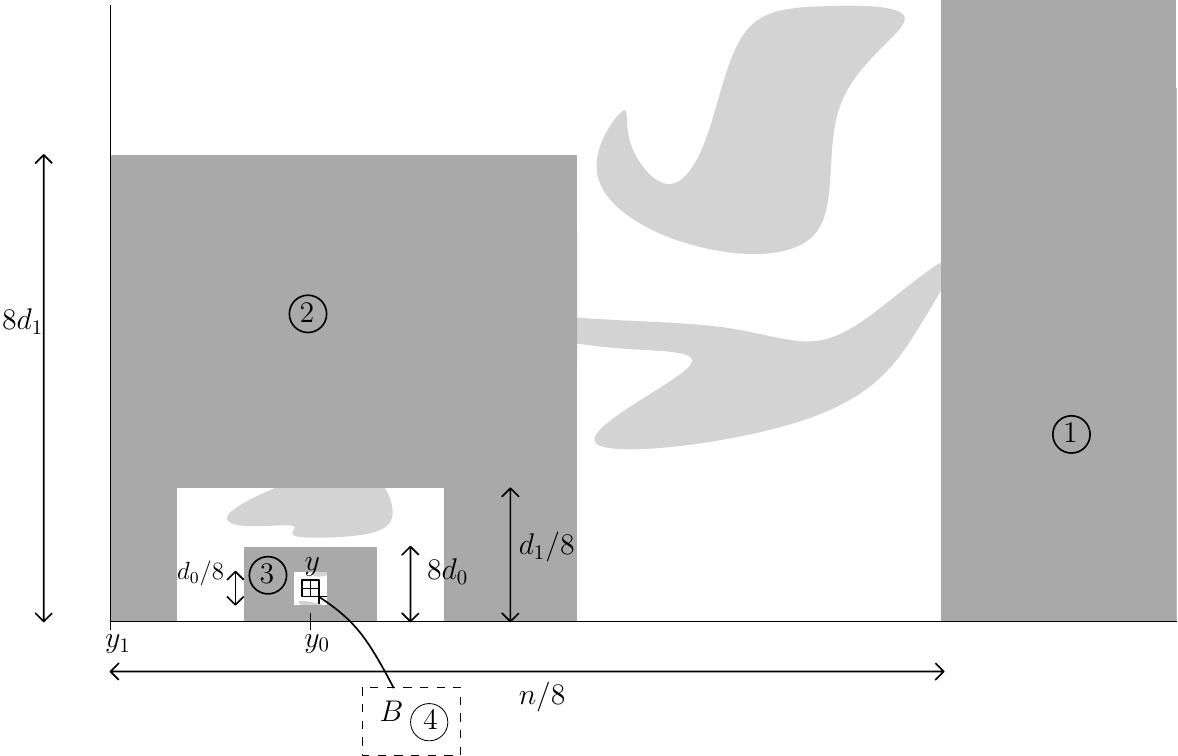}
\caption{\blue{In darkgray: the set $D$. In lightgray: the set $W \setminus D$. Steps (1)--(4) below ``take place'' in the four regions $\textcircled{\scriptsize{1}}$--$\textcircled{\scriptsize{4}}$ respectively. The configurations $\widehat{\omega}'$ and $\widehat{\omega}''$ are independent in the three darkgray regions $\textcircled{\scriptsize{1}}$--$\textcircled{\scriptsize{3}}$. This property will help us to glue paths, exactly like in the proof of Lemma \ref{l.fin1}. In the region $\textcircled{\scriptsize{4}}$ (i.e.\ in the box $B$ and its neighbourhood), we will rather use that these two configurations coincide. The reader can have in mind that the four darkgrey regions $\textcircled{\scriptsize{1}}$--$\textcircled{\scriptsize{4}}$ roughly are annuli (or annuli intersected with the half- or quarter-plane) and that in each of these annuli, the inner radius is of the same order as the outer radius. As a result, the gluing operations in these regions will not modify the order of the arm probabilities.}\label{fig:grey}}
\end{figure}

By the second part of Lemma~\ref{l.fin4} (and Lemma \ref{l.fin0}), there exists $h > 0$ such that, if $r$ is sufficiently large and if $G(y,r)$ is an event measurable with respect to $\omega \cap B(y,r) = \omega \cap B$, $G(y,d_0)$ is an event measurable with respect to $\omega \cap A(y;d_0/8,d_0/2)$, $G(y_0,d_0)$ is an event measurable with respect to $\omega \cap A(y_0;2d_0,8d_0)$, $G(y_0,d_1)$ is an event measurable with respect to $\omega \cap A(y_0;d_1/8,d_1/2)$, $G(y_1,d_1)$ is an event measurable with respect to $\omega \cap A(y_1;2d_1,8d_1)$ and $G(y_1,n)$ is an event measurable with respect to $\omega \setminus B(y_1,n/8)$ that are all of probability at least $1-h$, then we have
\begin{multline}\label{e.cons_de_hat}
\widehat{\beta}_4^{an,W}(y;r,d_0) \widehat{\beta}_3^{an,+,W}(y_0;d_0,d_1) \widehat{\beta}_2^{an,++,W}(y_1;d_1,n)\\
\leq \grandO{1} \Pro \left[ \widehat{\omega}',\widehat{\omega}'' \in \arm_4(y;r/2,d_0/4) \cap G(y,r)  \cap G(y,d_0) \right]\\
\hspace{1cm} \times \Pro \left[ \widehat{\omega}',\widehat{\omega}'' \in \arm_3^+(y_0;4d_0,d_1/4) \cap G(y_0,d_0) \cap G(y_0,d_1) \right]\\
\times \Pro \left[ \widehat{\omega}',\widehat{\omega}'' \in \arm_2^{++}(y_1;4d_1,n/4) \cap G(y_1,d_1) \cap G(y_1,n) \right] \, ,
\end{multline}
Let us now prove that, if the events $G(y,r)$, $G(y,d_0)$, $G(y_0,d_0)$, $G(y_0,d_1)$, $G(y_1,d_1)$ and $G(y_1,n)$ are suitably chosen, then
\begin{equation}\label{e.FIIIIIN}
\text{Right-hand side of } \eqref{e.cons_de_hat} \leq \grandO{1} \frac{ \Pro \left[ |\widehat{\eta}'\cap \square| = 1, \, \widehat{\omega}', \widehat{\omega}'' \in \Piv_x^q(g_n) \right] }{\alpha_4^{an}(r)} \, ,
\end{equation}
where $\widehat{\eta}'$ is the non-colored underlying point process of $\widehat{\omega}'$ and $x$ is the (random) point such that $\widehat{\eta}' \cap \square = \{ x \}$ on the event $\{ |\widehat{\eta}'\cap \square| = 1 \}$. Note that, by definition of $\widehat{\omega}'$ and $\widehat{\omega}''$ (and since $W$ does not intersect $B$), we have $\widehat{\omega}' \cap \square = \widehat{\omega}'' \cap \square$ (so $\widehat{\eta}' \cap \square = \widehat{\eta}'' \cap \square$). Before proving \eqref{e.FIIIIIN}, let us note that, by Lemma~\ref{l.fin0}, we have
\begin{multline*}
\Pro \left[ |\widehat{\eta}'\cap \square| = 1, \, \widehat{\omega}', \widehat{\omega}'' \in \Piv_x^q(g_n) \right]\\
\leq \Pro \left[ |\eta \cap \square| = 1, \, \omega', \omega'' \in \Piv_x^q(g_n) \right] = \E \left[ \un_{|\eta \cap \square|=1} \Pro \left[ \omega', \omega'' \in \Piv_x^q(g_n) \cond \eta \right] \right] \, .
\end{multline*}
Hence, \eqref{e.FIIIIIN} implies the desired result. The proof of~\eqref{e.FIIIIIN} follows four steps. Let $E_{r,d_0}=E_{r,d_0}(\delta)$, $E_{d_0,d_1}=E_{d_0,d_1}(\delta)$ and $E_{d_1,n}=E_{d_1,n}(\delta)$ the three events from the right-hand side of~\eqref{e.cons_de_hat}. So~\eqref{e.cons_de_hat} is
\[
\widehat{\beta}_4^{an,W}(y;r,d_0) \widehat{\beta}_3^{an,+,W}(y_0;d_0,d_1) \widehat{\beta}_2^{an,++,W}(y_1;d_1,n) \leq \grandO{1} \Pro \left[ E_{r,d_0} \right] \Pro \left[ E_{d_0,d_1} \right] \Pro\left[ E_{d_1,n} \right] \, .
\]
\begin{itemize}
\item[1)] Let $F_{d_1,n}$ denote the event that, in both $\widehat{\omega}'$ and $\widehat{\omega}''$, there exist one white arm and one black arm from $\partial B(y_1,4d_1) \cap [0,n]^2$ to the top and right sides of $[0,n]^2$ respectively (remember that we have assumed that $y_1$ is the left-bottom corner and that $y_0$ belongs to the bottom side). We choose $G(y_1,n)=G_\delta^{ext}(y_1,n)$, where $G^{ext}_\delta(z,\rho)$ is $G^{ext}_\delta(\rho)$ translated by $\rho$ (see \eqref{e.Gext} for the notation $G^{ext}_\delta(\rho)$) and with $\delta$ sufficiently small so that $\Pro \left[ G^{ext}_\delta(y_1,n) \right] \geq 1- h$. Since $\widehat{\omega}'$ and $\widehat{\omega}''$ are independent of each other outside of $B(y_1,n/8)$, by exactly the same arguments as in Remark~\ref{r.lextensionavecwidetilde} (but for a $2$-arm event in the quarter-plane instead of the $4$-arm event in the whole plane), we have $\Pro \left[ F_{d_1,n}, \, \widehat{\omega}', \widehat{\omega}'' \in G(y_1,d_1) \right] \geq \Omega(1) \Pro \left[ E_{d_1,n} \right]$ (for any choice of event $G(y_1,d_1)$).
\item[2)] Let $F_{d_0,n} \subseteq F_{d_1,n}$ be the event that, in both $\widehat{\omega}'$ and $\widehat{\omega}''$, there exist one black arm, one white arm and one black arm from $\partial  B(y_0,4d_0) \cap [0,n]^2$ to the left, top and right sides of $[0,n]^2$ respectively. In this step, we prove that we can choose $G(y_0,d_1)$ and $G(y_1,d_1)$ such that
\begin{equation}\label{e.FIIIIN}
\Pro \left[ F_{d_0,n}, \, \widehat{\omega}', \widehat{\omega}'' \in G(y_0,d_0) \right]
\geq \Omega(1) \Pro \left[ E_{d_0,d_1} \right] \Pro \left[ F_{d_1,n}, \, \widehat{\omega}', \widehat{\omega}'' \in G(y_1,d_1) \right] \, .
\end{equation}
This estimate is a little harder than the estimate of Step~1) above because of the interactions between the event $E_{d_0,d_1}$ and the event $F_{d_1,n} \cap \{ \widehat{\omega}', \widehat{\omega}'' \in G(y_1,d_1) \}$ at scale $d_1$. In Subsection~7.1 of~\cite{V1} (see Lemma~7.9 and just below this lemma), we have proved the following for every $r_1$ sufficiently large, $r_2 \geq 6r_1$ and $r_3 \geq 6r_2$:
\begin{multline*}
\Pro_{1/2} \left[ \arm_4(r_1,r_3) \right]\\
\geq \Omega(1) \Pro_{1/2} \left[ \arm_4(r_1,r_2/3) \cap G^{ext}_\delta(r_2/3)  \right] \cdot \Prob_{1/2} \left[ \arm_4(3r_2,r_3) \cap G^{int}_\delta(3r_2) \right] \, ,
\end{multline*}
where $G^{int}_\delta(\rho)$ is an ``interior'' analogue of $G^{ext}_\delta(\rho)$. With exactly the same proof we even have the following pointwise inequality:
\begin{multline*}
\Pro_{1/2} \left[ \arm_4(r_1,r_3) \cond \omega \cap \left( B(0,r_2/6) \cup B(0,6r_2)^c \right) \right]\\
\geq \Omega(1) \Pro_{1/2} \left[ \arm_4(r_1,r_2/3) \cap G^{ext}_\delta(r_2/3) \cond \omega \cap \left( B(0,r_2/6) \cup B(0,6r_2)^c \right) \right]\\
\times \Prob_{1/2} \left[ \arm_4(3r_2,r_3) \cap G^{int}_\delta(3r_2)  \cond \omega \cap \left( B(0,r_2/6) \cup B(0,6r_2)^c \right) \right] \, .
\end{multline*}
Since $\widehat{\omega}'$ and $\widehat{\omega}''$ are independent of each other in $A(y_0;d_1/8,8d_1)$, the proof of the above adapts readily to our case by letting $G(y_0,d_1)$ be an analogue of $G^{ext}_\delta(y_0,d_1/4)$ and by letting $G(y_1,d_1)$ be an analogue of $G^{int}_\delta(y_1,4d_1)$. Since the proof of \eqref{e.FIIIIN} does not require any new idea, we leave the details to the reader and we refer to \cite{V1} for more details.
\item[3)] The third step consists in proving that we can choose $G(y,d_0)$ and $G(y_0,d_0)$ such that
\[
\Pro \left[  F_{r,n} , \, \widehat{\omega}',\widehat{\omega}'' \in G(y,r) \right] \geq \Omega(1) \Pro \left[ E_{r,d_0} \right] \Pro \left[ F_{d_0,n}, \, \widehat{\omega}', \widehat{\omega}'' \in G(y_0,d_0) \right] \, ,
\]
where $F_{r,n} \subseteq F_{d_0,n}$ is the event that, in both $\widehat{\omega}'$ and $\widehat{\omega}''$, there are two black arms from $\partial B(y,r/2)$ to the left and right sides of $[0,n]^2$ and two white arms from $\partial B(y,r/2)$ to the top and bottom sides of $[0,n]^2$. The proof is exactly the same as in Step~2).
\item[4)]
Let $F_{1,n}$ denote the event that in both $\widehat{\omega}'$ and $\widehat{\omega}''$, there is only one point in $\square$ and there are two black arms from the cell $C$ of this point to the left and right sides of $[0,n]^2$ and two white arms from $C$ to the top and bottom sides of $[0,n]^2$ (and that $C$ is included in $[0,n]^2$). The fourth step consists in proving that we can choose $G(y,r)$ such that
\begin{equation}\label{e.edgqegffgd}
\Pro \left[ F_{1,n} \right] \geq \Omega(1)\alpha_4^{an}(r) \Pro \left[  F_{r,n} , \, \widehat{\omega}',\widehat{\omega}'' \in G(y,r) \right] \, .
\end{equation}
In \cite{V1} (see Lemma 4.6), we have proved the following: if $\square' \subseteq [0,r]^2$ is a $1 \times 1$ square at distance at least $r/3$ from the sides of $[0,r]^2$, let $F'_{1,r}$ denote the event that there is only one point in $\square'$ and there are two black arms from the cell $C'$ of this point to the left and right sides of $[0,r]^2$ and two white arms from $C'$ to the top and bottom sides of $[0,r]^2$ (and that $C'$ is included in $[0,r]^2$). Then, $\Pro [ F_{1,r}' ] \geq \Omega(1) \alpha^{an}_4(r)$. Since the configurations $\widehat{\omega}'$ and $\widehat{\omega}''$ coincide in $B=B(y,r)$, the proof of the desired result \eqref{e.edgqegffgd} is the same as in \cite{V1}, so we leave once again the details to the reader and refer to \cite{V1}.
\end{itemize}
The four steps imply~\eqref{e.FIIIIIN} since $F_{1,n} \subseteq \{ |\widehat{\eta}'\cap \square| = 1 \} \cap \{ \widehat{\omega}', \widehat{\omega}'' \in \Piv_x^q(g_n) \}$, which ends the proof of the claim.
\end{proof}
This also ends the proof of Lemma~\ref{l.first_moment_spec}.
\end{proof}

\appendix

\section{Simple properties of the $\mu$-dynamical processes}\label{s.simple}

\paragraph{A metric on $\Omega$.} As explained in Appendix~A.2.6 of \cite{daley2003introduction}, we can equip the set $\mathcal{M}_{\R^2 \times \{-1,1\}}^{B}$ of all (locally finite) Borel measures on $\R^2 \times \{-1,1\}$ with a metric $d$ such that i) $(\mathcal{M}_{\R^2 \times \{-1,1\}}^{B},d)$ is a Polish space, ii) the restriction\footnote{The set $\Omega$ is included in $\mathcal{M}_{\R^2 \times \{-1,1\}}^{B}$ if we see an element $\overline{\omega}$ of $\Omega$ as the corresponding measure $\sum_{(x,\eps) \in \overline{\omega}} \delta_{(x,\eps)}$.} of $d$ to $\Omega$ generates the (classical) $\sigma$-algebra that we have defined in Subsection~\ref{ss.main} and iii) the restriction of $d$ to $\Omega$ is given by the following expression, where $\overline{\omega}_1,\overline{\omega}_2 \in \Omega$ and where $\overline{\eta}_1$ and $\overline{\eta}_2$ are such that $\overline{\omega}_1 \in \{-1,1\}^{\overline{\eta}_1}$ and $\overline{\omega}_2 \in \{-1,1\}^{\overline{\eta}_2}$:
\begin{equation}\label{e.detd'}
d(\overline{\omega}_1,\overline{\omega}_2) = \int_0^{+\infty} e^{-r} \frac{d'_r(\overline{\omega}_1,\overline{\omega}_2)}{1+d'_r(\overline{\omega}_1,\overline{\omega}_2)} dr \, .
\end{equation}
The quantity $d'_r(\overline{\omega}_1,\overline{\omega}_2)$ of~\eqref{e.detd'} is defined by
\[
d'_r(\overline{\omega}_1,\overline{\omega}_2) = \inf_\phi \sup_x \, ||x-\phi(x)||_2 \, ,
\]
where the infimum above is over every bijection $\phi$ between $\overline{\eta}_1 \cap [-r,r]^2$ and $\overline{\eta}_2 \cap [-r,r]^2$ such that
\[
\forall x \in \overline{\eta}_1 \cap [-r,r]^2, \, \overline{\omega}_1(x) = 1 \Leftrightarrow \overline{\omega}_2(\phi(x)) = 1 \, ,
\]
and the supremum is over every $x \in \overline{\eta}_1 \cap [-r,r]^2$. When such a bijection does not exist, we let $d'_r(\overline{\omega}_1,\overline{\omega}_2)=1$.
\medskip

With this definition, it is clear that the frozen dynamical Voronoi percolation process is a càdlàg process with values in $(\Omega,d_{|\Omega})$ and that $\Pro_{1/2}$ is an invariant measure of this process. Now, let $\mu$ be the law of a Lévy process in the plane starting from $0$ \blue{and let us show that the same property holds for the $\mu$-dynamical Voronoi percolation process.} Let $\omega^\mu(0) \sim \Pro_{1/2}$ and let each point of the underlying configuration $\eta^{\mu}(0)$ evolve conditionally independently according to a process of law $\mu$. For each $t$, let $\omega^\mu(t)$ be the \blue{configuration} \blue{at time $t$. (1) Let us first prove that, for each $t$, a.s.\ $\omega^\mu(t)$ is locally finite. To this purpose, for each point $x \in \Z^2$, let $C_x=x+[-1/2,1/2[^2$ (note that $(C_x)_x$ is a partition of the plane). The expectation of $|\omega^{\mu}(t) \cap C_0|$ equals the sum on $x \in \Z^2$ of the expected number of points that were in $C_x$ at time $0$ and are in $C_0$ at time $t$. By translation invariance of the point process, this also equals the sum on $x \in \Z^2$ of the expected number of points that were in $C_0$ at time $0$ and are in $C_{-x}$ at time $t$. As a result, $\E [ |\omega^{\mu}(t) \cap C_0| ] = \E [ |\omega^{\mu}(0) \cap C_0| ] = 1$. This proves that, for each fixed $t$, a.s. $\omega^\mu(t)$ is locally finite.} (2) By using the Markov property of Lévy processes, it is not difficult to deduce from this that a.s., for every $M,T<+\infty$, there are only finitely many points that intersect $[-M,M]^2$ at some time $t \in [0,T]$. In particular, a.s.\ for every $t$, $\omega^\mu(t)$ is locally finite (i.e. we can replace ``for every $t$ a.s.'' by ``a.s. for every $t$''). (3) Note furthermore that classical  potential theory results on Lévy processes (see~\cite{bertoin1998levy}) imply that a.s. there is no collusion of points. By using that the time reversal process of a Lévy process is a càglàd Lévy process, we even have that, for each $T < +\infty$ and each $M<+\infty$, the infimum on $t \in [0,T]$ and on $x \neq y \in \eta^\mu(t) \cap [-M,M]^2$ of $||x-y||_2$ is positive. All these properties imply that the process $(\omega^\mu(t))_{t \in \R_+}$ is a well-defined càdlàg process with values in $(\Omega,d_{|\Omega})$. \blue{(4) Let us finally explain why $\Pro_{1/2}$ is an invariant measure of this process. Consider the torus $\T_R$ of side $2R$ (that is, the set obtained from $[-R,R]^2$ by identifying the top and bottom sides and the left and right sides), let $\omega^\mu_R(0)$ be a set of $(2R)^2$ independent uniform points in $\T_R \times \{-1,1\}$, define $(\omega^\mu_R(t))_{t \geq 0}$ by letting each point evolve in the torus independently and according to a process of law $\mu$ and note that the law of $\omega^\mu_R(t)$ is invariant in time. By the arguments from point (1) above, for each $t \ge 0$ and each $M<+\infty$, the probability that some points in $\omega^\mu(t) \cap [-M,M]^2$ were outside of $[-\rho,\rho]^2$ at time $0$ goes to $0$ as $\rho$ goes to $+\infty$. This (together with a similar property for $\omega^\mu_R(t)$) implies that, for each measurable set $A \subset \Omega$ that depends on the configuration restricted to a compact subset of the plane and each $t \ge 0$, we have $\Pro [ \omega^\mu_R(t) \in A ] \rightarrow \Pro [  \omega^\mu(t) \in A ]$ as $R \rightarrow +\infty$. This proves that the law of $\omega^\mu(t)$ is invariant in time.}
\medskip

Let us end this short section with a remark concerning giant cells.
\begin{rem}\label{r.inf_cell}
In this remark we observe that a.s. there is no time with an unbounded Voronoi cell. For the frozen dynamical process this is obvious since the point process $\eta$ does not evolve in time. Concerning the $\mu$-dynamical processes, one can do the following reasoning: first, fix some $T<+\infty$ and for each $1 \times 1$ box $B$ of the grid $\Z^2$ let $Z_B$ be the random variable that equals $1$ if there exists $x \in \eta^\mu(0) \cap B$ that stays in $B$ at least until time $T$. Note that $(Z_B)_B$ is a family of i.i.d. Bernoulli variables of some parameter $p(T) \in ]0,1[$. It is sufficient to prove the following claim:
\begin{claim}
A.s., for every configuration $\overline{\eta}$ that contains at least one point in each box $B$ such that $Z_B=1$, the Voronoi tiling of $\overline{\eta}$ has no unbounded cell.
\end{claim}
\begin{proof}
If $k \in \N$, we let $\dense(k)$ denote the event that, for every $u \in A(2^k,2^{k+1}) = [-2^{k+1},2^{k+1}]^2 \setminus ]-2^k,2^k[^2$, there is a box $B \subseteq A(2^k,2^{k+1})$ at distance less than $2^k/100$ from $u$ such that $Z_B =1$. Note that, if this event holds, then for any $\overline{\eta}$ that contains at least one point in each box $B$ such that $Z_B=1$, there is no Voronoi cell of $\overline{\eta}$ that intersects both $[-2^k,2^k]^2$ and $\partial[-2^{k+1},2^{k+1}]^2$. As a result, it is sufficient to prove that, a.s., infinitely many events $\dense(k)$ hold. To this purpose, we just note that the events $\dense(k)$ are independent and that
\[
\Pro \left[ \dense(k) \right] \geq 1 - \grandO{1} \exp \left( -\Omega(1) p(T)2^{2k} \right) \, .
\]
\end{proof}
\end{rem}

\section{Measurability issues about the annealed spectral sample}\label{a.meas}

\blue{In this appendix, we prove that, for any $A \in \mathcal{F}'$ and any bounded measurable function $h \, : \, \Omega \rightarrow \R$, the function
\begin{equation}\label{eq:meas}
\sum_{S \subseteq_f \eta, \, S \in A} \widehat{h^\eta}(S)^2
\end{equation}
is measurable. We actually prove the following more general result. Let $h_1, h_2 \, : \, \Omega \rightarrow \R$ be two bounded measurable functions and let $A_1,A_2 \in \mathcal{F}'$. Also, let $\mu$ be the law of a planar Lévy process starting from $0$ and define $(\eta(t))_{t \geq 0}=(\eta^{\mu}(t))_{t \geq 0}$ by letting each point of $\eta$ move independently and according to a process of law $\mu$. For every $S \subseteq \eta(0)$, let $S_t$ be the corresponding subset at time $t$ (these sets are a.s.\ well defined because (1) a.s.\ there is no collusion of points\footnote{We even have: for each $T < +\infty$ and each $M<+\infty$, the infimum on $t \in [0,T]$ and on $x \neq y \in \eta(t) \cap [-M,M]^2$ of $||x-y||_2$ is positive, see Appendix \ref{s.simple}.} and (2) a.s.\ no two particles jump at the same time\footnote{To prove this, note that a Lévy process jumps countably many times and that for every fixed time $t$, a.s. a Lévy process does not jump at time $t$.}).
\begin{lemma}\label{lem:meas}
The function
\[
\sum_{S \subseteq_f \eta(0)} \widehat{h_1^{\eta(0)}}(S)\widehat{h_2^{\eta(t)}}(S_t) 1_{S \in A_1, \, S_t \in A_2}
\]
is measurable.
\end{lemma}
Lemma \ref{lem:meas} implies \eqref{eq:meas} by choosing $A_1=A_2=A$, $h_1=h_2=h$ and $t=0$.  
\begin{proof}[Proof of Lemma \ref{lem:meas}]
We first note that the $\sigma$-algebra $\mathcal{F}'$ is generated by the sets $\{ \overline{\eta} \in \Omega' \, : \, \overline{\eta} \cap D = \emptyset \}$, where $D$ spans the Borel subsets of the plane. As a result, by the monotone class theorem, it is sufficient to prove the result for two such sets $A_1,A_2$. To this purpose, it is sufficient to prove the following general formula: Let $D_1,D_2$ be two Borel subsets of the plane. Then,
\begin{multline*}
\sum_{S \subseteq_f \eta(0)} \widehat{h_1^{\eta(0)}}(S)\widehat{h_2^{\eta(t)}}(S_t) 1_{S \cap D_1 = \emptyset, \, S_t \cap D_2 = \emptyset}\\
= \E \left[ \E \left[ h_1(\omega(0)) \cond \eta(0), \, \omega(0) \setminus D_1 \right] \E \left[ h_2(\omega(t)) \cond \eta(t), \, \omega(t) \setminus D_2 \right] \cond (\eta(s))_{0 \leq s \leq t} \right].
\end{multline*}
To show this formula we observe that
\[
\E \left[ h_1(\omega(0)) \cond \eta(0), \, \omega(0) \setminus D_1 \right] = \sum_{S \subseteq_f \eta(0), \, S \cap D_1 = \emptyset} \widehat{h_1^{\eta(0)}}(S) \chi_S(\omega(0))
\]
and similarly
\[
\E \left[ h_2(\omega(t)) \cond \eta(t), \, \omega(t) \setminus D_2 \right] = \sum_{S' \subseteq_f \eta(t), \, S' \cap D_2 = \emptyset} \widehat{h_2^{\eta(t)}}(S') \chi_{S'}(\omega(t)).
\]
Next, we observe that
\begin{align*}
& \E \left[ \left( \sum_{S \subseteq_f \eta(0), \, S \cap D_1 = \emptyset} \widehat{h_1^{\eta(0)}}(S) \chi_S(\omega(0)) \right) \left( \sum_{S' \subseteq_f \eta(t), \, S' \cap D_2 = \emptyset} \widehat{h_2^{\eta(t)}}(S') \chi_{S'}(\omega(t)) \right) \cond (\eta(s))_{0 \leq s \leq t} \right]\\
& = \sum_{S \subseteq_f \eta(0), \, S \cap D_1 = \emptyset \atop S' \subseteq_f \eta(t), \, S' \cap D_2 = \emptyset} \widehat{h_1^{\eta(0)}}(S)\widehat{h_2^{\eta(t)}}(S') \E \left[ \chi_S(\omega(0)) \chi_{S'}(\omega(t)) \cond (\eta(s))_{0 \leq s \leq t} \right].
\end{align*}
Finally,
\[
\E \left[ \chi_S(\omega(0)) \chi_{S'}(\omega(t)) \cond (\eta(s))_{0 \leq s \leq t} \right] = \un_{S'=S_t}.
\]
\end{proof}}

\section{The $2^{nd}$ moment method}\label{a.2nd}

In this appendix, we prove Lemma~\ref{l.secondmoment}. To this purpose, we follow the proof of the analogous result from~\cite{olle1997dynamical}. Recall that $f_R$ is the $1$-arm event and that
\[
X_R = \int_0^1 f_R \left( \omega(t) \right) dt \, .
\]
Assume that there exists a constant $C$ as in Lemma~\ref{l.secondmoment} and consider the following (random) set:
\[
T_R = \left\lbrace t \in [0,1] \, : \, 0 \overset{\omega(t)}{\longleftrightarrow} R \right\rbrace \, .
\]
By using the fact that $\left( X_R \right)_{R > 0}$ is decreasing and the Cauchy-Schwarz inequality, we obtain that
\begin{eqnarray*}
\Pro \left[ \forall R > 0, \, T_R \neq \emptyset \right] & \geq & \Pro \left[ \forall R > 0, \, X_R \neq 0 \right] \nonumber\\
& = & \underset{R \rightarrow + \infty}{\text{lim}} \Pro \left[ X_R > 0 \right] \nonumber\\
& \geq & \underset{R \rightarrow + \infty}{\text{liminf}} \, \frac{\E \left[ X_R \right]^2}{\E \left[ X_R^2 \right]} \nonumber\\
& \geq & 1/C > 0 \, .
\end{eqnarray*}
By Kolmogorov $0$-$1$ law, a.s. there are exceptional times or a.s. there is no exceptional time. As a result, it is sufficient to prove the following lemma.
\begin{lem}\label{l.for_every_R_TR}
We have the following a.s.: if for every $R > 0$ the set $T_R$ is non-empty, then there exists a time $t \in [0,1]$ at which there is an unbounded black component.
\end{lem}
\begin{proof}
\blue{The proof for frozen dynamical Voronoi percolation is exactly the same as that for Bernoulli percolation, see Section~5 of~\cite{schramm2010quantitative} (where the authors rely on Lemma~3.2 of~\cite{olle1997dynamical}). Let $\mu$ denote the law of a planar Lévy process starting from $0$ and consider a $\mu$-dynamical Voronoi percolation process. If $M < +\infty$, we let $\eta_M \subseteq \eta(0)$ be the set of points whose Voronoi cell intersects $[-M,M]^2$ at some time $t \in [0,1]$. If for every $M$ a.s. the adjacency matrix of $\eta_M$ changes its values only finitely many times between time $0$ and time $1$, then the proof is also exactly the same as in \cite{schramm2010quantitative}. This properties holds for instance for compound Poisson processes\footnote{Let us note that there exist compound Poisson processes that satisfy the hypotheses of Theorem \ref{t.main_levy}.}.}

Let us also note that, in the above cases, the fact that we have chosen that the color of the points at the boundary of both a black and a white cell are colored in black and white is not important.
\medskip

Let us write the proof in the general case of a $\mu$-dynamical Voronoi process. First note that if the sets $T_R$ were a.s. closed (hence compact), then the proof would have been easy. The first step of the proof consists in modifying a little the processes so that, for each $t \in \cap_{R > 0} \overline{T_R} \subseteq [0,1]$, there is an unbounded black component at time $t$ (where $\overline{T_R}$ is the closure of $T_R$).
\medskip

If $y(0) \in \eta(0)$, we write $y_t$ for the corresponding point of $\eta(t)$. For every $t \geq 0$, we define two graphs $G(t)$ and $\overline{G}(t)$ as follows:
\begin{enumerate}
\item The vertex set of both $G(t)$ and $\overline{G}(t)$ is $\eta(t)$.
\item Two points of $\eta(t)$ are adjacent in $G(t)$ if their Voronoi cells are adjacent (i.e. if the intersection of the two Voronoi cells is non-empty; remember that with our definition the Voronoi cells are closed sets).
\item The edge set of $\overline{G}(t)$ is defined using $(G(s))_{s \geq 0}$ as follows: two points $y_t,z_t \in \eta(t)$ are adjacent in $\overline{G}(t)$ if there exists $t_n \rightarrow t$ such that for each $n$, $\{y_{t_n},z_{t_n}\}$ is an edge of $G(t_n)$.
\end{enumerate}
Remember that in Appendix~\ref{s.simple} we have observed that a.s. there is no time with an unbounded Voronoi cell. As a result, a.s. the set of times $t \in [0,1]$ such that there is an infinite path of black vertices in $\overline{G}(t)$ contains $\cap_{R > 0} \overline{T_R}$. Hence, if for every $R > 0$ we have $T_R \neq \emptyset$, then such an infinite black path exists. Note also that a.s. for every $t \in [0,1]$ if there exists an infinite component in $G(t)$ made of black vertices then there exists an unbounded black component in the coloring of the plane induced by $\omega(t)$. This ends the proof provided that we show the following lemma.
\end{proof}

\begin{lem}\label{l.jump}
Let $\mu$ be the distribution of a planar Lévy process and let $(\omega(t))_{t \geq 0}$ be a $\mu$-dynamical Voronoi process. We use the same notations as in the proof of Lemma~\ref{l.for_every_R_TR}. A.s. for every $t \in \R_+$ there exists an infinite component in $G(t)$ made of black vertices if and only if there exists an infinite component made of black vertices in $\overline{G}(t)$.
\end{lem}
\begin{proof}
Let us first prove the following claim:
\begin{claim}\label{cl.jump}
We have the following a.s.: for every $t \geq 0$, if $\overline{G}(t) \neq G(t)$, then one of the Lévy processes is not continuous at time $t$.
\end{claim}
\begin{proof}
Let us first make the following deterministic observation: let $\overline{\eta}$ be an infinite locally finite set of points of the plane, consider the Voronoi tiling induced by $\overline{\eta}$, and let $y,z \in \overline{\eta}$. Then, the cells of $y$ and $z$ intersect each other if and only if $y$ and $z$ belong to the boundary of a disc $D$ whose interior does not contain any point of $\overline{\eta}$.
\medskip

Now, let $t \geq 0$ such that all the Lévy processes are continuous at time $t$ and let $y_t$ and $z_t$ be two points of $\eta(t)$ such that $\{ y_t,z_t \} \in \overline{G}(t)$. We want to prove that, except on a set of probability $0$ that does not depend on $t$, $\lbrace y_t,z_t \rbrace \in G(t)$. Since $\{ y_t,z_t \} \in \overline{G}(t)$, there exists $t_n \rightarrow t$ such that, for each $n$, $\lbrace y_{t_n},z_{t_n} \rbrace \in G(t_n)$. By the deterministic observation above, we deduce that there exist discs $D_n$ (with centers $a_n$ and radii $r_n$, say) such that $y_{t_n}$ and $z_{t_n}$ belong to the boundary of $D_n$ and $\eta(t_n)$ does not intersect the interior of $D_n$. By using the arguments from Remark~\ref{r.inf_cell}, one can easily show that, except on a set of probability $0$ that does not depend on $t$, $(r_n)_n$ and $(a_n)_n$ are bounded sequences (indeed, otherwise there would exist an open half-plane $H$ such that there is no point $y_0 \in \eta(0) \cap H$ such that for every $s \in [0,t+1]$, $y_s \in H$). Therefore, we can assume that $(r_n)_n$ converges to some $r \in \R_+$ and $(a_n)_n$ converges to some $a \in \R^2$. By continuity of the Lévy processes at time $t$ we have the two following properties: (a) $y_t$ and $z_t$ belong to the boundary of the disc centered at $x$ of radius $r$, and (b) this disc does not contain any point of $\eta(t)$ in its interior. We can now conclude by using once again the above deterministic observation.
\end{proof}
We are now ready to prove Lemma~\ref{l.jump}. Let $y_0 \in \eta(0)$ and, for every $\varepsilon > 0$, let $0 \leq s_1=s_1(\varepsilon) < s_2 = s_2(\varepsilon) < \cdots$ be the times at which the Lévy process attached to $y_0$ is discontinuous with a jump of size at least $\varepsilon$. Thanks to the claim and by $\sigma$-additivity, it is sufficient to prove that, for every $k \in \N_+$, a.s. there is no infinite black component in $\overline{G}(s_k)$. Fix such $k \in \N_+$ and consider the process $(\omega'(t))_{t \geq 0}$ which is defined using $(\omega(t))_{t \geq 0}$ as follows: (a) $\omega'(0)=\omega(0)$, (b) every point of $\eta(0) \setminus \{ y_0 \}$ evolves in $(\omega'(t))_{t \geq 0}$ exactly like in $(\omega(t))_{t \geq 0}$ and (c) the point $y_0$ evolves according to a Lévy process of distribution $\mu$ independent of everything else. Note that $\omega'(s_k) \sim \Pro_{1/2}$. In particular, a.s. there is no unbounded black component in $\omega'(s_k)$. As a result, a.s. there is no infinite component made of black vertices in $G'(s_k)$ (where $G'(s_k)$ is the obvious analogue of $G(s_k)$). Moreover, $\overline{G}(s_k)$ and $G'(s_k)$ only differ by finitely many edges, which cannot affect the existence of an infinite component. Hence, a.s. there is no infinite black component in $\overline{G}(s_k)$. This ends the proof.
\end{proof}

\section{Absence of noise sensitivity for $t_n$ sufficiently small}\label{a.t_n}

In this section, we prove Theorem~\ref{t.quant_NS_frozen} in the case $t_n n^2 \alpha_4^{an}(n) \underset{n\rightarrow+\infty}{\longrightarrow} 0$.  First note that it is sufficient to prove that
\[
\E \left[ h_n(\omega^{froz}(0)) h_n(\omega^{froz}(t_n)) \right] \underset{n \rightarrow +\infty}{\longrightarrow} 1 \, ,
\]
where $h_n$ is the $\pm 1$ indicator function of $\cross(n,n)$.

Consider the annealed pivotal event $\Piv_D(g_n)$ (where $D$ is a bounded Borel subset of the plane) and the quenched pivotal event $\Piv_x^q(g_n)$ from Definition~\ref{d.piv}. Note that $\Piv_D(g_n)$ is independent of $\eta \cap D$. As a result,
\begin{eqnarray}
\E_{1/2} \left[ \sum_{x \in \eta} \Prob^\eta_{1/2} \left[ \Piv_x^q(g_n) \right] \right] & \leq & \sum_{B \text{ box of } \Z^2} \E_{1/2} \left[ |\eta \cap B| \un_{\Piv_B(g_n)} \right] \nonumber\\
& = & \sum_{B \text{ box of } \Z^2} \Pro_{1/2} \left[ \Piv_B(g_n)\right] \, .\label{e.aazz}
\end{eqnarray}
In~\cite{V1} (Proposition~4.1), we have proved that
\begin{equation}\label{e.dkjghJIDGLMDjgdgf}
\sum_{B \text{ box of } \Z^2} \Pro_{1/2} \left[ \Piv_B(g_n)\right] \asymp n^2 \alpha_4^{an}(n) \, .
\end{equation}
Now, note that:
\begin{eqnarray*}
\E \left[ h_n(\omega^{froz}(0)) h_n(\omega^{froz}(t_n)) \right]
& = & \E \left[ \E \left[ h_n(\omega^{froz}(0)) h_n(\omega^{froz}(t_n)) \cond \eta \right] \right]\\
& = & \E \left[  \sum_{S \subseteq_f \eta} \widehat{h_n^{\eta}}(S)^2 e^{-t_n|S|} \right]\\
& \geq & \E \left[ \exp\left( -t_n \sum_{x \in \eta} \Prob_{1/2}^\eta \left[ \Piv_x^q(g_n) \right] \right) \right] \, , 
\end{eqnarray*}
where the last inequality is (6.10) of~\cite{garban2014noise}. The result now follows from the fact that, by \eqref{e.aazz} and \eqref{e.dkjghJIDGLMDjgdgf}, $\sum_{x \in \eta} \Prob_{1/2}^\eta \left[ \Piv_x^q(g_n) \right]\ll t_n^{-1}$ with high probability.

\section{The probability of a $4$-arm event conditioned on the configuration in a half-plane}\label{s.4hp}

In this section, we prove the following estimate on the $4$-arm event conditioned on the configuration in a half-plane.
\begin{lem}\label{l.half1}
Let $H$ be the lower half-plane $\{ (x_1,x_2) \in \R^2 \, : \, x_2 \leq 0 \}$. There exists $\varepsilon>0$ such that, for every $1 \leq r \leq R < +\infty$,
\[
\E \left[ \Prob_{1/2}^\eta \left[ \arm_4(r,R) \cond \omega \cap H \right]^2 \right] \leq \frac{1}{\varepsilon} \left( \frac{r}{R} \right)^{\varepsilon} \alpha_4^{an}(r,R) \, .
\]
\end{lem}
\begin{proof}
Let us prove the following stronger result, where $\widehat{\arm}_4(r,R) \supseteq \arm_4(r,R)$ is the event from Definition~\ref{d.hat}: there exists $\varepsilon >0$ such that, for every $1 \leq r \leq R < +\infty$,
\begin{equation}\label{e.strong_half}
\E \left[ \Prob_{1/2}^\eta \left[ \widehat{\arm}_4(r,R) \cond \omega \cap H \right]^2 \right] \leq \frac{1}{\varepsilon} \left( \frac{r}{R} \right)^{\varepsilon} \alpha_4^{an}(r,R) \, .
\end{equation}
To this purpose, we follow the proof of Lemma~C.1 of~\cite{GV} which is the analogous result for Bernoulli percolation on $\Z^2$. Let $\omega',\omega'' \sim \Pro_{1/2}$ be two configurations that have the same underlying point process $\eta$, that coincide on $\eta \cap H$ and that are conditionally independent on $\eta \setminus H$. As observed in the beginning of Subsection~5.3 of~\cite{garban2010fourier}, we have
\begin{equation}\label{e.observ_gps}
\Ex_{1/2}^\eta \left[ \Prob_{1/2}^\eta \left[ \widehat{\arm}_4(r,R) \cond \omega \cap H \right]^2 \right]= \Pro \left[ \omega',\omega'' \in \widehat{\arm}_4(r,R) \cond \eta  \right] \, .
\end{equation}
The general idea of the proof is that, even if we condition on the event $\{ \omega'' \in \arm_4(r,R)\}$, then with high probability there are $\Omega(1) \log(R/r)$ scales at which the following holds in $\omega'$: there are so many black connections in $H^c$ that, if $\omega' \in \arm_4(r,R)$, then there must be a $3$-arm event in $H$. Since (by the results of Subsection~\ref{ss.arm}), the probability of the $3$-arm event in the half-plane is much less than the probability of the $4$-arm event in the full plane, this together with quasi-multiplicativity estimates would imply the desired result. The technical difficulty is that (for some reasons that will become clear below), it seems hard to make this reasoning work at the annealed level so we have to work at the quenched level.

Fix some $M \in [100,+\infty[$ to be chosen later and let $\rho \in [100M^2,+\infty[$. The organization of the proof is as follows. Paragraphs A, B and C are devoted to preliminary results. In Paragraph D, we follow the above ideas in order to prove a quenched analogue of \eqref{e.strong_half} in the special case where $r=\rho^i$ and $R=\rho^{i+1}$ for some $i \in \N^*$. To this purpose, we study $4$-arm and $3$-arm events at the following $\log_M(\rho)$ scales: $\rho^i,\rho^iM,\rho^iM^2,\cdots,\rho^{i+1}$ (i.e., in this paragraph, the constant $M$ is the constant used to decide how the scales that we study evolve). In Paragraph E, we conclude the proof. In this paragraph, we will study the models at the scales $\rho,\rho^2,\cdots$ i.e. the constant $\rho$ will be the constant used to decide how the scales that we study evolve.

To simplify the notations in Paragraph D, for any $i \in \N^*$ and $k \in \N$ we let
\[
\rho^i_k := \rho^i M^k
\]
which has to be thought as the ``$k^{th}$ intermediate scale at the global scale $\rho^i$''.

\paragraph{A. The box-crossing and spatial independence properties.} We will use the following events, where the ``$\dense$'' events and the ``$\qbc$'' events are defined in Definitions~\ref{d.dense} and~\ref{d.qbc}:
\[
\widetilde{\gp}^i(\rho,M) := \bigcap_{k=0}^{+\infty} \dense_{1/(100M)} \left( A(\rho^i_k,\rho^i_{k+1}) \right) \cap \qbc_{1/(100M)}^2 \left( A(\rho^i_{k},\rho^i_{k+1}) \right)
\]
(where $\gp$ means ``Good Point process''). We recall that the events $\qbc_{1/(100M)}^2 \left( A(\rho^i_{k},\rho^i_{k+1}) \right)$ provide box-crossing properties for all the rectangles that are included in $A(\rho^i_{k},\rho^i_{k+1})$ and are drawn on the grid $a\Z^2$ for some $a$ of the order of $\rho_{k+1}^i/(100M) = \rho_k^i/100$. However, the box-crossing constant depends on $M$ (see Definition~\ref{d.qbc}). The events $\dense_{1/(100M)} \left( A(\rho^i_{k},\rho^i_{k+1}) \right)$ provide spatial independence properties. Below, we will use obvious independence properties without mentioning them explicitly when we work under the probability measure $\Prob_{1/2}^\eta$ for some $\eta \in \widetilde{\gp}^i(\rho,M)$.

It is easy to see that for each $k \in \N$ and each $i \in \N^*$, we have
\[
\Pro \left[ \dense_{1/(100M)}(A(\rho^i_{k},\rho^i_{k+1})) \right] \geq 1 - \exp(-\Omega(1)(\rho_k^i)^2) \, . 
\]
Moreover, Definition~\ref{d.qbc} implies that for each $k \in \N$ and each $i \in \N^*$, we have
\[
\Pro \left[ \qbc^2_{1/(100M)}(A(\rho^i_k,\rho^i_{k+1})) \right] \geq 1 - \grandO{1}\frac{1}{(\rho_{k+1}^i)^2} \, . 
\]
As a result, for every $i \in \N^*$ we have
\begin{equation}\label{e.le1}
\Pro \left[ \widetilde{\gp}^i(\rho,M) \right] \geq 1 - \grandO{1}\frac{1}{(\rho^i)^2} \geq 1 - \grandO{1}\frac{1}{\rho^2} \, .
\end{equation}

\paragraph{B. The quenched quasi-multiplicativity properties.} We will use the following result which is a direct consequence of Proposition~4.16 of~\cite{V2}.
\begin{prop}\label{p.quen_QM}
There exists $C \in [1,+\infty[$ such that for every $\rho' \geq 1$ the following holds with probability larger than $1-C(\rho')^{-2}$: for every $r_1,r_2,r_3 \in [\rho',+\infty[$ that satisfy $r_1 \leq r_2 \leq r_3$, we have:
\[
\frac{1}{C} \, \Prob_{1/2}^\eta \left[ \arm_4(r_1,r_3) \right] \leq \Prob_{1/2}^\eta \left[ \arm_4(r_1,r_2) \right]\,  \Prob_{1/2}^\eta \left[ \arm_4(r_2,r_3) \right] \leq C  \, \Prob_{1/2}^\eta \left[ \arm_4(r_1,r_3) \right] \, .
\]
\end{prop}
Let $i \in \N^*$ and assume that $\eta$ belongs to the following event:
\[
\overline{\gp}^i(\rho) := \bigcap_{k=0}^{+\infty} \dense_{1/100} \left( A(2^{k-2}\rho^i,2^{k+2}\rho^i) \right) \cap \qbc_{1/100}^2 \left( A(2^{k-2}\rho^i,2^{k+2}\rho^i) \right) \, .
\]
Then, $\eta$ satisfies sufficiently many box-crossing properties so that for any $r_1 \geq \rho^i$ we have
\[
\Prob^\eta_{1/2} \left[ \arm_4(r_1,2r_1) \right] \geq \Omega(1) \, .
\]
Moreover, $\eta$ satisfies sufficiently many spatial independence properties so that
\[
\widehat{\arm}_4(r_1,r_2) \subseteq \arm_4(2r_1,r_2/2) 
\]
for any $r_1,r_2 \in [\rho^i,+\infty[$ such that $r_1 \leq r_2/4$. As a result, if the event of Proposition~\ref{p.quen_QM} holds for $\rho'=\rho^i$ and if $\eta \in \overline{\gp}^i(\rho)$ then the quasi-multiplicativity property is also true for the quantities $\Prob_{1/2}^\eta \left[ \widehat{\arm}_4(r,R) \right]$ i.e. there exists a constant $C' \in [1,+\infty[$ such that, for every $r_1,r_2,r_3 \in [\rho^i,+\infty[$ that satisfy $r_1 \leq r_2 \leq r_3$, we have
\begin{equation}\label{e.quenched_QM_hat}
\frac{1}{C'} \, \Prob_{1/2}^\eta \left[ \widehat{\arm}_4(r_1,r_3) \right] \leq \Prob_{1/2}^\eta \left[ \widehat{\arm}_4(r_1,r_2) \right]\,  \Prob_{1/2}^\eta \left[ \widehat{\arm}_4(r_2,r_3) \right] \leq C'  \, \Prob_{1/2}^\eta \left[ \widehat{\arm}_4(r_1,r_3) \right] \, .
\end{equation}
Let us write $\text{\textup{QQM}}^i(\rho)$ (for Quenched Quasi-Multiplicativity) for the event that the above holds. By the same calculations as in paragraph~A above, $\overline{\gp}^i(\rho)$ holds with probability larger than $1- \grandO{1} \frac{1}{\rho^2}$. Together with Proposition~\ref{p.quen_QM}, this implies that
\begin{equation}\label{e.le1bisbis}
\Pro \left[ \text{\textup{QQM}}^i(\rho) \right] \geq  1- \grandO{1} \frac{1}{(\rho^i)^2} \geq 1 -\grandO{1} \frac{1}{\rho^2} \, .
\end{equation}
\paragraph{C. The quenched probabilities of the $3$-arm event in the half-plane and of the $4$-arm event.} In this paragraph, we consider the annuli $A_k^i=A(\rho_k^i+3\rho_k^i/10,\rho_{k+1}^i-3\rho_k^i/10)$ and the half-annuli $A_k^{i,-} = \{ (x_1,x_2) \in A_k^i \, : \, x_2 \leq 2\rho_k^i/10 \}$. We write $\arm_3(A_k^{i,-})$ for the $3$-arm event in $A_k^{i,-}$ and we let
\[
\widehat{\arm}_3(A_k^{i,-}) = \left\lbrace \Pro_{1/2} \left[ \widehat{\arm}_3(A_k^{i,-}) \cond \omega \cap A_k^i \right] > 0 \right\rbrace \, .
\]
We first prove the following claim.
\begin{claim}\label{cl.3and4}
There exists an absolute constant $c>0$ such that the following holds. For every $\delta \in ]0,1[$ there exists a constant $M_0=M_0(\delta)<+\infty$ such that, if $M \geq M_0$ then for every $k,i \in \N$ we have
\[
\Pro \left[ \Prob_{1/2}^\eta \left[ \widehat{\arm}_3(A_k^{i,-}) \right] \leq \delta \Prob_{1/2}^\eta \left[ \widehat{\arm}_4(\rho_k^i,\rho_{k+1}^i) \right] \right] \geq c \, .
\]
\end{claim}
\begin{proof}
Fix some $\delta \in ]0,1[$. If $M$ is sufficiently large then, by Propositions~\ref{p.4},~\ref{p.univ} and~\ref{p.hat},
\[
\Pro_{1/2} \left[ \widehat{\arm}_3(A_k^{i,-}) \right] \leq \frac{\delta}{2} \Pro_{1/2} \left[ \widehat{\arm}_4(\rho_k^i,\rho^i_{k+1}) \right] \, .
\]
Moreover, by Proposition~\ref{p.hat}, there exists a constant $C'' \in [1,+\infty[$ such that
\[
\sqrt{\E \left[ \Prob_{1/2}^\eta \left[  \widehat{\arm}_4(\rho_k^i,\rho_{k+1}^i) \right]^2 \right]} \leq C''  \Pro_{1/2} \left[ \widehat{\arm}_4(\rho_k^i,\rho_{k+1}^i) \right] \, .
\]
As a result (by applying the Cauchy-Schwarz inequality at the third line),
\begin{align*}
&  \Pro_{1/2} \left[ \widehat{\arm}_4(\rho_k^i,\rho^i_{k+1}) \right]\\
& \leq \frac{1}{\delta} \, \Pro_{1/2} \left[ \widehat{\arm}_3(A_k^{i,-}) \right] +  \E \left[ \Prob_{1/2}^\eta \left[ \widehat{\arm}_4(\rho_k^i,\rho_{k+1}^i) \right] \un_{\Prob_{1/2}^\eta \left[ \widehat{\arm}_3(A_k^{i,-}) \right] \leq \delta \Prob_{1/2}^\eta \left[ \widehat{\arm}_4(\rho_k^i,\rho_{k+1}^i) \right]} \right]\\
& \leq \frac{\delta}{2\delta}  \, \Pro_{1/2} \left[ \widehat{\arm}_4(\rho_k^i,\rho^i_{k+1}) \right]\\
& \hspace{1em}+  \sqrt{\E \left[ \Prob_{1/2}^\eta \left[ \widehat{\arm}_4(\rho^i_k,\rho^i_{k+1}) \right]^2 \right] } \sqrt{\Pro \left[ \Prob_{1/2}^\eta \left[ \widehat{\arm}_3(A_k^{i,-}) \right] \leq \delta \Prob_{1/2}^\eta \left[ \widehat{\arm}_4(\rho_k^i,\rho_{k+1}^i) \right]  \right]}\\
& \leq \frac{1}{2}  \, \Pro_{1/2} \left[ \widehat{\arm}_4(\rho_k^i,\rho_{k+1}^i) \right]\\
& \hspace{1em} + C'' \, \Pro_{1/2} \left[ \widehat{\arm}_4(\rho_k^i,\rho_{k+1}^i) \right] \sqrt{\Pro \left[ \Prob_{1/2}^\eta \left[ \widehat{\arm}_3(A_k^{i,-}) \right] \leq \delta \Prob_{1/2}^\eta \left[ \widehat{\arm}_4(\rho_k^i,\rho_{k+1}^i) \right]  \right]} \, .
\end{align*}
This implies the result with $c=1/(2C'')^2$.
\end{proof}

\paragraph{D. A quenched analogue of \eqref{e.strong_half}.} Remember that the configurations $\omega'$ and $\omega''$ have the same underlying point process $\eta$. In this paragraph, we restrict ourselves to the cases where there exists $i \in \N^*$ such that $r=\rho^i$ and $R=\rho r = \rho^{i+1}$ and we use the events from Paragraphs~A, B and C to study the quantity
\[
\Pro \left[ \omega',\omega'' \in \widehat{\arm}_4(r,R) \cond \eta  \right] \, .
\]
To simplify the calculations below, we assume that $\log_M(\rho)=\log_M(R/r)$ is an integer (the proof in the general case is the same).
\medskip

Fix some $\delta \in ]0,1[$ to be chosen later and assume that $M\geq M_0$ where $M_0=M_0(\delta)$ is the constant from Claim~\ref{cl.3and4}. We let $N_i= |\mathcal{N}_i|$ where $\mathcal{N}_i$ is the set of all the integers $k \in \{ 0, \cdots,\log_M(\rho)-1 \}$ such that
\[
\Prob_{1/2}^\eta \left[ \widehat{\arm}_3(A_k^{i,-}) \right] \leq \delta \Prob^\eta_{1/2} \left[ \widehat{\arm}_4(\rho^i_k,\rho^i_{k+1}) \right] \, .
\]
Note that $N_i$ stochastically dominates a binomial random variable of parameters $\log_M(\rho)$ and $c$ where $c$ is the constant from Claim~\ref{cl.3and4}. As a result, there exists an absolute constant $a \in ]0,1[$ such that
\begin{equation}\label{e.a_et_N_i}
\Pro \left[ N_i \leq a\log_M(\rho) \right] \leq \frac{1}{a}e^{-a\log_M(\rho)} \, .
\end{equation}
We consider the following event
\[
\gp^i(\rho,M) = \widetilde{\gp}^i(\rho,M) \cap \text{\textup{QQM}}^i(\rho) \cap \left\lbrace  N_i \geq a \log_M(\rho)  \right\rbrace  \, .
\]
We now prove the following claim.
\begin{claim}\label{cl.app_quen}
Assume that there exists $i \in \N^*$ such that $r=\rho^i$ and $R=\rho r$. If $M$ is sufficiently large, then there exists a constant $d=d(M) > 0$ (that does not depend on $i$ and $\rho$) such that for every $\eta \in \gp^i(\rho,M)$ we have
\[
\Pro \left[ \omega',\omega'' \in \widehat{\arm}_4(r,R) \cond \eta  \right] \leq \rho^{-d}  \Prob_{1/2}^\eta \left[ \widehat{\arm}_4(r,R) \right] \, .
\]
\end{claim}

\begin{proof}
Let $i$ and $\rho$ as in the statement of the claim and let $\eta \in \gp^i(\rho,M)$. For each $k \in \{0,\cdots,\log_M(\rho)-1\}$ let $E(k)=E^i(k)$ denote the event that there are black paths as in Figure~\ref{f.pleindanneaux}. Since $\eta \in \widetilde{\gp}^i(\rho,M)$, there exists $c'=c'(M) > 0$ such that, for each $k$, $\Prob_{1/2}^\eta \left[ E(k) \right] \geq c'(M)$.  Let $\widetilde{N}_i \leq N_i$ be the number of integers $k \in \mathcal{N}_i$ such that $\omega' \in E(k)$. Let $k_1 < \cdots < k_{\widetilde{N}_i}$ denote these integers and let $l_1<\cdots<l_m$ be a possible occurrence of these. By spatial independence, we have
\begin{multline}
\Pro \left[ \omega', \, \omega'' \in \widehat{\arm}_4(r,R)) \cond \eta, \, \widetilde{N}_i=m, k_1 = l_1, \cdots, k_m = l_m \right] \leq \Pro \left[ \omega'' \in \widehat{\arm}_4(r,\rho_{l_1}^i) \cond \eta \right]\\
\times \prod_{q = 1}^{m-1} \left( \Pro \left[ \omega' \in \widehat{\arm}_4(\rho_{l_q}^i,\rho_{l_q+1}^i) \cond \eta, \, \omega' \in E(l_q) \right] \Pro \left[ \omega'' \in \widehat{\arm}_4(\rho_{l_q+1}^i,\rho^i_{l_{q+1}}) \cond \eta \right] \right)\\
\times \Pro \left[ \omega' \in \widehat{\arm}_4(\rho_{l_m}^i,\rho_{l_m+1}^i) \cond \eta, \, \omega' \in E(l_m) \right] \Pro \left[ \omega'' \in \widehat{\arm}_4(\rho_{l_m+1}^i,R) \cond \eta \right] \, .\label{e.FromGV}
\end{multline}
Let us study the quantities $\Pro \left[ \omega' \in \widehat{\arm}_4(\rho_{l_q}^i,\rho^i_{l_q+1}) \cond \eta, \, \omega' \in E(l_q) \right]$ for $q \in \{ 1, \cdots, m\}$. To this purpose, note that we have (see Figure~\ref{f.pleindanneauxetchemins}):
\[
\Prob_{1/2}^\eta \left[ \arm_4(\rho_k^i+3\rho^i_k/10,\rho^i_{k+1}-3\rho^i_k/10) \cond E(k) \right] \leq \Prob_{1/2}^\eta \left[ \arm_3(A_k^{i,-}) \right] \, .
\]
Actually, in order to obtain this estimate, we have to condition on the upper paths that cross the rectangles $[\rho_k^i+2\rho_k^i/10,\rho^i_{k+1}-2\rho^i_k/10] \times [\rho^i_k/10,2\rho^i_k/10]$ and $[-(\rho^i_k+2\rho^i_k/10),-(\rho^i_{k+1}-2\rho^i_k/10)] \times [\rho_k^i/10,2\rho^i_k/10]$ and observe that (since we work under the quenched measure) this conditioning does not affect the configuration below these paths. \textbf{We could not apply this argument at the annealed level}.
\medskip

\begin{figure}[h!]
\centering
\includegraphics[scale=0.42]{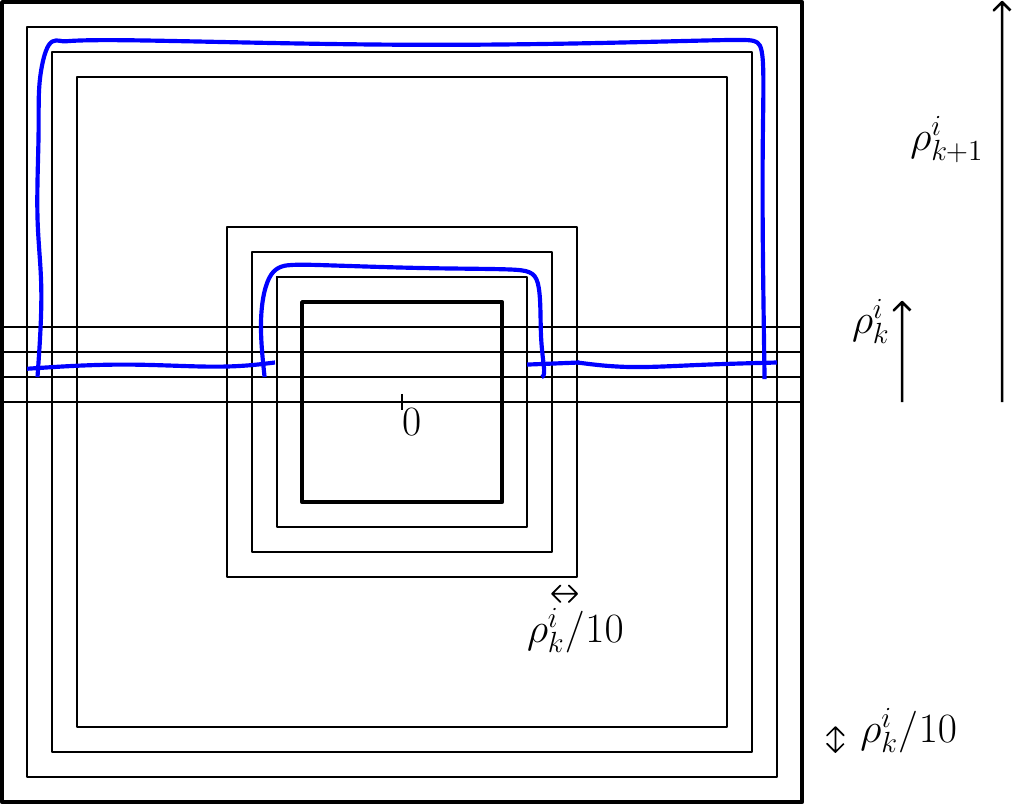}
\caption{A realization of the event $E(k)=E^i(k)$.\label{f.pleindanneaux}}
\end{figure}

\begin{figure}[h!]
\centering
\includegraphics[scale=0.42]{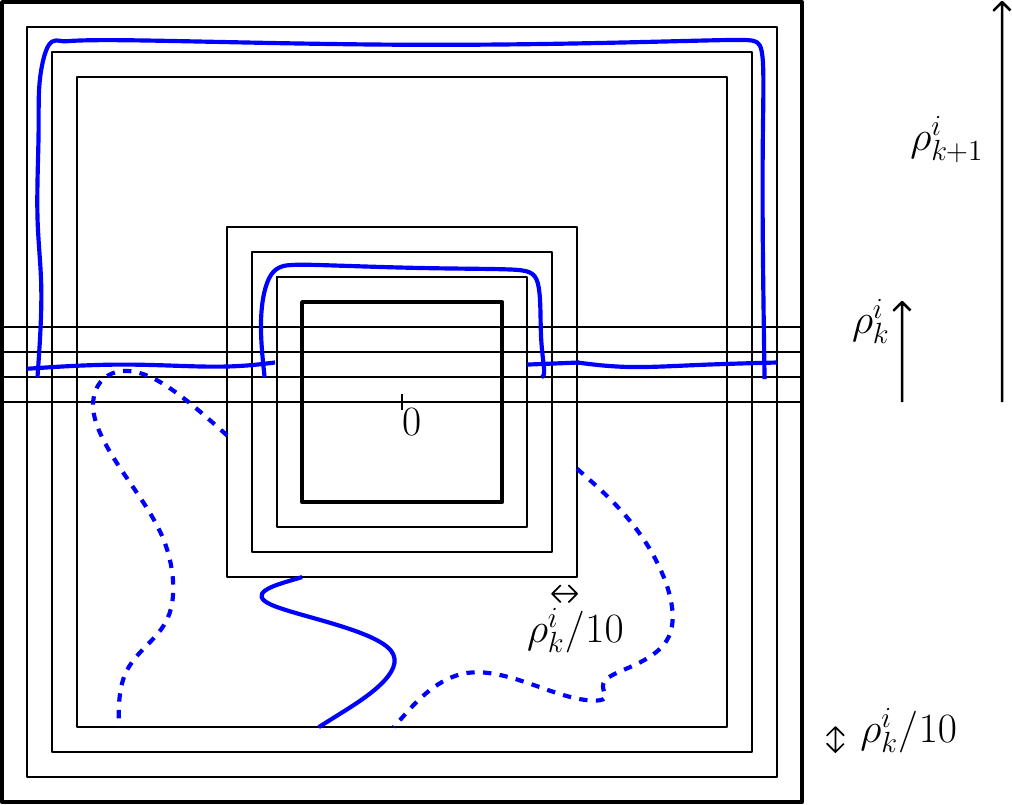}
\caption{A realization of the event $E(k)$ and of the $4$-arm event implies the realization of a $3$-arm event in a half-plane.\label{f.pleindanneauxetchemins}}
\end{figure}

Since $\eta \in \widetilde{\gp}^i(\rho,M)$ we have enough spatial independence properties so that $\widehat{\arm}_4(\rho_k^i,\rho^i_{k+1}) \subseteq \arm_4(\rho^i_k+3\rho^i_k/10,\rho^i_{k+1}-3\rho^i_k/10)$ for each $k \in \N$. As a result for each $k \in \{0,\cdots,\log_M(\rho)-1 \}$ we have
\[
\Prob_{1/2}^\eta \left[ \widehat{\arm}_4(\rho^i_k,\rho^i_{k+1}) \cond E(k) \right] \leq \Prob_{1/2}^\eta \left[ \arm_3(A_k^{i,-}) \right]  \leq  \Prob_{1/2}^\eta \left[ \widehat{\arm}_3(A_k^{i,-}) \right]  \, .
\]
Since $k_1,\cdots, k_{\widetilde{N}_i} \in \mathcal{N}_i$,~\eqref{e.FromGV} and the above imply that
\begin{multline*}
\Pro \left[ \omega', \, \omega'' \in \widehat{\arm}_4(r,R)) \cond \eta, \, \widetilde{N}_i=m, k_1 = l_1, \cdots, k_m = l_m \right] \leq \delta^m \times \Pro \left[ \omega'' \in \widehat{\arm}_4(r,\rho_{l_1}^i) \cond \eta \right]\\
\times \prod_{q = 1}^{m-1} \left( \Pro \left[ \omega' \in \widehat{\arm}_4(\rho_{l_q}^i,\rho_{l_q+1}^i) \cond \eta\right] \Pro \left[ \omega'' \in \widehat{\arm}_4(\rho_{l_q+1}^i,\rho^i_{l_{q+1}}) \cond \eta \right] \right)\\
\times \Pro \left[ \omega' \in \widehat{\arm}_4(\rho^i_{l_m},\rho^i_{l_m+1}) \cond \eta \right] \Pro \left[ \omega'' \in \widehat{\arm}_4(\rho^i_{l_m+1},R) \cond \eta \right] \, .\label{e.FromGV}
\end{multline*}
Next, the quenched quasi-multiplicativity property~\eqref{e.quenched_QM_hat} implies that the above is at most
\[
\delta^m (C')^{2m} \Prob_{1/2}^\eta \left[ \widehat{\arm}_4(r,R) \right] \, .
\]
Since $\eta \in \gp^i(\rho,M)$, we have $N_i \geq a \log_M(\rho)$. Moreover, since (conditionally on $\eta$) $\widetilde{N}_i$ stochastically dominates a binomial random variable of parameters $N_i$ and $c'(M)$, there exists $b=b(M)>0$ such that
\[
\Pro \left[ \widetilde{N}_i \leq b \log_M(\rho) \cond \eta \right] \leq \frac{1}{b}\exp(-b\log_M(\rho)) \, .
\]
Note furthermore that, conditionally on $\eta$, $\widetilde{N}_i$ is independent of $\omega''$. We finally obtain that
\[
\Pro \left[ \omega', \, \omega'' \in \widehat{\arm}_4(r,R)) \cond \eta \right] \leq \left( \frac{1}{b}\exp(-b\log_M(\rho)) + \left( \delta (C')^2 \right)^{b\log_M(\rho)} \right) \Prob_{1/2}^\eta \left[ \widehat{\arm}_4(r,R) \right] \, .
\]
This ends the proof (if we choose for instance $\delta = 1/(2(C')^2)$).
\end{proof}

\paragraph{E. Integretion on $\eta$.} Let us conclude the proof of~\eqref{e.strong_half}. By Claim~\ref{cl.app_quen}, we have the following if $M$ is sufficiently large and if there exists $i \in \N^*$ such that $r=\rho^i$ and $R=\rho r$:
\[
\Pro \left[ \omega', \omega'' \in \widehat{\arm}_4(r,R) \right]
\leq \Pro \left[ \eta \notin \gp^i(\rho,M), \, \omega', \omega'' \in \widehat{\arm}_4(r,R) \right] + \rho^{-d} \Pro_{1/2} \left[\widehat{\arm}_4(r,R) \right] \, .
\]
By Proposition~\ref{p.hat}, the above is less than or equal to
\[
\Pro \left[ \eta \notin \gp^i(\rho,M), \, \omega', \omega'' \in \widehat{\arm}_4(r,R) \right] + \grandO{1} \rho^{-d} \alpha^{an}_4(r,R) \, .
\]
If we combine the above with~\eqref{e.le1} and~\eqref{e.le1bisbis}, we obtain that
\begin{multline*}
\Pro \left[ \omega', \omega'' \in \widehat{\arm}_4(r,R) \right] \leq \grandO{1} \frac{1}{\rho^2}\\
+ \Pro \left[ N_i \leq a \log_M(\rho), \,  \omega', \omega'' \in \widehat{\arm}_4(r,R) \right] +  \grandO{1} \rho^{-d} \alpha^{an}_4(r,R) \, .
\end{multline*}
By the Cauchy-Schwarz inequality, we have
\begin{multline*}
\Pro \left[ N_i \leq a  \log_M(\rho), \,  \omega', \omega'' \in \widehat{\arm}_4(r,R) \right]\\
\hspace{-8em}= \E \left[ \un_{N_i \leq a  \log_M(\rho)} \Pro \left[ \omega', \omega'' \in \widehat{\arm}_4(r,R) \cond \eta \right]  \right]\\
\leq \sqrt{\Pro \left[ N_i \leq a  \log_M(\rho) \right]} \sqrt{\E \left[\Pro \left[ \omega', \omega'' \in \widehat{\arm}_4(r,R) \cond \eta \right]^2 \right]}\\
\leq \sqrt{\Pro \left[ N_i \leq a  \log_M(\rho) \right]} \sqrt{\E \left[\Prob_{1/2}^\eta \left[ \widehat{\arm}_4(r,R) \right]^2 \right]} \, .
\end{multline*}
Proposition~\ref{p.hat} implies that $\sqrt{\E \left[\Prob_{1/2}^\eta \left[ \widehat{\arm}_4(r,R) \right]^2 \right]} \leq \grandO{1} \alpha^{an}_4(r,R)$. Hence, if we combine this with~\eqref{e.a_et_N_i}, we obtain that the above is less than or equal to
\[
\grandO{1}   \sqrt{\frac{1}{a}e^{-a \log_M(\rho)}}\alpha_4^{an}(r,R) \, .
\]
As a result, if there exists $i \in \N^*$ such that $r=\rho^i$ and $R=\rho r$, then
\begin{multline*}
\E \left[ \Prob_{1/2}^\eta \left[ \widehat{\arm}_4(r,R) \cond \omega \cap H \right]^2 \right] = \Pro \left[ \omega', \omega'' \in \widehat{\arm}_4(r,R) \right]\\
\leq \grandO{1} \left( \frac{1}{\rho^2} + \sqrt{\frac{1}{a} e^{-a \log_M(\rho)}}\alpha_4^{an}(r,R) +  \rho^{-d} \alpha_4^{an}(r,R) \right) \, .
\end{multline*}
By Proposition~\ref{p.4}, there exists $d'>0$ such that $\frac{1}{\rho^2} \leq \grandO{1} \rho^{-d'} \alpha_4^{an}(r,R)$. As a result (still if $M$ is sufficiently large and if there exists $i \in \N^*$ such that $r=\rho^i$ and $R=\rho r$), there exists $d''=d''(M)>0$ such that the above is at most
\[
\rho^{-d''} \alpha_4^{an}(r,R) = \left( \frac{r}{R} \right)^{d''} \alpha_4^{an}(r,R) \, .
\]
Fix a constant $M$ sufficiently large so that the above holds. Finally, we have obtained the desired result (i.e.~\eqref{e.strong_half}) for any $r,R$ satisfying $r=\rho^i$ and $R=\rho r$ for some $i \in \N^*$ and some $\rho \in [100M^2,+\infty[$. Fix some $\rho \in [100M^2,+\infty[$ to be chosen later. Let us conclude that~\eqref{e.strong_half} holds in general. To this purpose, first note that, by the quasi-multiplicativity property for the quantities $\alpha_4^{an}(\cdot,\cdot)$, it is enough to prove the result in the cases where there exist $i < j \in \N^*$ such that $r=\rho^i$ and $R=\rho^j$. Therefore, let $i < j \in \N^*$, let $r=\rho^i$ and $R=\rho^j$, and let us study the quantity $\Pro \left[ \omega', \omega'' \in \widehat{\arm}_4(r,R) \right]$. By spatial independence and by using the result in the cases already proved, we have
\[
\Pro \left[ \omega', \omega'' \in \widehat{\arm}_4(r,R) \right] \leq \prod_{l=i}^{j-1} \Pro \left[ \omega', \omega'' \in \widehat{\arm}_4(\rho^l,\rho^{l+1}) \right]
\leq \prod_{l=i}^{j-1} \rho^{-d''} \alpha_4^{an}(\rho^l,\rho^{l+1}) \, ,
\]
By the quasi-multiplicativity property of the quantities $\alpha_4^{an}(\cdot,\cdot)$, there exists an absolute constant $C''' \in [1,+\infty[$ such that the above is less than or equal to
\[
(C''' \rho^{-d''/2})^{j-i} \, \rho^{-d''(j-i)/2}  \, \alpha_4^{an}(r,R) = (C''' \rho^{-d''/2})^{j-i} \, \left( \frac{r}{R} \right)^{d''/2}  \, \alpha_4^{an}(r,R)  \, .
\]
This ends the proof (if we choose $\rho$ such that $C''' \rho^{-d''/2}\leq 1$).
\end{proof}

In Subsection~\ref{ss.proofs_spec}, we need the following consequence of Lemma~\ref{l.half1}. Remember that $f_R$ is the $1$-arm event.
\begin{lem}\label{l.half}
Let $R \in [1,+\infty[$ and $1 \leq \rho_1 < \rho_2 < +\infty$. Also, let $A$ be an annulus included in $[-R,R]^2$ of the form $A(x;\rho_1,\rho_2)=x+[-\rho_2,\rho_2]^2 \setminus ]-\rho_1,\rho_1[^2$ and let $B$ be its inner square. Assume that $A$ is at distance at least $\rho_2$ from $0$ (in particular, neither $A$ nor $B$ contains $0$). Let $H$ be a half-plane whose boundary contains the center of $A$ and is parallel to the $x$ or $y$ axis. Then, there exists an absolute constant $\varepsilon_2>0$ such that the following holds:
\[
\E \left[ \Prob^\eta_{1/2} \left[ \Piv_{B}^A(f_R) \cond \omega \cap H \right]^2 \right] \leq \grandO{1} (\rho_1/\rho_2)^{\varepsilon_2} \alpha_4^{an}(\rho_1,\rho_2) \, ,
\]
where $\Piv_{B}^A(f_R)$ is the pivotal event from Definition~\ref{d.piv_hat}.
\end{lem}
\begin{proof}
We first write the proof in the case where $\rho_2 \leq \rho_1^2$. Let 
\[
\dense(\rho_1,\rho_2) = \dense_{1/100}(A(x;\rho_1,2\rho_1)) \cap \dense_{1/100}(A(x;\rho_2/2,\rho_2)) \, ,
\]
where the events ``$\dense$'' are the events from Definition~\ref{d.dense}. Note that, if $\omega \in \dense(\rho_1,\rho_2) \cap \Piv_{B}^A(f_R)$, then the $4$-arm event in $A(x;2\rho_1,\rho_2/2)$ (denoted by $\arm_4(x;2\rho_1,\rho_2/2)$) holds. As a result,
\[
\E \left[ \Prob^\eta_{1/2} \left[ \Piv_{B}^A(f_R) \cond \omega \cap H \right]^2 \right]
 \leq \E \left[ \Prob^\eta_{1/2} \left[ \arm_4(x;2\rho_1,\rho_2/2) \cond \omega \cap H \right]^2 \right] + \Pro \left[ \neg \dense(\rho_1,\rho_2) \right] \, .
\]
By using that $\Pro \left[ \neg \dense(\rho_1,\rho_2) \right]$ decays to $0$ super-polynomially fast in $\rho_1$, by using that $\rho_2 \leq \rho_1^2$, and by using Lemma~\ref{l.half1}, we obtain that there exists $\varepsilon>0$ such that
\[
\E \left[ \Prob^\eta_{1/2} \left[ \Piv_{B}^A(f_R) \cond \omega \cap H \right]^2 \right]  \leq \grandO{1} (\rho_1/\rho_2)^{\varepsilon} \alpha_4^{an}(2\rho_1,\rho_2/2) \leq \grandO{1} (\rho_1/\rho_2)^{\varepsilon} \alpha_4^{an}(\rho_1,\rho_2) \, .
\]
(The last inequality comes from the quasi-multiplicativity property (and~\eqref{e.poly}).)
\medskip

Let us now prove the result in the general case. To this purpose, let $M$ be sufficiently large to be chosen later. Also, let $A_k=A(x;\rho_1 M^k,\rho_1 M^{k+1})$  and let $B_k$ be the inner square of $A_k$. The events $\Piv_{B_k}^{A_k}(f_R)$, $k = 0, \cdots, \log_M(\rho_2/\rho_1)-1$, are independent and for every $k \in \{0,\cdots, \log_M(\rho_2/\rho_1)-1 \}$, $\Piv_{B_k}^{A_k}(f_R)$ contains $\Piv_{B}^{A}(f_R)$. Therefore, the result in the case $\rho_2  \leq \rho_1^2$ and the quasi-multiplicativity property for the quantities $\alpha_4^{an}(\cdot,\cdot)$ imply that there exists $C \in [1,+\infty[$ such that, in the general case,
\[
\E \left[ \Prob^\eta_{1/2} \left[ \Piv_{B}^A(f_R) \cond \omega \cap H \right]^2 \right] \leq C^{\log_M(\rho_2/\rho_1)} (\rho_1/\rho_2)^{\varepsilon} \alpha_4^{an}(\rho_1,\rho_2) \, .
\]
Choosing $M$ such that $\log(C)/\log(M) \leq \varepsilon/2$ ends the proof. 
\end{proof}

\bibliographystyle{alpha}
\bibliography{ref_perco}

\end{document}